\documentclass[12pt]{amsart}

\newtheorem{theorem}{Theorem}[section]
\newtheorem{lemma}[theorem]{Lemma}
\newtheorem{proposition}[theorem]{Proposition}
\newtheorem{corollary}[theorem]{Corollary}

 \theoremstyle{definition}
\newtheorem{definition}[theorem]{Definition}

\theoremstyle{remark}
\newtheorem{remark}[theorem]{Remark}

\numberwithin{equation}{section}

\begin{document}

\title[Fully non-linear elliptic equations]
{Regularity of fully non-linear elliptic equations on Hermitian manifolds. II}

 \author{Rirong Yuan}
\email{rirongyuan@stu.xmu.edu.cn}

\thanks{Research supported in part by the National Natural Science Foundation of China (Grant No. 11801587)
}
 
\date{}

\begin{abstract}

In this paper we investigate the regularity and solvability of solutions to Dirichlet problem for 
 fully non-linear elliptic equations with gradient terms on Hermitian manifolds, which include among others 
 the Monge-Amp\`ere equation for $(n-1)$-plurisubharmonic functions. 
Some significantly new features of regularity assumptions on the boundary and boundary 
data are obtained, which reveal how the shape of the boundary influences such regularity assumptions.
Such new features follow from quantitative boundary estimates which specifically enable us to 
apply a blow-up argument to derive the gradient estimate.  
  Interestingly, the subsolutions are constructed when the background space is moreover a product of
 a closed Hermitian manifold with a compact Riemann surface with boundary.



 \end{abstract}

 \maketitle

\tableofcontents

 \section{Introduction}



Let $(M,J,\omega)$ be a compact Hermitian manifold of complex
dimension $n\geq 2$ possibly with  boundary, $\partial M$, 
and $\omega= \sqrt{-1}  g_{i\bar j} dz^i\wedge d\bar z^j$ denote the K\"ahler form being compatible with complex structure $J$.

\vspace{1mm} 
This paper is primarily devoted to investigating second order fully nonlinear elliptic equations 
for deformation of real $(1,1)$-forms depending linearly on gradient terms,
\begin{equation}
\label{mainequ-ios}
\begin{aligned}
F(\mathfrak{g}[u]) := f(\lambda(\mathfrak{g}[u]))= \psi \mbox{ in  } M,
\end{aligned}
\end{equation} 
where $\lambda(\mathfrak{g}[u])$ 
are eigenvalues of $\mathfrak{g}[u]$  with respect to $\omega$.
In addition if $M$ has  
boundary, $\bar M:=M\cup \partial M$,  we study equation \eqref{mainequ-ios} with prescribing  boundary data
 \begin{equation}
\label{mainequ1}
\begin{aligned}
u=  \varphi,   \mbox{ on }\partial M.
\end{aligned}
\end{equation}
Here $\psi$ and $\varphi$ are sufficiently smooth functions. 

 \vspace{1mm} 
 The study of equations generated by symmetric functions of eigenvalues  
 goes back to the work of Caffarelli-Nirenberg-Spruck \cite{CNS3} concerning the Dirichlet problem in bounded domains of $\mathbb{R}^n$, 
 and to the work of Ivochkina \cite{Ivochkina1981} which considers some special cases.
  As in \cite{CNS3},  $f$ is a smooth symmetric function defined in an open symmetric and convex cone $\Gamma$ 
with vertex at origin, $\Gamma_n\subseteq\Gamma\subset \Gamma_1$ and  boundary 
$\partial \Gamma\neq \emptyset$, where 
 $\Gamma_k=\{\lambda\in\mathbb{R}^n: \sigma_j(\lambda)>0, \forall 1\leq j\leq k\}$, 
  and $\sigma_j$ is the $j$-th elementary symmetric function. 
 Moreover,  $f$ satisfies the following 
   fundamental conditions:
\begin{equation}
\label{elliptic}
 f_{i}:=f_{\lambda_i}(\lambda)=\frac{\partial f}{\partial \lambda_{i}}(\lambda)> 0  \mbox{ in } \Gamma,\  1\leq i\leq n,
\end{equation}
\begin{equation}
\label{concave}
 f \mbox{ is  concave in } \Gamma,
\end{equation}
\begin{equation}
 \label{nondegenerate}
\delta_{\psi,f}:= \inf_{M} \psi  -\sup_{\partial \Gamma} f >0,
 \end{equation}
where 
 $\sup_{\partial \Gamma}f :=\sup_{\lambda_{0}\in \partial \Gamma } \limsup_{\lambda\rightarrow \lambda_{0}}f(\lambda).$

   \vspace{1mm} 
      In order to study equation  \eqref{mainequ-ios} 
    within the framework of elliptic equations,  we shall look for  solutions  
in the class of $C^2$-\textit{admissible} functions $u$  
satisfying  $\lambda(\mathfrak{g}[u])\in \Gamma$.   
The constant $\delta_{\psi,f}$  measures whether  or not the equation is degenerate.
More explicitly, if $\delta_{\psi,f}>0$ (respectively, $\delta_{\psi,f}$ vanishes) then the equation
 is called non-degenerate (respectively, degenerate).
Moreover, $\sup_M\psi<\sup_{\Gamma} f$ is necessary for the solvability of equation \eqref{mainequ-ios} within the framework of elliptic equations, which is automatically satisfied when $\sup_{\Gamma}f=+\infty$ or there is  
a subsolution satisfying \eqref{existenceofsubsolution} or \eqref{existenceofsubsolution2}.

  \vspace{1mm}
 A notion of subsolution is used to study Dirichlet problem \eqref{mainequ-ios}-\eqref{mainequ1}.
 We call  $\underline{u}$ an \textit{admissible} subsolution if  
  it is an   \textit{admissible} function $\underline{u}\in C^2(\bar M)$ satisfying
 \begin{equation}
\label{existenceofsubsolution}
\begin{aligned}
f(\lambda(\mathfrak{g}[\underline{u}])) \geq  \psi      \mbox{ in } M,  \mbox{  and  }
\underline{u}=  \varphi  \mbox{ on } \partial  M.
\end{aligned}
\end{equation}
   Such a subsolution is a key ingredient in
 deriving \textit{a priori} estimates, especially for boundary estimate for Dirichlet problem (cf. \cite{Guan1993Boundary,Hoffman1992Boundary,Guan1998The}), moreover, it 
  relaxes geometric restrictions to boundary and so
  plays important roles in some  geometric problems (cf. \cite{Chen,GuanP2002The,Guan2009Zhang}).
Recently, influenced by the work of Guan \cite{Guan12a}, an extended notion of  a subsolution was proposed
by Sz\'ekelyhidi  \cite{Gabor}.
 Following  Sz\'ekelyhidi, a $C^2$ function $\underline{u}$ is a $\mathcal{C}$-subsolution of equation \eqref{mainequ-ios},
 if  for each $z\in \bar M$ the set $(\lambda(\mathfrak{g}[\underline{u}](z))+\Gamma_n)\cap \partial\Gamma^{\psi(z)}$ is   bounded, 
 where $\partial\Gamma^\sigma =\{\lambda\in \Gamma:$ $f(\lambda)=\sigma\}$ denotes the level hypersurface. Namely,
\begin{equation}
\label{existenceofsubsolution2}
\begin{aligned}
\lim_{t\rightarrow +\infty}f(\lambda(\mathfrak{g}[\underline{u}])+te_i)>\psi, \mbox{ in } \bar M \mbox{ for each } i=1,\cdots, n,
\end{aligned}
\end{equation}
where $e_i$ is the $i$-$\mathrm{th}$   standard basis vector. 
       The notion of a $\mathcal{C}$-subsolution  turns out to be applicable for the setting of closed manifolds (cf. \cite{Gabor,GTW15,CollinsJacobYau}).
        For the $J$-flow introduced by Donaldson \cite{Donaldson99J} and  Chen \cite{Chen2004J},
   it is exactly same as the condition $(1.6)$  of Song-Weinkove \cite{Song2008On}
   (see Fang-Lai-Ma \cite{Fang11LaiMa} for  extension to complex inverse $\sigma_k$ flow).

 \vspace{1mm}
The most important equation of this type  is perhaps the complex Monge-Amp\`{e}re  equation
  that is closely related to   Calabi's conjecture.
 In his celebrated work \cite{Yau78}, Yau proved Calabi's conjecture and obtained
 Calabi-Yau theorem on prescribed volume form  
which implies in particular that each closed K\"ahler manifold of vanishing first Chern class endows with a Ricci flat K\"ahler metric. 
Yau also showed  that the existence of K\"ahler-Einstein metric on closed K\"{a}hler manifolds of  $c_1(M)<0$,
which was also proved  by Aubin \cite{Aubin1976Equations} independently.
(Yau's work was partially extended by Tosatti-Weinkove \cite{Tosatti2010Weinkove} to closed Hermitian manifolds).
 Another fundamental work  concerning
  complex Monge-Amp\`{e}re equation was done by Bedford-Taylor \cite{BT76, BT82} on weak solutions and pluripontential theory, and by
   Caffarelli-Kohn-Nirenberg-Spruck \cite{Caffarelli1985The} who treated with
  Dirichlet problem for  complex Monge-Amp\`{e}re equation on strictly pseudoconvex domains in $\mathbb{C}^n$. 
  See \cite{Kolodziej1998Acta} for deep results and extension.
Recently, the Monge-Amp\`ere equation for $(n-1)$-plurisubharmonic (PSH) functions attracts a lot of attention and interests due to its relation to Gauduchon's conjecture from non-K\"ahler geometry.

\vspace{1mm}
 Our work is motivated by 
increasing interests and problems from non-K\"ahler geometry: 
\begin{itemize}
\item The deformation for $\tilde{\chi}+\partial\overline{\gamma}+\overline{\partial}\varsigma$,  corresponding to \eqref{mainequ} 
if $\eta^{1,0}$ is holomorphic, in Aeppli cohomology group $H_A^{1,1}(M)$. 
\item The deformation of $(n-1,n-1)$-forms via the map
 $\bigwedge^{n-1,n-1}(M)\overset{*}\longrightarrow \bigwedge^{1,1}(M)$,  
 induced by Hodge Star-operator $*$ with respect to $\omega$. 
It is closely related to the Calabi-Yau theorem for Gauduchon and balanced metrics (see for example \eqref{MA-n-1} and \eqref{FWW-equ} respectively). 
 
 \end{itemize}
 
The latter topic is studied in last section; while the rest of this paper is  
devoted to investigating the former one in which 
\begin{equation}
\label{mainequ}
\begin{aligned}
  f(\lambda(\mathfrak{g}[u]))= \psi,  \mbox{  }
  \mathfrak{g}[u] =\tilde{\chi}+\sqrt{-1}\partial \overline{\partial} u+\sqrt{-1} (\partial u\wedge  \overline{\eta^{1,0}}+ \eta^{1,0} \wedge  \overline{\partial} u), \mbox{ in } M,
\end{aligned}
\end{equation}
where  $\tilde{\chi}$ is a smooth real $(1,1)$-form and
   $\eta^{1,0} =\eta_{i}d z^i$  is a smooth $(1,0)$-form  on the background Hermitian manifold.
   Here, we write $\eta_{\bar i}=\bar\eta_i$.

 \vspace{1mm}  
    Before we state our main results, we first present some notation. 
    Let ${L}_{\partial M}$  denote the Levi form of $\partial M$, 
    and $\Delta$  the complex Laplacian operator with respect to $\omega$.
In addition,  we denote as in  \cite{Trudinger95}
\begin{equation}
\begin{aligned}
 \Gamma_{\infty}=\{(\lambda_1,\cdots,\lambda_{n-1})\in \mathbb{R}^{n-1}: 
 (\lambda_1,\cdots,\lambda_{n-1},R)\in \Gamma \mbox{ for some } R\}, \nonumber
\end{aligned}
\end{equation}
 the projection of $\Gamma$ onto $\mathbb{R}^{n-1}$. Also, we denote 
\begin{equation}\begin{aligned}\Gamma_{\mathbb{R}^1}^\infty=\{c\in\mathbb{R}: (t,\cdots, t, c)\in\Gamma, \mbox{ for some } t>0\}. \nonumber\end{aligned}\end{equation}
Moreover, $\overline{\Gamma}_{\infty}$ and $\overline{\Gamma}_{\mathbb{R}^1}^{\infty}$ are respective the closure of $\Gamma_{\infty}$ and $\Gamma_{\mathbb{R}^1}^{\infty}$.

 \vspace{1mm}
There is an outstanding problem in deriving gradient bound for 
equation \eqref{mainequ-ios},
which is  
widely open in the general case. 
 In this paper,  we apply  blow-up argument to establish gradient estimate for equations \eqref{mainequ} and \eqref{mainequ-gauduchon-general*} 
 with the assumption 
   \begin{equation}
\label{bdry-assumption1}
\begin{aligned}
\lambda_{\omega'}(- {L}_{\partial M})\in \overline{\Gamma}_{\infty}, \mbox{  } \omega'=\omega|_{T_{\partial M}\cap JT_{\partial M}},
\end{aligned}
\end{equation} 
 \begin{equation}
\label{bdry-assumption1-Gauduchon}
\begin{aligned}
\mathrm{tr}_{\omega'}(- {L}_{\partial M})\in \overline{\Gamma}_{\mathbb{R}^1}^{\infty},
\end{aligned}
\end{equation} 
correspondingly, which includes among others
  \textit{pseudoconcave} and \textit{mean pseudoconcave}\renewcommand{\thefootnote}{\fnsymbol{footnote}}\footnote{
We say $\partial M$ is \textit{pseudoconcave} (respectively, is \textit{mean pseudoconcave}) if   
 the Levi form 
  is negative semidefinite $({L}_{\partial M}\leq0)$ (repectively, has nonpositive trace
 $\mathrm{tr}_{\omega'}({L}_{\partial M})\leq 0$), which includes among others  \textit{holomorphically flat} in the sense that there exist holomorphic coordinates $(z_1, \cdots,z_n)$ such that $\partial M$ is locally given by $\mathfrak{Re} (z_n) = 0$.
See also \cite{Cartan-1933} for Cartan's theorem on characterization of real analytic Levi flat hypersurfaces.} boundary, respectively.
In order to apply Sz\'ekelyhidi's \cite{Gabor} Liouville type
 theorem extending a result of Dinew-Ko{\l}odziej \cite{Dinew2017Kolo},
 we further assume
  \begin{equation}
\label{addistruc}
\begin{aligned}
\mbox{For each $\sigma<\sup_{\Gamma}f$ and } \lambda\in \Gamma, \mbox{   } \lim_{t\rightarrow +\infty}f(t\lambda)>\sigma.
\end{aligned}
\end{equation}

  An interesting fact is that  \eqref{bdry-assumption1} and  \eqref{bdry-assumption1-Gauduchon} are automatic 
   if $\Gamma$ is of type 2 in the sense of \cite{CNS3}, since $\Gamma_\infty=\mathbb{R}^{n-1}$,  
 $\Gamma_{\mathbb{R}^1}^{\infty}=\mathbb{R}$.
 As a somewhat surprising consequence of main results presented below, without imposing 
 restrictions to Levi form of boundary, we can solve Dirichlet problem
  for  degenerate 
  equations 
  with $\Gamma$ being of type 2.
  
 \vspace{1mm} 
 Our strategy proposes a new and unified approach to the study of fully nonlinear elliptic equations possibly with degenerate right-hand side.  It would be applied to further geometric problems.
 Our main results can be stated as follows.


 \begin{theorem}
\label{thm1-diri-estimate}
Suppose the boundary is smooth and
 satisfies \eqref{bdry-assumption1}.
 Suppose, in addition to \eqref{elliptic}, \eqref{concave}, \eqref{nondegenerate}, \eqref{addistruc},  $\psi\in C^\infty(\bar M)$ and
  $\varphi\in C^\infty (\partial M)$, that 
    Dirichlet problem  \eqref{mainequ} and \eqref{mainequ1} admits an admissible subsolution
   $\underline{u}\in C^{2,1}(\bar M)$.
Then the Dirichlet problem   admits a unique smooth admissible solution $u$ with
 \begin{equation}
 \label{c2alpha-estimate1}
\sup_{ M} \Delta u\leq C,
 \end{equation}
where $C$ is a uniformly positive constant depending on   
$\partial M$ up to third order derivatives, $|\varphi|_{C^{3}(M)}$,  $\sup_M|\nabla\psi|$, $\inf_{z\in M}\inf_{\xi\in T^{1,0}_z M, |\xi|=1} \partial\overline{\partial}\psi(\xi,\bar\xi)$  and other known data  but not on
 $(\delta_{\psi,f})^{-1}$.
 In addition, if $\partial M$ is \textit{holomorphically flat}  and $\varphi\equiv \mathrm{constant}$ then $C$ depends on
 $\partial M$  up to second order derivatives and other known data.
  
\end{theorem}

\begin{remark}
Throughout this paper we say $C$ does not depend  on $(\delta_{\psi, f})^{-1}$ 
if it remains uniformly bounded   as $\delta_{\psi, f}$ tends to zero, while we say  
 $\kappa$ depends not on $\delta_{\psi, f} $ if $\kappa$ has a uniformly positive lower bound as $\delta_{\psi, f}\rightarrow 0$.  Moreover, in  the theorems, we assume  
$\varphi$ is extended to $\bar M$ with the same regularity, still denoted $\varphi$.
\end{remark}

 \begin{theorem}
 \label{thm3-diri-estimate-de}
Suppose $\partial M$ is  smooth and satisfies \eqref{bdry-assumption1}. 
Let $\psi\in C^{1,1}(\bar M)$, $\varphi\in C^{2,1}(\partial M)$ and 
$f\in C^\infty(\Gamma)\cap C(\bar \Gamma)$, $\bar\Gamma=\Gamma\cup\partial\Gamma$.
 Suppose that there is  a strictly  admissible subsolution  satisfying for some $\delta_0>0$ there holds
 \begin{equation}
\label{existenceofsubsolution-de}
\begin{aligned}
f(\lambda(\mathfrak{g}[\underline{u}])) \geq  \psi +\delta_0  \mbox{ in } \bar M, \mbox{   }
\underline{u}= \varphi   \mbox{ on } \partial  M. 
\end{aligned}
\end{equation}
Suppose in addition that  $\underline{u}\in C^{2,1}(\bar M)$,   \eqref{elliptic}, \eqref{concave} and \eqref{addistruc} hold. Then
  Dirichlet problem \eqref{mainequ} and \eqref{mainequ1} for degenerate equations  with
  $\delta_{\psi,f}=0$ supposes a  (weak)   solution $u\in C^{1,\alpha}(\bar M)$, $\forall 0<\alpha<1$,
 with $\lambda(\mathfrak{g}[u])\in \bar \Gamma$ and $\Delta u \in L^{\infty}(\bar M)$.
 \end{theorem}
 
Without specific clarification,  throughout this paper, $(X, J_X, \omega_X)$ is a closed 
Hermitian manifold of complex dimension $(n-1)$ and $(S, J_S, \omega_S)$ is a compact Riemann surface with boundary.
 Also, we denote 
   \begin{equation}
   \label{cone-set}
  \begin{aligned}
   \mathfrak{F}(\varphi)=\{w\in C^{2}(\bar M): w|_{\partial M}=\varphi, \mbox{  } \lambda (\mathfrak{g}[w]+ c\pi_2^*\omega_S)\in \Gamma  \mbox{ in } \bar M  \mbox{ for some } c>0\}  \nonumber
  \end{aligned}
\end{equation}
for $\varphi\in C^2(\partial M)$, where and hereafter 
$$
\pi_1: X\times S\rightarrow X \mbox{ and }
\pi_2: X\times S\rightarrow S
$$
denote the nature projections, and $\nu$ denotes the unit inner normal vector along boundary.

\vspace{1mm}
Let $(M,J,\omega)=(X\times S,J,\omega)$ be a  product of  
$(X, J_X, \omega_X)$ with 
$(S, J_S, \omega_S)$,  
which equips with the induced complex structure $J$ and with Hermitian metric $\omega$ being compatible with $J$ 
($\omega$ is not necessary to be $\omega=\pi_1^*\omega_X+\pi_2^*\omega_S$). 
On such products,  we are able to construct  (strictly) \textit{admissible} subsolutions with 
  $\frac{\partial}{\partial \nu} \underline{u}|_{\partial M}< 0$ for Dirichlet problem \eqref{mainequ} and \eqref{mainequ1}, provided
  $\eta^{1,0}=\pi_2^*\eta^{1,0}_S$ for  a smooth $(1,0)$-form $\eta^{1,0}_S$ on 
 $S$, and there holds 
  \begin{equation}
  \label{cone-condition1}
  \begin{aligned}
 \lim_{t\rightarrow +\infty} f(\lambda(\mathfrak{g}[w]+ t\pi_2^*\omega_S))>\psi  \mbox{ in } \bar M \mbox{ for some } w\in \mathfrak{F}(\varphi).
 \end{aligned}
\end{equation}
Condition \eqref{cone-condition1}  is always satisfied if    $\mathfrak{F}(\varphi)\neq \emptyset$
  (which is necessary for the solvability of Dirichlet problem  
  within the framework of elliptic equations) and
  \begin{equation}
  \label{unbound} 
  \begin{aligned}
\lim_{t\rightarrow+\infty}f(\lambda_1,\cdots,\lambda_n+t)=\sup_\Gamma f, \mbox{  } \forall \lambda=(\lambda_1,\cdots,\lambda_n)\in \Gamma.
\end{aligned}
\end{equation}
  Also, please refer to Section \ref{constructionofsubsolutions} for idea and details.
As a result, we obtain

   \begin{theorem}
  \label{mainthm-10-degenerate}  
 Let  $(M, J,\omega)=(X\times S,J, \omega)$
 be as mentioned above.  
 Suppose, in addition to \eqref{elliptic}, \eqref{concave} and \eqref{addistruc}, that   $\eta^{1,0}=\pi_2^*\eta^{1,0}_S$ and condition
  \eqref{cone-condition1} holds for some  $w\in C^{2,1}(\bar M)\cap \mathfrak{F}(\varphi)$.
   Then the following statements are true. 
   \begin{itemize}
   \item Suppose in addition that the given data $\partial S\in C^3$, $\varphi\in C^{3}(\partial M)$,  
     $\psi \in C^{2}(\bar M)$ and $\inf_M\psi>\sup_{\partial \Gamma}f$. Then 
  Dirichlet problem \eqref{mainequ} and \eqref{mainequ1}  admits a 
   unique $C^{2,\alpha}$-smooth admissible solution   for some $0<\alpha<1$. 
    In particular, the solution is smooth,
   if the given data $ \varphi$, $\psi$, $\partial S$ are all smooth.
\item 
If $\partial S\in C^{2,1}$, 
$\varphi\in C^{2,1}(\partial M)$, $\psi\in C^{1,1}(\bar M)$, $f\in C^\infty(\Gamma)\cap C(\bar\Gamma)$ and  
 $\inf_{M} \psi=sup_{\partial \Gamma}f$,
    then the Dirichlet problem for degenerate equation
    has a weak solution 
   $u\in C^{1,\alpha}(\bar M)$, $\forall 0<\alpha<1$,
 with  $\lambda(\mathfrak{g}[u])\in \bar \Gamma$ and $\Delta u \in L^{\infty}(\bar M)$. 
 \end{itemize}

\end{theorem}

Applying each one of Theorems \ref{thm3-diri-estimate-de} and  \ref{mainthm-10-degenerate} to 
  homogeneous  complex Monge-Amp\`ere equation on $M=X\times A$, 
 where $A=\mathbb{S}^1\times [0,1]$ and $X$ is a closed K\"ahler manifold, one immediately obtains Chen's \cite{Chen}  result on existence and regularity of (weak) geodesics in the space of K\"ahler metrics \cite{Donaldson99,Semmes92,Mabuchi87}. See for instance \cite{Arezzo2003Tian,Blocki09geodesic,Chen2008Tian,Donaldson2002Holomorphic,Lempert2013Vivas,Phong-Sturm2010}  
for  complement and progress on understanding how geodesics, homogeneous complex Monge-Amp\`ere equation, are related the geometry of $X$.

  \vspace{1mm}
 Due to the importance and interests of Calabi-Yau theorem from 
 K\"ahler geometry,  algebraic geometry and mathematical physics, several generalizations of Calabi-Yau 
 theorem to other special non-K\"ahler metrics, 
 including Gauduchon and balanced metrics introduced in \cite{Gauduchon77,Michelsohn1982Acta},
have been studied by many specialists over the past several decades. 
We are also guided towards the study of Dirichlet problem of
Monge-Amp\`ere equation for $(n-1)$-PSH functions  associated with Gauduchon's conjecture.

\begin{theorem}
\label{Gauduchon-yuan}
Let $(M,J,\omega)$ be a compact Hermitian manifold with smooth \textit{mean pseudoconcave} boundary, 
$\varphi$ and $\phi$ be smooth. 
For $u$ we denote
\begin{equation}
\label{Omega_u}
\begin{aligned}
\Omega_u^{n-1}=\omega_0^{n-1}+\partial\gamma+\overline{\partial \gamma},
\mbox{  } \gamma=\frac{\sqrt{-1}}{2}\overline{\partial} u\wedge\omega^{n-2}.
\end{aligned}
\end{equation}
Suppose the Dirichlet problem of Monge-Amp\`ere equation for $(n-1)$-PSH functions
\begin{equation}
\label{MA-n-1}
\begin{aligned}
\left(*\frac{1}{(n-1)!}\Omega_u^{n-1}\right)^{n}=e^{(n-1)\phi} \omega^n \mbox{ in } \bar M, \mbox{  } u=\varphi \mbox{ on } \partial M
\end{aligned} 
\end{equation}
has a $C^{2,1}$-smooth subsolution $\underline{u}$ with $*\Omega_{\underline{u}}^{n-1}>0$ in $\bar M$.
Then the Dirichlet problem admits a unique smoothly solution $u$ with $*\Omega_{{u}}^{n-1}>0$.

 Moreover, if $(M, J, \omega)=(X\times S, J, \pi_1^*\omega_X+\pi_2^*\omega_S)$ is a product with factor $(X,J_X,\omega_X)$ being balanced,
then such subsolutions can be constructed.
 
\end{theorem}

Theorem \ref{Gauduchon-yuan} is a consequence of Theorem \ref{thm1-general-equation} below, in which 
we prove the existence for more general equations than Monge-Amp\`ere equation for $(n-1)$-PSH functions
\begin{equation}
\label{mainequ-gauduchon-general*}
\begin{aligned}
 f(\lambda(*\Phi[u]))=\psi  
\end{aligned}
\end{equation}
where 
  $\Phi[u] =*\chi+\frac{1}{(n-2)!}\sqrt{-1}\partial\overline{\partial}u\wedge\omega^{n-2}+\frac{\varrho}{(n-1)!}\mathfrak{Re}(\sqrt{-1}\partial u\wedge \overline{ \partial}\omega^{n-2})$, $\chi$ is a smooth real $(1,1)$-form, and $\varrho$ is a function.
 In particular,    
  $\Phi[u]=*\chi+\frac{1}{(n-2)!}(\partial\gamma+\overline{\partial \gamma})$  if $\varrho=n-1$; while $\varrho=0$
  it goes back to equation \eqref{mainequ} as special cases.

 \begin{theorem}
\label{thm1-general-equation}
Let $(M, J, \omega)$ be a compact Hermitian manifold with smooth boundary.
 Suppose,  in addition to \eqref{elliptic}, \eqref{concave}, \eqref{nondegenerate} and \eqref{addistruc}, that
  $\psi$, $\varphi$, $\varrho$  are all smooth.
Assume that there is an $C^{2,1}$-admissible subsolution to 
equation  \eqref{mainequ-gauduchon-general*}  with prescribing boundary data \eqref{mainequ1}. 
In addition we assume  \eqref{bdry-assumption1-Gauduchon} holds. 
Then the Dirichlet problem admits a unique smooth admissible solution.

Furthermore, if $(M, J, \omega)=(X\times S, J,\pi_1^*\omega_X+\pi_2^*\omega_S)$ and $\omega_X$ is balanced, then we can construct such (strictly) subsolutions, provided 
\begin{equation}
\label{cone-condition1-general}
\begin{aligned}
\lim_{t\rightarrow +\infty}f(\lambda(*\Phi[\underline{v}]+t\pi_1^*\omega_X))>\psi \mbox{ in } \bar M
\end{aligned}
\end{equation}
for an admissible function $\underline{v}\in C^{2,1}(\bar M)$ with $\underline{v}|_{\partial M}=\varphi$.

\end{theorem}
 
Also, corresponding results for degenerate equations are presented in Section \ref{appendix-gauduchon}.

 \vspace{1mm}
   If there exists an admissible function $\underline{v}$ with boundary data
   $\underline{v}|_{\partial M}=\varphi$,
then  \eqref{cone-condition1-general}  leads naturally to the following condition being much more broader than \eqref{unbound},
 \begin{equation}
 \label{unbound-strong} 
 \begin{aligned}
\lim_{t\rightarrow+\infty}f(\lambda_1+t,\lambda_2+t, \cdots,\lambda_{n-1}+t, \lambda_n)=\sup_\Gamma f, \mbox{  } \forall \lambda=(\lambda_1,\cdots,\lambda_n)\in \Gamma.
\end{aligned}
\end{equation}
 Such an admissible function $\underline{v}$ exists if there is  $\underline{w}$ with prescribing boundary data
  $\underline{w}|_{\partial M}=\varphi$ and 
  \begin{equation}
 \label{admiss-value2} 
 \begin{aligned}
 \lambda(*\Phi[\underline{w}]+t\pi_1^*\omega_X)\in \Gamma \mbox{ for }t\gg1.
 \end{aligned}
 \end{equation}
  
 \vspace{1mm}
In contrast with the condition $(1.6)$ of \cite{Song2008On}  for $J$-flow, cone condition in \cite{Fang11LaiMa} for 
complex inverse $\sigma_k$ flow as well as the notion of $\mathcal{C}$-subsolution,
our conditions \eqref{cone-condition1} and \eqref{cone-condition1-general} are much more 
easy to verify, since these two cone conditions restrict only to asymptotic behavior along the
  directions $\pi_2^* \omega_S$ and $\pi_1^* \omega_X$, respectively.
   
  \vspace{1mm}
 It is noteworthy that if $\Gamma$ is of type 2, then
  $\mathfrak{F}(\varphi)\neq \emptyset$ and \eqref{admiss-value2} hold automatically 
  and furthermore, condition \eqref{addistruc} implies condition \eqref{unbound}.
 Moreover, when $\Gamma=\Gamma_k$, $2\leq k\leq n-1$,
   \eqref{admiss-value2}  is clearly satisfied for arbitrary $C^2$ boundary data
   since $\Gamma_{\mathbb{R}^1}^{\infty}= \mathbb{R}.$

 \vspace{1mm}
In conclusion, with appropriate assumptions on boundary imposed above, 
we settle the gradient estimate for fully nonlinear elliptic equations \eqref{mainequ} and  \eqref{mainequ-gauduchon-general*}.  
 However, on complex manifolds  without imposing such assumptions,
 the gradient estimate still lacks of understanding except for certain special and restricted equations.
We are referred to  \cite{Blocki09gradient,Guanp2008Gradient,Guan2010Li,Hanani1996,Guan2015Sun,Zhangxw2010,yuan2018CJM}
  for direct proof, without using second order estimate, of
 gradient  estimate  for $\omega$-plurisubharmonic solutions 
 (with $\omega+\sqrt{-1}\partial\overline{\partial}u\geq 0$) to complex Monge-Amp\`ere equation, complex inverse $\sigma_k$ equation and complex $k$-Hessian equation. In \cite{yuan17K} the author 
 offers a unified and straightforward approach to gradient  estimate for
$\omega$-plurisubharmonic solutions to equations with \eqref{unbound}.

\vspace{2mm}
The   paper is organized as follows.
In Section \ref{mainproof} we sketch proof of gradient estimate.
In Section \ref{preliminaries} we present some useful lemmas, notation and computation. A criterion for condition \eqref{addistruc} is also obtained there.
In  Section \ref{Diri-pseudoconcave} we derive second order estimate which in particular includes quantitative boundary estimate. 
  In Section \ref{constructionofsubsolutions} we construct subsolutions on products.
  In Section \ref{Dirichlet-problem-less-rugularity} we discuss Dirichlet problem on complex manifolds with  less regular boundary and boundary data.
In Section \ref{appendix-gauduchon} we solve the Dirichlet problem for equation \eqref{mainequ-gauduchon-general*} which includes Monge-Amp\`ere equation for $(n-1)$-PSH functions associated to Gauduchon's conjecture as a special case.
 In Section \ref{further-discussion}    uniqueness of weak solution for degenerate equations and quantitative boundary estimate with different assumptions are briefly discussed.
In Appendix \ref{appendix},  we finally append the proof of Lemma \ref{yuan's-quantitative-lemma}, which is a key ingredient in the proof of Propositions \ref{proposition-quar-yuan1} and  \ref{proposition1-normal}  and so of quantitative boundary estimates.
  
  \vspace{1mm}
I wish to thank  Professors Bo Guan,  Chunhui Qiu  and Xi Zhang for their support and encouragement. 
Thanks also go to Professor Xinan Ma for useful suggestions on the presentation.

\section{Sketch of proof of gradient estimate} 
\label{mainproof}
  This paper is part of series of papers that are devoted to deriving gradient estimate and to investigating 
 fully nonlinear second order elliptic equations on Hermitian manifolds. 
 See also earlier work \cite{yuan2017,yuan17K} and follow-up work \cite{yuan2019Kahlercone}.
 
\vspace{1mm}
The gradient estimate in this paper is based on blow-up argument used in \cite{Dinew2017Kolo,Gabor,Chen}.
 To do it, it is required to establish second order estimate of the form
 \begin{equation}
 \label{sec-estimate-quar1}
\begin{aligned}
\sup_M \Delta u \leq C(1+\sup_M |\nabla u|^2).
\end{aligned}
\end{equation}
 
When the  background space is a closed K\"ahler manifold,  such a second order estimate was obtained by Hou-Ma-Wu  \cite{HouMaWu2010} for complex $k$-Hessian equation. 
Recently, Hou-Ma-Wu's result was extended extensively by  Sz\'ekelyhidi \cite{Gabor} to more general fully nonlinear elliptic equations on closed  Hermitian manifolds
under the assumption of the existence of $\mathcal{C}$-subsolutions, also by Tosatti-Weinkove  \cite{Tosatti2017Weinkove,Tosatti2013Weinkove} and Zhang \cite{ZhangDk}
for the Monge-Amp\`ere equation for $(n-1)$-PSH functions (without gradient terms, $\varrho=0$) and complex $k$-Hessian equation on closed Hermitian manifolds, respectively.

\vspace{1mm}
When $M$ has boundary, i.e. $\partial M\neq \emptyset$, the proof of \eqref{sec-estimate-quar1} is much more complicated. 
By treating with 
two different types of complex derivatives due to the gradient terms in equation carefully,  we prove in Theorem \ref{globalsecond-Diri} that
\begin{equation}
\label{quantitative-2nd-boundary-estimate}
\begin{aligned}
\sup_{ M} \Delta u
\leq C(1+ \sup_{M}|\nabla u|^{2} +\sup_{\partial M}|\Delta u|).  
\end{aligned}
\end{equation}
 To achieve our goal, with the estimate above at hand, a specific problem that we have in mind is to 
   derive the following quantitative boundary estimate  
  \begin{equation}
  \label{bdy-sec-estimate-quar1}
\begin{aligned}
\sup_{\partial M} \Delta u \leq C(1+\sup_M |\nabla u|^2).
\end{aligned}
\end{equation}
This is done in Theorem \ref{mix-general-thm1} 
when $\partial M$ satisfies \eqref{bdry-assumption1}. 
 
The proof is based on the following proposition. 
\begin{proposition}
\label{proposition-quar-yuan1}
Let $(M, J,\omega)$ be a compact Hermitian manifold with $C^2$ boundary. 
In addition we assume condition \eqref{bdry-assumption1} is satisfied.
We denote   
$T^{1,0}_{\partial M}=T^{1,0}_{\bar M} \cap T^{\mathbb{C}}_{ \partial M}$,   
$\xi_n=\frac{1}{\sqrt{2}}(\mathrm{{\bf \nu}}-\sqrt{-1}J\nu)$ and 
   $\nu$ denotes as defined above the unit inner normal vector along the boundary.
Suppose, in addition to \eqref{elliptic}, \eqref{concave}, \eqref{nondegenerate}, \eqref{addistruc},
$\psi\in C^0(\bar M)$ and $\varphi\in C^2(\partial M)$,  that 
 there is a $C^2$-admissible subsolution obeying \eqref{existenceofsubsolution}.
Let $u\in  C^{2}(\bar M)$ be an \textit{admissible} solution
 to Dirichlet problem \eqref{mainequ} and \eqref{mainequ1}. Fix $x_0\in \partial M$.
Then for   $\xi_\alpha,\xi_\beta \in T^{1,0}_{\partial M, x_0}$  $(\alpha,\beta=1,\cdots, n-1)$ which satisfy
  $g(\xi_\alpha, \bar\xi_{\beta}) =\delta_{\alpha\beta}$ at $x_0$,
we have
\begin{equation}
\label{yuan-prop1}
\begin{aligned}
 \mathfrak{g}(\xi_n, J\bar \xi_n)(x_0) 
  \leq C\left(1 +  \sum_{\alpha=1}^{n-1} |\mathfrak{g}(\xi_\alpha, J\bar \xi_n)(x_0)|^2\right), 
\end{aligned}
\end{equation}
where $C$ is a uniformly positive constant  depending  only on   $|u|_{C^0(\bar M)}$,   
$|\underline{u}|_{C^{2}(\bar M)}$, $\partial M$ up to second order derivatives 
 and other known data
 (but neither on $\sup_{\bar M}|\nabla u|$ nor on $(\delta_{\psi,f})^{-1}$).
 In addition, if $\partial M$ is pseudoconcave, then condition \eqref{addistruc} can  be removed.
\end{proposition}

When the boundary is supposed to be Levi flat,
 Proposition \ref{proposition-quar-yuan1}  ($\eta^{1,0}=0$) was first proved by the author in an earlier work \cite{yuan2017}, 
 in which the following lemma proposed there is crucial.   
  
\begin{lemma}
[\cite{yuan2017}]
\label{yuan's-quantitative-lemma}
Let $\mathrm{A}$ be an $n\times n$ Hermitian matrix
\begin{equation}\label{matrix3}\left(\begin{matrix}
d_1&&  &&a_{1}\\ &d_2&& &a_2\\&&\ddots&&\vdots \\ && &  d_{n-1}& a_{n-1}\\
\bar a_1&\bar a_2&\cdots& \bar a_{n-1}& \mathrm{{\bf a}} \nonumber
\end{matrix}\right)\end{equation}
with $d_1,\cdots, d_{n-1}, a_1,\cdots, a_{n-1}$ fixed, and with $\mathrm{{\bf a}}$ variable. 
Let $\epsilon>0$ be a fixed positive constant.
Suppose that the parameter $\mathrm{{\bf a}}$ in $\mathrm{A}$ satisfies  the quadratic 
 growth condition
  \begin{equation}
 \begin{aligned}
\label{guanjian1-yuan}
\mathrm{{\bf a}}\geq \frac{2n-3}{\epsilon}\sum_{i=1}^{n-1}|a_i|^2 +(n-1)\sum_{i=1}^{n-1} |d_i|+ \frac{(n-2)\epsilon}{2n-3}.
\end{aligned}
\end{equation}
 Then the eigenvalues $\lambda=(\lambda_1,\cdots, \lambda_n)$ of $\mathrm{A}$ (possibly with an order) behavior like
\begin{equation}
\begin{aligned}
 |d_{\alpha}-\lambda_{\alpha}|
<   \epsilon, \mbox{  }\forall 1\leq \alpha\leq n-1; \mbox{  }
0\leq \lambda_{n}-\mathrm{{\bf a}}
< (n-1)\epsilon. \nonumber 
\end{aligned}
\end{equation}
\end{lemma}


 \vspace{1mm}
 Finally, we prove in Section \ref{appendix-gauduchon} 
 the quantitative boundary estimates \eqref{bdy-sec-estimate-quar1} 
 and solve the Dirichlet problem for Monge-Amp\`ere equation for $(n-1)$-PSH functions associated to Gauduchon's conjecture.
Also, the key ingredient is to set up 

 \begin{proposition}
 \label{proposition1-normal}
 Let $(M,\omega)$ be a compact Hermitian manifold with $C^2$  
 boundary. 
Suppose in addition that \eqref{bdry-assumption1-Gauduchon} holds.
 Let  $u$ be the $C^2$ admissible solution of Dirichlet problem
  \eqref{mainequ-gauduchon-general*} and \eqref{mainequ1}. We denote 
  $\mathfrak{g}[u]=\frac{1}{n-1}(\mathrm{tr}_{\omega}(*\Phi[u]))\omega-(*\Phi[u])$.
 Suppose, in addition to  \eqref{elliptic}, \eqref{concave}, \eqref{nondegenerate} and \eqref{addistruc}, that there is an admissible subsolution $\underline{u}\in C^2(\bar M)$.
 Then the key inequality \eqref{yuan-prop1} also holds for
   a uniformly positive constant depending only on $|u|_{C^0(\bar M)}$, $|\underline{u}|_{C^2(\bar M)}$, $\partial M$
up to second order derivatives and other known data (but neither on $(\inf_M \psi)^{-1}$ nor on $\sup_{ M}|\nabla u|$).
Moreover,   \eqref{addistruc} can  be removed in the case when  $\partial M$ is mean pseudoconcave.
 \end{proposition}

   Our quantitative boundary estimate  
   significantly weakens the real analytic Levi flat/\textit{holomorphically flat} assumption on boundary imposed in earlier work  \cite{yuan2017}, in which the author  
extensively extends results of Chen \cite{Chen}  and Phong-Strum \cite{Phong-Sturm2010}  
for Dirichlet problem of complex Monge-Amp\`ere equation on compact K\"ahler manifolds with  
 \textit{holomorphically flat} boundary.
It would be interesting to extend quantitative boundary estimate 
 \eqref{bdy-sec-estimate-quar1}   
  to complex manifolds with general boundary,\renewcommand{\thefootnote}{\fnsymbol{footnote}}\footnote{Shortly after T. Collins and S. Picard posted their paper \cite{Collins-Picard2019} to arXiv.org, 
 I learned  that, they prove quantitative boundary estimate \eqref{bdy-sec-estimate-quar1}
 for complex $k$-Hessian equation on general complex manifolds. 
Meanwhile,
T. Collins also informed me that 
a rather different version of  Lemma \ref{yuan's-quantitative-lemma} 
  was also proved 
 in  \cite{Collins2018Yau} independently.  Indeed it was used in a very different way than our application.  
 I want to thank T. Collins for informing me their  work.}
 as  \cite{Boucksom2012} did for complex Monge-Amp\`ere equation.

 \section{Preliminaries}
\label{preliminaries}
 We denote $\Gamma^\sigma=\{\lambda\in \Gamma: f(\lambda)>\sigma\}$  
  $Df(\lambda)=(f_1(\lambda),\cdots, f_n(\lambda))$.
 Conditions \eqref{elliptic}-\eqref{concave} tell   $\partial\Gamma^\sigma=\{\lambda\in \Gamma: f(\lambda)=\sigma\}$ is convex and
 \begin{equation}
\label{bb}
\begin{aligned}
f(\lambda)-f(\mu)\geq \sum_{i=1}^n f_i(\lambda)(\lambda_i-\mu_i), \mbox{ for } \lambda, \mu\in \Gamma.
\end{aligned}
\end{equation}
Moreover, we have some stronger results. 
 
 \begin{lemma}
[\cite{GSS14}]
\label{guan2014}
 Suppose 
  \eqref{elliptic} and \eqref{concave} hold.
 Let $K$ be a compact subset of $\Gamma$ and $\beta>0$. There is a constant $\varepsilon>0$ such that,
 for  $\mu\in K$ and $\lambda\in \Gamma$, when $|\nu_{\mu}-\nu_{\lambda}|\geq \beta$,
\begin{equation}
\label{2nd}
\begin{aligned}
\sum_{i=1}^n f_{i}(\lambda)(\mu_{i}-\lambda_{i})\geq f(\mu)-f(\lambda)+\varepsilon   (1+\sum_{i=1}^n f_{i}(\lambda)),
\end{aligned}
\end{equation}
where $\nu_{\lambda}=Df(\lambda)/|Df(\lambda)|$ denotes the unit normal vector to the level surface $\partial\Gamma^{f(\lambda)}$.

\end{lemma}

\begin{lemma}
 [\cite{Gabor}]
\label{gabor'lemma}
Suppose there exists a  $\mathcal{C}$-subsolution $\underline{u}\in C^{2}(\bar M)$.
Then there exist two positive constants $R_0$ and $ \varepsilon$ with the  property: If $|\lambda|\geq R_0$, 
then either
\begin{equation}
\label{haha1}
\begin{aligned}
F^{i\bar j}(\underline{\mathfrak{g}}_{i\bar j}-\mathfrak{g}_{i\bar j})
\geq  \varepsilon F^{i\bar j}g_{i\bar j},
\end{aligned}
\end{equation}
or
\begin{equation}
\label{2nd-case}
\begin{aligned}
F^{i\bar j}\geq  \varepsilon (F^{p\bar q}g_{p\bar q})g^{i\bar j}.
\end{aligned}
\end{equation}
Here $ (g^{i\bar j})=(g_{i\bar j})^{-1}$.
\end{lemma}

\begin{remark}
\label{addcond-remark-psiu}
 
 Lemmas \ref{guan2014} and \ref{gabor'lemma}, 
originated from 
 the work of Guan \cite{Guan12a,Guan14}, play important roles in the proof of \textit{a priori} 
 estimates. 
From the original proof of Lemma 2.2 in \cite{GSS14},
we know that the constant $\varepsilon$ in  Lemma \ref{guan2014} depends only on
$\underline{\lambda}$, $\beta$ and other known data but not on $(\delta_{f(\lambda),f})^{-1}$,
($\delta_{f(\lambda),f}=f(\lambda)-\sup_{\partial \Gamma}f$).  
By the original proof of Proposition 5 of \cite{Gabor}, the $R_0$ and $\varepsilon$ in Lemma \ref{gabor'lemma} depend only on $\underline{\lambda}$ but not on $(\delta_{f(\lambda),f})^{-1}$.
 Please refer to \cite{GSS14,Gabor} for details.
 

\end{remark}

In this paper we prove the following lemma
 which is used to prove quantitative boundary estimate and also to give a different proof of Lemma 9 of \cite{Gabor}.
Let's denote $$\vec{1}=(1,\cdots, 1)\in \mathbb{R}^n.$$

\begin{lemma} \label{asymptoticcone1}
If $f$ satisfies \eqref{elliptic} and \eqref{concave}, then the following three statements are equivalent each other.
\begin{itemize}
 \item f satisfies \eqref{addistruc}.
 \item $\sum_{i=1}^n f_i(\lambda)\mu_i>0 \mbox{ for any } \lambda, \mu\in \Gamma$. In particular 
 $\sum_{i=1}^n f_i(\lambda)\lambda_i>0.$
  \item  For each $\epsilon>0$ such that $\lambda-\epsilon\vec{1}\in \Gamma$,   
  $\sum_{i=1}^n f_i(\lambda)(\lambda_i-\epsilon)\geq 0$. 
\end{itemize}
\end{lemma}
\begin{proof}
Condition \eqref{addistruc}  yields for any given $\lambda$, $\mu \in \Gamma$, there is $T\geq1$ such that for each $t>T$, there holds
 $t\mu\in\Gamma^{f(\lambda)}$, which, together with the convexity of level sets, 
  implies $Df(\lambda)\cdot (t\mu-\lambda)>0$. Thus, $Df(\lambda)\cdot\lambda>0$ (if one takes $\mu=\lambda$)
  and $Df(\lambda)\cdot\mu>0$.
 
Given a $\sigma<\sup_\Gamma f$.  Let $a=1+c_\sigma$, where
 $c_\sigma$ is the positive  constant
 given by
  \begin{equation}
 \label{kappa1-sigma}
 f(c_\sigma \vec{1})=\sigma.
  \end{equation}
 For any $\lambda\in\Gamma$, one has $t\lambda-a\vec{1}\in\Gamma$  for
  $t\gg1$ (depending on $\lambda$ and $a$).  By \eqref{bb} and  the third statement, $f(t\lambda)\geq f(a\vec{1})+f_i(t\lambda)(t\lambda_i-a)\geq f(a\vec{1})>\sigma$.
\end{proof}
Let us sketch the proof of Lemma 9 in \cite{Gabor} by using Lemma \ref{asymptoticcone1}. 
With \eqref{bb} and the second statement,
$f(\lambda+c_\sigma\vec{1})\geq f(c_\sigma\vec{1})+f_i(\lambda+c_\sigma\vec{1})\lambda_i>\sigma$, which yields $\Gamma+c_\sigma\vec{1}\subset\Gamma^\sigma$, i.e. Part (a) of Lemma 9 in \cite{Gabor}. While the Part (b) holds for $\kappa =\frac{1}{1+ c_{\sigma}}(f((1+ c_{\sigma} )\vec{1})-\sigma)$, according to \eqref{bb} and $\sum_{i=1}^n f_i(\lambda)\lambda_i>0$ in $\Gamma$.

\vspace{2mm}
\noindent{\bf Notation.} In what follows  one uses the derivatives with respect to  Chern connection $\nabla$ of $\omega$.
In local coordinates $z=(z_1,\cdots, z_{n})$,
we write
$\partial_{i}=\frac{\partial}{\partial z_{i}}$, 
$\overline{\partial}_{i}=\frac{\partial}{\partial \bar z_{i}}$,
 $\nabla_{i}=\nabla_{\frac{\partial}{\partial z_{i}}}$,
 $\nabla_{\bar i}=\nabla_{\frac{\partial}{\partial \bar z_{i}}}$.
 For   a smooth function $v$
 \begin{equation}
 \label{jian5}
\begin{aligned}
v_i:=\,&
\partial_i v,  v_{\bar i}:=
\partial_{\bar i} v,
v_{i\bar j}:= 
\nabla_{\bar j}\nabla_{i} v=
\partial_i\overline{\partial}_j v, 
 v_{ij}:=
 \nabla_{j}\nabla_{i} v
 =\partial_i \partial_j v -\Gamma^k_{ji}v_k, 
\cdots
\end{aligned}
\end{equation}
where $\Gamma_{ij}^l$ are the Christoffel symbols
 defined  by 
$\nabla_{\frac{\partial}{\partial z_i}} \frac{\partial}{\partial z_j}=\Gamma_{ij}^k \frac{\partial}{\partial z_k}.$
The torsion and curvature tensors are  
$T^{k}_{ij}= g^{k\bar l}  (\frac{\partial g_{j\bar l}}{\partial z_i} - \frac{\partial g_{i\bar l}}{\partial z_j}),
 \mbox{ and }
 R_{i\bar jk\bar l}= -\frac{\partial^2 g_{k\bar l}}{\partial z_i \partial\bar z_j}
 + g^{p\bar q}\frac{\partial g_{k\bar q}}{\partial z_i}\frac{\partial g_{p\bar l}}{\partial \bar z_j}.$

For 
$\eta=\sqrt{-1}\eta_{i\bar j}dz_i\wedge d\bar z_j$, we denote $\eta_{i\bar j k}:=\nabla_k\eta_{i\bar j}, \mbox{   } \eta_{i\bar j k\bar l}:=\nabla_{\bar l}\nabla_k\eta_{i\bar j}.$
For simplicity,   
\begin{equation}
\label{def-chi}
\begin{aligned}
\chi[u]:=\chi(z,\partial u,\overline{\partial} u)=\tilde{\chi}+
\sqrt{-1} \partial u\wedge  \overline{\eta^{1,0}}+\sqrt{-1}  \eta^{1,0} \wedge  \overline{\partial} u.
\end{aligned}
\end{equation}
Moreover,  $\chi[u]$  satisfies \eqref{key-chi} below
 and then  
 the structural assumption $(1.6)$ in \cite{yuan2018CJM} holds automatically,
  which plays a key role in proof of global second estimate.  
 Note that $\chi[u]=\sqrt{-1}\chi_{i\bar j}dz^i \wedge d\bar z^j$ depends on both $\partial u$ and $\bar \partial u$. 
One shall use  the notation: 
\begin{equation}
\begin{aligned}
\chi_{i\bar j k} :=
\nabla_k \chi_{i\bar j}=\,&
 \nabla'_k \chi_{i\bar j}+\chi_{i\bar j,\zeta_{\alpha}}u_{\alpha k}
+\chi_{i\bar j,\bar\zeta_{\alpha}}u_{\bar \alpha k} 
= 
\chi_{i\bar j,k}+\chi_{i\bar j,\zeta_{\alpha}}u_{\alpha k}
+\chi_{i\bar j,\bar\zeta_{\alpha}}u_{\bar \alpha k},   \nonumber
\end{aligned}
\end{equation}
where $\chi_{i\bar j,k}=\nabla'_k \chi_{i\bar j}$, and $\nabla'_k \chi_{i\bar j}$
denotes the partial covariant derivative of $\chi(z, \zeta,\bar \zeta)$ when viewed as depending on $z\in M$ only,
 while the meanings of $\chi_{i\bar j,\zeta_{\alpha}}$ and  $\chi_{i\bar j,\bar\zeta_{\beta}}$ are explicit.
By \eqref{def-chi} we know that
\begin{equation}
\label{key-chi}
\begin{aligned}
\chi_{i\bar j,\zeta_{\alpha}}=\delta_{i\alpha}\bar\eta_{j}, \mbox{  } \chi_{i\bar j,\bar\zeta_{\alpha}}=\delta_{j\alpha}\eta_i, \mbox{  }
\chi_{i\bar j,k}=\tilde{\chi}_{i\bar jk}+\eta_{i,k}u_{\bar j}+u_i \eta_{\bar j,k}.
\end{aligned}
\end{equation}
So $\chi_{i\bar j k}=\chi_{i\bar j,k}+\chi_{i\bar j,\zeta_{i}}u_{i k}+\chi_{i\bar j,\bar\zeta_{j}}u_{k\bar j}.$
The above equality implies that the assumption $(1.6)$ in \cite{yuan2018CJM} automatically holds for equation \eqref{mainequ},
and thus our results partially extends the results there.
Moreover, 
\begin{equation}
\label{jinguang1}
\begin{aligned}
\chi_{i\bar j k\bar l} =\chi_{i\bar j,k\bar l}+\chi_{i\bar j,\zeta_{i}\bar l}u_{i k}+\chi_{i\bar j,\bar\zeta_{j}\bar l}u_{k\bar j}+\chi_{i\bar j,k\zeta_{i}} u_{i \bar l} 
+ \chi_{i\bar j,k\bar \zeta_{j}} u_{\bar j\bar l}+\chi_{i\bar j,\zeta_{i}} u_{i k\bar l}+\chi_{i\bar j,\bar\zeta_{j}} u_{k\bar j \bar l}.\end{aligned}
\end{equation}

\vspace{1mm}
 Given a Hermitian matrix $A=\left(a_{i\bar j}\right)$, we write
 $F^{i\bar j}(A) = \frac{\partial F}{\partial a_{i\bar j}} (A), \mbox{  }
F^{i\bar j, k\bar l}(A) = \frac{\partial^{2} F}{\partial a_{i\bar j}\partial a_{k\bar l}} (A).$
Throughout this paper, without specific clarification, we denote by $\mathfrak{g}=  \mathfrak{g}[u]$ and 
 $\underline{\mathfrak{g}}=\mathfrak{g}[\underline{u}]$ for the solution $u$ and subsolution $\underline{u}$.
 And we also denote
  $F^{i\bar j}=F^{i\bar j}((\mathfrak{g}_{i\bar j})), \mbox{  } F^{i\bar j, k\bar l}=F^{i\bar j, k\bar l}((\mathfrak{g}_{i\bar j})).$$
Then $$\sum_{i,j=1}^n F^{i\bar j} {g}_{i\bar j}=\sum_{i=1}^n  f_i,  \mbox{  }\sum_{i,j=1}^n F^{i\bar j}\mathfrak{g}_{i\bar j}=\sum_{i=1}^n  f_i\lambda_i.$

\section{Second order estimates}
\label{Diri-pseudoconcave}

 In this section we derive estimates up to second order for Dirichlet problem \eqref{mainequ} and \eqref{mainequ1}.
  First of all, we present the following lemma.
\begin{lemma}
\label{Lnablau}
Let   $u\in C^{3}(M)\cap C^1(\bar M)$ be the \emph{admissible} solution to   equation \eqref{mainequ}
 and $\mathcal{L}$ be the linearized operator which is locally given by
\begin{equation}
\label{linearoperator21}
\begin{aligned}
\mathcal{L}v
=\,& F^{i\bar j}v_{i\bar j}+F^{i\bar j}\chi_{i\bar j,\zeta_{k}}v_{k}
 + F^{i\bar j}\chi_{i\bar j,\bar\zeta_{k}}v_{\bar k}   
=F^{i\bar j}v_{i\bar j}+F^{i\bar j} v_{i}\bar\eta_{j} + F^{i\bar j} \eta_i v_{\bar j}
\end{aligned}
\end{equation}
for $v\in C^{2}(M)$.
Then, at the point where  $g_{i\bar j}=\delta_{ij}$,  one has the following identity
\begin{equation}
\label{L-gradient1}
\begin{aligned}
\mathcal{L}(|\nabla u|^2)
 =\,&
F^{i\bar j} (  u_{ki} u_{\bar k \bar j} + u_{ i\bar k} u_{ k \bar j})
 +F^{i\bar j} R_{i\bar j k\bar l}u_{l}u_{\bar k}   -2\mathfrak{Re}\{ F^{i\bar j}T^{l}_{ik} u_{\bar k} u_{l\bar j}\} \\
 \,& + 2\mathfrak{Re}\{(\psi_{ {k} }-F^{i\bar j}\chi_{i\bar j, k})u_{\bar k} \}
  + 2\mathfrak{Re}\{ F^{i\bar j}  T_{k i}^l \eta_{\bar j} u_l u_{\bar k}\}.
\end{aligned}
\end{equation}
\end{lemma}

 \subsection{Quantitative boundary estimate}
 \label{bdestimates}
We derive quantitative boundary estimate when the boundary obeys \eqref{bdry-assumption1}.
\begin{theorem}
 \label{mix-general-thm1}
 Suppose, in addition to $\partial M$ is a $C^3$-smooth   boundary, that $\varphi\in C^3(\partial M)$, 
 $\psi\in C^1(\bar M)$,  \eqref{elliptic}, \eqref{concave}, \eqref{nondegenerate},
 \eqref{existenceofsubsolution} and \eqref{addistruc} hold.  Suppose in addition that \eqref{bdry-assumption1} holds. 
 Then for any admissible solution 
$u\in C^3(M)\cap C^2(\bar M)$ to Dirichlet problem \eqref{mainequ} and \eqref{mainequ1}, we have
  \begin{equation}
 \label{good-quard-pseudoconcave}
 \begin{aligned}
 \sup_{\partial M} \Delta u \leq C(1+\sup_{\partial M}|\nabla u|^2)(1+\sup_M |\nabla u|^2),
 \end{aligned}
 \end{equation}
where $C$ is a uniformly positive constant $C$ depending on $|\varphi|_{C^{3}(\bar M)}$,  
$|\underline{u}|_{C^{2}(\bar M)}$, 
$\sup_{M}|\nabla \psi|$, $\partial M$ 
up to third derivatives,
and other known data (but neither on $\sup_{M}|\nabla u|$ nor on $(\delta_{\psi,f})^{-1}$).
Furthermore, if $\partial M$ is pseudoconcave, then \eqref{addistruc} can be removed.
 \end{theorem}
 
 Combining with \eqref{herer} below one derives \eqref{bdy-sec-estimate-quar1}.
 Moreover, we have slightly delicate statement when $M=X\times S$ and more generally $M$ admits \textit{holomorphically flat} boundary.
 
 \begin{theorem}
 \label{mix-Leviflat-thm1}
 Suppose, in addition to \eqref{elliptic}, \eqref{concave}, \eqref{nondegenerate},
 \eqref{existenceofsubsolution},  $\varphi\in C^3(\partial M)$ and
 $\psi\in C^1(\bar M)$, that $\partial M$ is \textit{holomorphically flat}.
Then for any admissible solution $u\in C^3(M)\cap C^2(\bar M)$ to Dirichlet problem \eqref{mainequ} and \eqref{mainequ1}, 
there is a uniformly positive constant  $C$ depending only on $|\varphi|_{C^{3}(\bar M)}$, $|\psi|_{C^{1}(\bar M)}$, 
$|\underline{u}|_{C^{2}(\bar M)}$, $\partial M$
up to second derivatives 
and other known data (but neither on $\sup_{M}|\nabla u|$ nor on $(\delta_{\psi,f})^{-1}$) such that
 \begin{equation}
 \label{good-quard}
 \begin{aligned}
 \sup_{\partial M} \Delta u \leq C(1+\sup_{\partial M}|\nabla u|^2)(1+\sup_M |\nabla u|^2).  
 \end{aligned}
 \end{equation}
Moreover, if the boundary data $\varphi$ is exactly a constant
then the $C$ depends only on $\partial M$ up to second order derivatives, 
$|\underline{u}|_{C^{2}(\bar M)}$, $\sup_{M}|\nabla \psi|$
and other known data. 

 \end{theorem}
 
 In particular, if $M=X\times S$ and $\varphi\in C^2(\partial S)$, then one has a more subtle result:
  $C$ in  \eqref{good-quard} indeed depends on $|\varphi|_{C^2(\bar S)}$.
  
   \begin{theorem}
 \label{mix-Leviflat-thm1-product}
 Let $M=X\times S$, $\partial S\in C^2$ and $\varphi\in C^2(\partial S)$, and we suppose \eqref{elliptic}, \eqref{concave}, \eqref{nondegenerate}
 and \eqref{existenceofsubsolution} hold.  
 Then \eqref{good-quard} holds for $C$ depending only on  $|\varphi|_{C^{2}(\bar S)}$, $|\psi|_{C^{1}(\bar M)}$, 
$|\underline{u}|_{C^{2}(\bar M)}$, $\partial S$
up to second derivatives 
and other known data.
 \end{theorem}

\subsubsection{Tangential operators on the boundary}
\label{Tangential-opera-1}
For  a  given point $x_0\in \partial M$,
we choose local holomorphic coordinates
\begin{equation}
\label{goodcoordinate1}
z=(z_{1},\ldots,z_{n}), \mbox{  } z_{j}=x_{j}+\sqrt{-1}y_{j},
 \end{equation}
  centered at $x_0$ in a neighborhood which we assume to be contained in $M_{\delta}$,
 such that $x_0=\{z=0\}$, $g_{i\bar j}(0)=\delta_{ij}$ and $\frac{\partial}{\partial x_{n}}$ is the interior normal direction to $\partial M$ at $x_0$.
 For convenience we set
\begin{equation}
t_{2k-1}=x_{k}, \ t_{2k}=y_{k},\ 1\leq k\leq n-1;\ t_{2n-1}=y_{n},\ t_{2n}=x_{n}.  \nonumber
\end{equation}

 \vspace{1mm} 
 Let  $\rho$ be the distance function to the fixed point $x_0\in \partial M$, and $\sigma$ denote 
 the distance function to the boundary. Let's denote 
 \begin{equation}
 \label{distances-domain} 
 \begin{aligned}
 \Omega_{\delta}:=\{z\in M: \rho(z)<\delta\}, \mbox{   }
 M_\delta:=\{z\in M: \sigma(z)<\delta\}.
 \end{aligned}
 \end{equation}
 Note that $|\nabla \sigma|=\frac{1}{2}$ on $\partial M$,
$(\rho^{2})_{i\bar j}(0)=\delta_{ij}$.
We know that
$\{\frac{1}{2}g_{i\bar j}\}\leq\{(\rho^{2})_{i\bar j}\}\leq  2\{g_{i\bar j}\}$,
 $\frac{1}{2}\{\delta_{ij}\}\leq \{g_{i\bar j}\}\leq 2\{\delta_{ij}\}$,
$|\nabla \sigma|\geq\frac{1}{4} $ and
$\sigma$ is $C^2$ in $\Omega_{\delta}$ for some small constant $\delta >0$.

\vspace{1mm} 
Now we derive the $C^0$-estimate,  boundary $C^1$  estimates and the boundary 
estimates for pure tangential derivatives.
Let  $w\in C^2(M)\cap C^{1}(\bar M)$ be a function satisfying
\begin{equation}
\label{supersolution}
\begin{aligned}
 \mathrm{tr}_{\omega}\mathfrak{g}[w] = 0 \mbox{ in } M, \mbox{  }
w =\varphi    \mbox{ on } \partial  M.
\end{aligned}
\end{equation}
 The solvability of Dirichlet problem \eqref{supersolution} can be found in   \cite{GT1983}.
 Together with the boundary value condition, the maximum principle yields
\begin{equation}
\label{daqindiguo1}
\begin{aligned}
  \underline{u}\leq u  \leq w,    \mbox{ in } M.  
\end{aligned}
\end{equation}
Moreover, 
\begin{equation}
\label{bdr-ind-2}
\begin{aligned}
(w-\underline{u})_\nu|_{\partial M}\geq (u-\underline{u})_{\nu}|_{\partial M}\geq 0.
\end{aligned}
\end{equation}
 Since $u-\varphi=0$ on $\partial M$, we can therefore write $u-\varphi=h\sigma \mbox{ in } \bar M_{\delta}$ where  $h=\frac{(u-\varphi)_{\nu}}{\sigma_{\nu}}=\frac{(u-\varphi)_{x_n}}{\sigma_{x_n}}$ on $\partial M$. We thus define
 the tangential operator on $\partial M$ 
 \begin{equation}
  \label{tangential-oper-general1}
\begin{aligned} 
 \mathcal{T}=\nabla_{\frac{\partial}{\partial t_{\alpha}}}- \widetilde{\eta}\nabla_{\frac{\partial}{\partial x_{n}}}, 
 \mbox{ for each fixed }
  1\leq \alpha< 2n,
\end{aligned}
\end{equation}
 where $\widetilde{\eta}=\frac{\sigma_{t_{\alpha}}}{\sigma_{x_{n}}}$, $\sigma_{x_{n}}(0)=1,$ $\sigma_{t_\alpha}(0)=0$.
  One has $\mathcal{T}(u-\varphi)=0$ on $\partial M\cap \bar\Omega_\delta$.
 The boundary value condition also gives for each $1\leq \alpha, \beta<n$, $u_\alpha(0)=\underline{u}_\alpha(0) $ and
\begin{equation}
\label{bdr-ind}
\begin{aligned}
(u_{\alpha\bar\beta}-\underline{u}_{\alpha\bar\beta})(0)=(u-\underline{u})_{\nu}(\frac{\sigma_{\alpha\bar\beta}}{\sigma_{\nu}})=
-(u-\underline{u})_{\nu}(0){L}_{\partial M}(\partial_\alpha,\overline{\partial}_{\beta})(0).
\end{aligned}
\end{equation}
Similarly, 
\begin{equation}\label{left-up}
(u-\varphi)_{t_{\alpha}t_{\beta}}(0)= (u-\varphi)_{x_{n}}(0)\sigma_{t_{\alpha}t_{\beta}}(0)=\frac{(u-\varphi)_\nu}{\sigma_\nu}(0)\sigma_{t_{\alpha}t_{\beta}}(0),
\mbox{  }\forall 1\leq \alpha,\beta<2n.
\end{equation}
Thus
\begin{equation}
\label{herer}
\begin{aligned}
\sup_{ M} |u|+\sup_{\partial M}|\nabla u |  \leq C, \mbox{  }
|u_{t_{\alpha}t_{\beta}}(0)| \leq \hat{C}, \mbox{ } \forall  1\leq \alpha,  \beta <2n,
\end{aligned}
\end{equation}
where  $C$ 
is a uniform constant depending on $|w|_{C^1(\bar M)}$ and $|\underline{u}|_{C^1(\bar M)}$,   $\hat{C}$  
is a positive constant  depending on
 $| \varphi|_{C^{2}(\bar M)}$ and other known data under control.

\vspace{1mm} 
Let's turn our attention to the setting of complex manifolds with \textit{holomorphically flat} boundary.
Given $x_0\in \partial M$, one can pick local holomorphic coordinates 
\begin{equation}
\begin{aligned}
\label{holomorphic-coordinate-flat}
(z_1,\cdots, z_n), \mbox{  } z_i=x_i+\sqrt{-1}y_i, 
\end{aligned}
\end{equation}
 centered at $x_0$ such that
 $\partial M$ is locally of the form $\mathfrak{Re}(z_n)=0$ and $g_{i\bar j}(x_0)=\delta_{ij}$. 
Under the holomorphic coordinate \eqref{holomorphic-coordinate-flat}, we can take
\begin{equation}
\label{tangential-oper-Leviflat1}
\begin{aligned}
\mathcal{T}=D:=   \pm\frac{\partial}{\partial x_\alpha}, \pm\frac{\partial}{\partial y_\alpha},  
  \mbox{  } 1\leq \alpha\leq n-1.
\end{aligned}
\end{equation}
It  is noteworthy that such local holomorphic coordinate system \eqref{holomorphic-coordinate-flat} is only needed in the proof  of Proposition \ref{mix-Leviflat}. In addition,  when $M=X\times S$,
$D=\pm \frac{\partial}{\partial x_{\alpha}},
\pm \frac{\partial}{\partial y_{\alpha}},$ where  $z'=(z_1,\cdots z_{n-1})$ is local holomorphic coordinate of $X$.

\vspace{1mm} 
For simplicity we denote  the tangential  operator on $\partial M$ by
\begin{equation}
\label{tangential-operator123}
\begin{aligned}
\mathcal{T}=\nabla_{\frac{\partial}{\partial t_{\alpha}}}-\gamma\widetilde{\eta}\nabla_{\frac{\partial}{\partial x_{n}}}.
\end{aligned}
\end{equation}
Here  $\gamma=0$ (i.e. $\mathcal{T}=\nabla_{\frac{\partial}{\partial t_{\alpha}}}=D$) if $\partial M$ is \textit{holomorphically flat}, while  for general boundary we take $\gamma=1$.
On $\partial M\cap \bar\Omega_\delta$,   $(u-\varphi)|_{\partial M}=0$  
 gives 
\begin{equation}\begin{aligned}\label{bdr-t}
\mathcal{T}(u-\varphi)=0 \mbox{ and } |(u-\varphi)_{t_{\alpha}}|\leq C\rho \sup_{\partial M}|\nabla (u-\varphi)|, \forall 1\leq \alpha<2n,
\end{aligned}
\end{equation}
see \cite{Guan1998The}.
See also \cite{Guan2010Li,Guan2015Sun}.

\subsubsection{Quantitative boundary estimates for tangential-normal derivatives}
\label{Quantitative-boundes-mix}
We derive quantitative boundary estimates for tangential-normal derivatives by using barrier functions. 
This type of construction of barrier functions  
goes back at least  to \cite{Hoffman1992Boundary,Guan1993Boundary,Guan1998The}. 
We shall point out that the constants in proof of quantitative boundary estimates,
such as $C, C_{\Phi},  C_1, C_1', C_2, A_1, A_2, A_3,$   etc, 
depend on  neither  $|\nabla u|$ nor  $(\delta_{\psi,f})^{-1}$.

By direct calculations, one derives  
$u_{x_{k} l}=u_{l x_{k}}+T^{p}_{kl}u_{p}$,
 $u_{y_{k} l}=u_{l y_{k}}+\sqrt{-1}T^{p}_{kl}u_{p}$, 
 $(u_{x_k})_{\bar j}=u_{x_k\bar j}+\overline{\Gamma_{kj}^l} u_{\bar l}$, 
 $(u_{y_k})_{\bar j}=u_{y_k\bar j}-{\sqrt{-1}}\overline{\Gamma_{kj}^l} u_{\bar l}$,
$(u_{x_k})_{i\bar j}=u_{x_ki\bar j}+\Gamma_{ik}^lu_{l\bar j}+\overline{\Gamma_{jk}^l} u_{i\bar l}-g^{l\bar m}R_{i\bar j k\bar m}u_l$,
$(u_{y_k})_{i\bar j}=u_{y_ki\bar j}+\sqrt{-1}(\Gamma_{ik}^l u_{l\bar j}-\overline{\Gamma_{jk}^l} u_{i\bar l})-\sqrt{-1}g^{l\bar m}R_{i\bar j k\bar m}u_l$,
and
\begin{equation}
\label{yuan-Bd2}
\begin{aligned}
F^{i\bar j}u_{x_{k}i\bar j}
\,&
= F^{i\bar j}u_{i\bar j x_{k}}+g^{l\bar m}F^{i\bar j}R_{i\bar j k\bar m}u_{l}-2\mathfrak{Re}(F^{i\bar j}T^{l}_{ik}u_{l\bar j}), \\
F^{i\bar j}u_{y_{k}i\bar j}
\,&
=F^{i\bar j}u_{i\bar j y_{k}}+\sqrt{-1}g^{l\bar m}F^{i\bar j}R_{i\bar j k\bar m}u_{l}+
2\mathfrak{Im}(F^{i\bar j}T^{l}_{ik}u_{l\bar j}).
\end{aligned}
\end{equation}
 Hence, one has 
 $\mathcal{L}(\pm u_{t_{\alpha}}) \geq \pm\psi_{t_{\alpha}}
 -C(1+|\nabla u|)\sum_{i=1}^n f_i -C\sum_{i=1}^n f_i |\lambda_i|$.

\vspace{1mm}
Writing $b_{1}=1+2\sup_{\bar \Omega_\delta} |\nabla u|^{2}+2\sup_{\bar \Omega_\delta} |\nabla \varphi|^{2}.$
\begin{lemma}
\label{yuan-key0}
Given $x_0\in\partial M$.
Let $u$ be a $C^3$  admissible solution to equation \eqref{mainequ}, and $\Phi$ is defined as
 \begin{equation}
 \label{Phi-def1}
 \begin{aligned}
 \Phi=\pm \mathcal{T}(u-\varphi)+\frac{\gamma}{\sqrt{b_1}}(u_{y_{n}}-\varphi_{y_{n}})^2 \mbox{ in } \Omega_\delta.
 \end{aligned}
 \end{equation}
 Then there is a positive constant $C_{\Phi}$ depending on
$|\varphi|_{C^{3}(\bar M)}$, $|\chi|_{C^{1}(\bar M)}$,
 $|\nabla \psi|_{C^{0}(\bar M)}$
 and other known data 
 such that
\begin{equation}
\label{yuan-1}
\begin{aligned}
\mathcal{L}\Phi \geq 
 -C_{\Phi}  \sqrt{b_1}  \sum_{i=1}^n f_i  - C_{\Phi}  \sum_{i=1}^n f_i|\lambda_i|
-C_\Phi  \mbox{ on } \Omega_{\delta} \nonumber
\end{aligned}
\end{equation}
 for some small positive constant $\delta$.
 In particular, if  $\partial M$ is holomorphically flat and $\varphi\equiv \mathrm{constant}$
then $C_{\Phi}$ depends on $|\chi|_{C^{1}(\bar M)}$,
 $|\nabla \psi|_{C^{0}(\bar M)}$
 and other known data.
Furthermore, if $M=X\times S$ and $\varphi\in C^2(\partial S)$ then we have barrier function $\Phi=  D u$,
 and then the constant $C_{\Phi}$ depends on $|\varphi|_{C^{2}(\bar S)}$, $|\chi|_{C^{1}(\bar M)}$,
 $|\nabla \psi|_{C^{0}(\bar M)}$ and other known data. 
  Where $D=\pm \frac{\partial}{\partial x_{\alpha}}, \pm \frac{\partial}{\partial y_{\alpha}}$ is as mentioned above.

\end{lemma}

\begin{proof}
By direct calculation and Cauchy-Schwarz inequality, one obtains $F^{i\bar j}(\widetilde{\eta})_i (u_{x_{n}})_{\bar j}=F^{i\bar j}(\widetilde{\eta})_i(2u_{n\bar j}
+\sqrt{-1}(u_{y_n})_{ \bar j})\leq C\sum_{i=1}^nf_i|\lambda_i|+ \frac{1}{\sqrt{b_1}}F^{i\bar j}u_{y_n i} u_{y_n \bar j} +C\sqrt{b_1}\sum_{i=1}^n f_i.$
Thus
\begin{equation}
\label{bdy-g1}
\begin{aligned}
\mathcal{L}(\pm\mathcal{T}u)
 \geq \,&
 -C\sqrt{b_1}\sum_{i=1}^n f_i  
  - C\sum_{i=1}^n f_i|\lambda_i|
 -\frac{\gamma}{\sqrt{b_1}} F^{i\bar j}u_{y_n i}u_{y_n \bar j}. \nonumber
 \end{aligned}
\end{equation}
Applying \eqref{yuan-Bd2}, we obtain
\begin{equation}
\label{yuan-hao2}
\begin{aligned}
\mathcal{L}((u_{y_{n}}-\varphi_{y_{n}})^2) \geq \,&
F^{i\bar j}u_{y_n i}u_{y_n \bar j}-C(1+|\nabla u|^2)   \sum_{i=1}^n f_i   - C|\nabla u|\sum_{i=1}^n f_i|\lambda_i|  - C(1+|\nabla u|). \nonumber
\end{aligned}
\end{equation}
We thus complete the proof.
\end{proof}

To estimate the quantitative boundary estimates for mixed derivatives, 
we should employ barrier function of the form
\begin{equation}
\label{barrier1}
\begin{aligned}
v= (\underline{u}-u)
- t\sigma
+N\sigma^{2}   \mbox{  in  } \Omega_{\delta},
\end{aligned}
\end{equation}
where $t$, $N$ are positive constants to be determined.

\vspace{1mm} 
In what follows we denote  $\widetilde{u}=u-\varphi$.
 Let $\delta>0$ and $t>0$ be sufficiently small with $N\delta-t\leq 0,$ such that, in $\Omega_{\delta}$,    $v\leq 0$, $\sigma$ is $C^2$ and
\begin{equation}
\label{bdy1}
\begin{aligned}
 \frac{1}{4} \leq |\nabla \sigma|\leq 2,  \mbox{  }
  |\mathcal{L}\sigma | \leq   C_2\sum_{i=1}^n f_i,     \mbox{  }
  |\mathcal{L}\rho^2| \leq C_2\sum_{i=1}^n f_{i}. 
\end{aligned}
\end{equation}
In addition, we  can choose $\delta$ and $t$  small enough  such that
\begin{equation}
\label{yuanbd-11}
\begin{aligned}
\max\{|2N\delta-t|, t\}\leq \min\{\frac{\varepsilon}{2C_{2}}, \frac{\beta}{16\sqrt{n}C_2} \},
\end{aligned}
\end{equation}
where $\beta:= \frac{1}{2}\min_{\bar M} dist(\nu_{\underline{\lambda} }, \partial \Gamma_n)$,
 $\varepsilon$ is the constant corresponding to $\beta$ in Lemma \ref{guan2014}, and $C_2$ is the constant in \eqref{bdy1}.

\vspace{1mm} 
We construct the barrier function as follows:
\begin{equation}
\label{Psi}
\begin{aligned} 
\widetilde{\Psi} =A_1 \sqrt{b_1}v -A_2 \sqrt{b_1} \rho^2 + \frac{1}{\sqrt{b_1}} \sum_{\tau<n}|\widetilde{u}_{\tau}|^2+ A_3 \Phi \mbox{ on } \Omega_\delta.
\end{aligned}
\end{equation}

\begin{proposition}
\label{mix-general}
 
 Let $(M,\omega)$ be a compact Hermitian manifold with $C^3(\bar M)$-smooth boundary (without any assumption on Levi form of boundary).
In addition we suppose the other assumptions as in Theorem  \ref{mix-general-thm1} hold. 
Then 
for any admissible solution $u\in C^3(M)\cap C^2(\bar M)$ to the Dirichlet problem
 there is a uniformly positive constant $C$ depending on $|\varphi|_{C^{3}(\bar M)}$, 
 $|\underline{u}|_{C^{2}(\bar M)}$,  
$\sup_{M}|\nabla\psi|$, $\partial M$ 
up to third derivatives
and other known data (but neither on $(\delta_{\psi,f})^{-1}$ nor on $\sup_{M}|\nabla u|$)
such that
\begin{equation}
\label{quanti-mix-derivative-00} 
|\nabla^2 u(X,\nu)|\leq C(1+\sup_{\partial M}|\nabla u|)(1+\sup_{M}|\nabla u|)
 \end{equation}
 for any $X\in T_{\partial M}$ with $ |X|=1$,   
where  $\nabla^2 u$ denotes the real Hessian of $u$.
\end{proposition}

\begin{proposition}
\label{mix-Leviflat}
With the assumptions as in Theorem \ref{mix-Leviflat-thm1}, then $\mbox{ for any } X\in  T_{\partial M}  \cap J T_{\partial M}  \mbox{ with } |X|=1$,
the admissible solution $u\in C^3(M)\cap C^2(\bar M)$ to the Dirichlet problem satisfies
\begin{equation}
\label{quanti-mix-derivative-1}
|\nabla^2 u(X,\nu)|\leq C(1+\sup_{\partial M}|\nabla u|)(1+\sup_{M}|\nabla u|)
\end{equation}
where $C$ depends on $|\varphi|_{C^{3}(\bar M)}$, $|\psi|_{C^{1}(\bar M)}$ and $|\underline{u}|_{C^{2}(\bar M)}$,
$\partial M$ 
up to second derivatives
and other known data (but neither on $(\delta_{\psi,f})^{-1}$ nor on $\sup_{M}|\nabla u|$). 
Moreover, when the boundary data $\varphi$ is exactly a constant
then the $C$ depends only on $\partial M$ 
up to second derivatives, $|\underline{u}|_{C^{2}(\bar M)}$,  $\sup_{M}|\nabla\psi|$
and other known data. 
\end{proposition}

 \begin{proposition}
 \label{mix-Leviflat-product}
With the assumptions as in Theorem \ref{mix-Leviflat-thm1-product}, then there exists a positive constant depending on 
$|\varphi|_{C^{2}(\bar S)}$, $|\psi|_{C^{1}(\bar M)}$, 
$|\underline{u}|_{C^{2}(\bar M)}$, $\partial S$
up to second derivatives 
and other known data such that \eqref{quanti-mix-derivative-1} holds.
 \end{proposition}

\begin{proof}
[Proof of Propositions \ref{mix-general}, \ref{mix-Leviflat} and \ref{mix-Leviflat-product}]

If $A_2\gg A_3\gg1$ then one has  $\widetilde{\Psi}\leq 0 \mbox{ on } \partial \Omega_\delta$, here we use  \eqref{bdr-t}.
Note $\widetilde{\Psi}(x_0)=0$.  
It suffices  to prove $$\mathcal{L}\widetilde{\Psi}\geq 0 \mbox{ on } \Omega_\delta.$$
(Then $\widetilde{\Psi}\leq 0 $ in $\Omega_{\delta}$, and  $(\nabla_\nu \widetilde{\Psi})(x_0)\leq 0$ gives the bound).

\vspace{1mm} 
 By a direct computation one has
\begin{equation}
\label{L-v}
\begin{aligned}
\mathcal{L}v\geq F^{i\bar j} (\mathfrak{\underline{g}}_{i\bar j}-\mathfrak{g}_{i\bar j})-C_2 |2N\sigma-t|\sum_{i=1}^n f_i +2N F^{i\bar j}\sigma_i \sigma_{\bar j}.
\end{aligned}
\end{equation}
By Lemma 6.2 in \cite{CNS3},
 $F^{i\bar j}   \mathfrak{\underline{g}}_{i\bar j}=F^{i\bar j}(\mathfrak{g})  \mathfrak{\underline{g}}_{i\bar j} \geq \sum_{i=1}^n f_i(\lambda) \underline{\lambda}_i=\sum_{i=1}^n f_i  \underline{\lambda}_i$.
Some straightforward computations yield
\begin{equation}
\begin{aligned}
\mathcal{L} (\sum_{\tau<n}|\widetilde{u}_{\tau}|^2  )
\geq
 \frac{1}{2}\sum_{\tau<n}F^{i\bar j} \mathfrak{g}_{\bar \tau i} \mathfrak{g}_{\tau \bar j}
   -C_1'\sqrt{b_1} \sum_{i=1}^n f_{i}|\lambda_{i}|  -C_1'\sqrt{ b_1}-C_1' b_1\sum_{i=1}^n f_{i}, \nonumber
\end{aligned}
\end{equation}
where we use  the elementary inequality $|a-b|^2\geq  \frac{1}{2}|a|^2- |b|^2$.
By Proposition 2.19 in  \cite{Guan12a}, there is an index $r$ so that $\sum_{\tau<n} F^{i\bar j}\mathfrak{g}_{\bar\tau i}\mathfrak{g}_{\tau \bar j}\geq   \frac{1}{4}\sum_{i\neq r} f_{i}\lambda_{i}^{2}.$
So,
\begin{equation}
\label{L-u-2}
\begin{aligned}
\mathcal{L} (\sum_{\tau<n}|\widetilde{u}_{\tau}|^2  )
\geq
 \frac{1}{8}\sum_{i\neq r} f_{i}\lambda_{i}^{2}
   -C_1'\sqrt{b_1} \sum_{i=1}^n  f_{i}|\lambda_{i}|-C_1' b_1 \sum_{i=1}^n f_{i} -C_1'\sqrt{ b_1}.
\end{aligned}
\end{equation}
In particular, if $M=X\times S$ and $\varphi\in C^2(\partial S)$, then $\widetilde{u}_{\tau}=u_{\tau}$ for each $1\leq\tau\leq n-1$. 

\vspace{1mm}  
When $\lambda_r\geq 0$ where $r$ is the index as in \eqref{L-u-2},
 we can use the following inequality  
 \begin{equation}
 \label{inequ-1}
\begin{aligned}
\sum_{i=1}^n f_i |\lambda_i| = \sum_{i=1}^n f_i\lambda_i -2\sum_{\lambda_i<0} f_{i} \lambda_i
\leq  \frac{\epsilon}{16\sqrt{b_1}}\sum_{\lambda_i< 0} f_i\lambda_i^2 +(\sup_{\bar M}|\underline{\lambda}|+\frac{16\sqrt{b_1}}{\epsilon})\sum_{i=1}^n f_i  
\end{aligned}
\end{equation}
to control the bad term $\sum_{i=1}^n f_i |\lambda_i|$.

\vspace{1mm} 
The case $\lambda_r<0$ is slightly harder than the formal case $\lambda_r\geq 0$. 
If $f$ further satisfies  \eqref{addistruc}, as in \cite{yuan2017}, we have
$\sum_{i=1}^n f_i |\lambda_i| 
< 2\sum_{\lambda_i\geq 0} f_i\lambda_i \leq  \frac{\epsilon}{16\sqrt{b_1}}\sum_{\lambda_i\geq0} f_i\lambda_i^2 +\frac{16\sqrt{b_1}}{\epsilon} \sum_{i=1}^n f_i$
according to Lemma  \ref{asymptoticcone1}, then we 
can control  $\sum_{i=1}^n f_i |\lambda_i|$. 
While for general functions without assuming \eqref{addistruc},
we use a trick used in 
second version of \cite{Guan14} to treat it:
\begin{equation}
 \label{inequ-2}
\begin{aligned}
\sum_{i=1}^n f_i |\lambda_i| 
\leq  \frac{\epsilon}{16\sqrt{b_1}}\sum_{\lambda_i\geq 0} f_i\lambda_i^2
  +(\sup_{\bar M}|\underline{\lambda}|+\frac{16\sqrt{b_1}}{\epsilon})\sum_{i=1}^n f_i +\sum_{i=1}^n f_{i}(\underline{\lambda}_i-\lambda_i). 
\end{aligned}
\end{equation}
 By \eqref{inequ-1} and \eqref{inequ-2} we thus obtain
 \begin{equation}
 \label{flambda}
\begin{aligned}
\sum_{i=1}^n f_i |\lambda_i|
\leq  \frac{\epsilon}{16\sqrt{b_1}}\sum_{i\neq r} f_i\lambda_i^2 +(\sup_{\bar M}|\underline{\lambda}|+\frac{16\sqrt{b_1}}{\epsilon})\sum_{i=1}^n f_i +\sum_{i=1}^n f_{i}(\underline{\lambda}_i-\lambda_i).
\end{aligned}
\end{equation}

Taking  $\epsilon$ 
as $\epsilon=\frac{1}{C_1'+A_3C_\Phi }$.
 By straightforward but careful computation, one gets
\begin{equation}
\label{bdy-main-inequality}
\begin{aligned}
\mathcal{L}\widetilde{\Psi} \geq \,&
\{A_1 \sqrt{b_1}-(A_3 C_\Phi+C_1')   \} \sum_{i=1}^n f_{i}(\underline{\lambda}_i-\lambda_i)
 +\frac{1}{16\sqrt{b_1}} \sum_{i\neq r} f_i\lambda_i^2
\\ \,&
+ 2A_1N \sqrt{b_1} F^{i\bar j}\sigma_i \sigma_{\bar j}
   -   \{   C_1'+ A_2C_2+A_3C_\Phi  +A_1C_2  |2N\sigma-t|
    \\
  \,&
  +16(C_1'+A_3C_\Phi  )^2      + \frac{1}{\sqrt{b_1}}\sup_{\bar M}|\underline{\lambda}|(C_1'+A_3C_\Phi)
 \} \sqrt{b_1}\sum_{i=1}^n f_i
 \\ \,&
 -(C_1' +A_3 C_\Phi ).
\end{aligned}
\end{equation}

{\bf Case I}: If $|\nu_{\lambda }-\nu_{\underline{\lambda} }|\geq \beta$,   then by Lemma \ref{guan2014}
 we have
\begin{equation}
\label{guan-key1}
\begin{aligned}
\sum_{i=1}^n f_{i}(\underline{\lambda}_i-\lambda_i)  \geq \varepsilon  (1+\sum_{i=1}^n f_i ),
\end{aligned}
\end{equation}
where we take $\beta= \frac{1}{2}\min_{\bar M} dist(\nu_{\underline{\lambda} }, \partial \Gamma_n)$ as above,
$\varepsilon$ is the positive constant in Lemma \ref{guan2014}. 
Note that \eqref{yuanbd-11} implies $A_1C_2 |2N\sigma-t|\leq \frac{1}{2}A_1 \varepsilon$. Taking $A_1\gg 1$ we can derive $$\mathcal{L}\widetilde{\Psi} \geq 0 \mbox{ on } \Omega_\delta.$$

{\bf Case II}: Suppose that $|\nu_{\lambda }-\nu_{\underline{\lambda} }|< \beta$ and therefore
$\nu_{\lambda }-\beta \vec{1} \in \Gamma_{n}$ and
\begin{equation}
\label{2nd-case1}
\begin{aligned}
f_{i} \geq  \frac{\beta }{\sqrt{n}} \sum_{j=1}^n f_{j}. 
\end{aligned}
\end{equation}

As in \cite{GSS14}   there exist  two uniformly positive constants $c_0$ and $C_0$, such that
\begin{equation}
\label{bdy22}
\begin{aligned}
\sum_{i\neq r} f_i \lambda_i^2 \geq c_0 |\lambda|^2 \sum_{i=1}^n f_i -C_0 \sum_{i=1}^n f_{i}.
\end{aligned}
\end{equation}
The original  proof  uses $\sum_{i=1}^n f_i (\underline{\lambda}_i-\lambda_i)\geq 0$.
We can check that $c_0$ depends only on $\beta$ and $n$, and $C_0$ depends only on $\beta$, $n$ and $\sup_{\bar M}|\underline{\lambda} |$.

\vspace{1mm} 
All the bad terms containing $\sum_{i=1}^n f_i$ in \eqref{bdy-main-inequality} can be controlled by the good term
\begin{equation}
\label{bbvvv}
\begin{aligned}
A_1N \sqrt{b_1} F^{i\bar j}\sigma_i \sigma_{\bar j} \geq \frac{A_1N \beta \sqrt{b_1}}{16\sqrt{n}}\sum_{i=1}^n f_i  \mbox{ on } \Omega_\delta.
\end{aligned}
\end{equation}
On the other hand,   
the bad term  $ C_1' +A_3 C_\Phi $
can be dominated by \eqref{bbvvv} and
\begin{equation}
\label{lambda-sumf-1}
 \frac{c_0 }{32 \sqrt{b_1}}|\lambda|^2 \sum_{i=1}^n f_i + \frac{A_1N\beta \sqrt{b_1}}{32\sqrt{n}}\sum_{i=1}^n f_i
 \geq  \frac{\sqrt{c_0\beta A_1N}  }{16\sqrt{n}} |\lambda|\sum_{i=1}^n f_i.  \nonumber
\end{equation}
Then  $\mathcal{L}(\widetilde{\Psi}) \geq 0 $ on $\Omega_\delta$, if one chooses $A_1N\gg 1$.
Here we use
 \begin{equation}
 \label{bound-lambda}
 \begin{aligned}
  |\lambda|\sum _{i=1}^n f_{i}(\lambda)\geq b_0,   \mbox{ if } |\lambda|\geq R_0; \mbox{  }
   \sum_{i=1}^n f_i(\lambda) \geq b_0',   \mbox{ if } |\lambda|\leq R_0,
 \end{aligned}
 \end{equation}
 where  $R_0= 1+c_{\sup_{\bar M} \psi}$, $b_0= \frac{1}{2} (f(R_0\vec 1)-\sup_{\bar M} \psi)$  and $b_0'=\frac{1}{1+2R_0}  (f((1+R_0)\vec{1})-f(R_0\vec{1}) )$.
 Here  $c_{\sup_{\bar M} \psi}$ is the constant defined as in \eqref{kappa1-sigma}.
The proof uses \eqref{bb}. 
\end{proof}

\subsubsection{Proof of Proposition \ref{proposition-quar-yuan1}}

The proof follows the outline of proof of  Proposition 4.1 of \cite{yuan2017}.
Fix $x_0\in \partial M$. 
\vspace{1mm}  In what follows all the discussions
  will be given at $x_0$, and  the Greek letters $\alpha, \beta$ range from $1$ to $n-1$.
  For the local coordinate $z=(z_1,\cdots,z_n)$   given  by \eqref{goodcoordinate1},
we   assume further that $ (\underline{\mathfrak{g}}_{\alpha\bar \beta})$ is diagonal at $x_0$. 
Let's denote
\begin{equation}
\underline{A}(R)=\left(
\begin{matrix}
\mathfrak{\underline{g}}_{1\bar 1}&&  &&\mathfrak{g}_{1\bar n}\\
&\mathfrak{\underline{g}}_{2\bar 2}&& &\mathfrak{g}_{2\bar n}\\
&&\ddots&&\vdots \\
&& &  \mathfrak{\underline{g}}_{{(n-1)} \overline{(n-1)}}& \mathfrak{g}_{(n-1)\bar n}\\
\mathfrak{g}_{n\bar 1}&\mathfrak{g}_{n\bar 2}&\cdots& \mathfrak{g}_{n\overline{(n-1)}}& R  \nonumber
\end{matrix}
\right),
\end{equation} 
 \begin{equation}
A(R)=\left(
\begin{matrix}
\mathfrak{g}_{1\bar 1}&\mathfrak{g}_{1\bar 2}&\cdots &\mathfrak{g}_{1\overline{(n-1)}} &\mathfrak{g}_{1\bar n}\\
\mathfrak{g}_{2\bar 1} &\mathfrak{g}_{2\bar 2}&\cdots& \mathfrak{g}_{2\overline{(n-1)}}&\mathfrak{g}_{2\bar n}\\
\vdots&\vdots&\ddots&\vdots&\vdots \\
\mathfrak{g}_{(n-1)\bar 1}&\mathfrak{g}_{(n-1)\bar 2}& \cdots&  \mathfrak{g}_{{(n-1)}\overline{(n-1)}}& \mathfrak{g}_{(n-1)\bar n}\\
\mathfrak{g}_{n\bar 1}&\mathfrak{g}_{n\bar 2}&\cdots& \mathfrak{g}_{n \overline{(n-1)}}& R  \nonumber
\end{matrix}
\right).
\end{equation}

 The ellipticity and concavity of equation, couple with Lemma 6.2 in \cite{CNS3}, 
 yield that
\begin{equation}
\label{A-B}
\begin{aligned}
F(A)-F(B)\geq F^{i\bar j}(A)(a_{i \bar j}-b_{i \bar j})  
\end{aligned}
\end{equation}
for Hermitian matrices $A=\{a_{i\bar j} \}$ and $B=\{b_{i\bar j}\}$ with $\lambda(A)$, $\lambda(B)\in \Gamma$.
As in proof of Proposition 4.1 of \cite{yuan2017},  we can apply 
Lemma  \ref{yuan's-quantitative-lemma} 
and \eqref{A-B}
to show that there exist two uniformly positive constants $\varepsilon_{0}$, $R_{0}$
 depending  on  $\mathfrak{\underline{g}}$  and $f$, such that
\begin{equation}
\label{opppp}
\begin{aligned}
\,& f(\mathfrak{\underline{g}}_{1\bar 1}-\varepsilon_{0}, \cdots, \underline{\mathfrak{g}}_{(n-1)\overline{(n-1)}}-\varepsilon_0, R_0)\geq \psi
\end{aligned}
\end{equation}
and $(\mathfrak{\underline{g}}_{1\bar 1}-\varepsilon_{0}, \cdots, \underline{\mathfrak{g}}_{(n-1)\overline{(n-1)}}-\varepsilon_0, R_0)\in \Gamma$.
Here the openness of $\Gamma$ and the fact that if $A$ is diagonal then so is $F^{i\bar j}(A)$
 are both needed.  
Besides, the author  gives in  \cite{yuan2017} another proof for \eqref{opppp},  
with only using Lemma \ref{yuan's-quantitative-lemma} but without using \eqref{A-B}.

\vspace{1mm}
 Fix a uniformly positive constant $\epsilon'$ such that 
\begin{equation}
\label{bdry-assump-4}
\begin{aligned}
\epsilon' \sup_{\partial M} (w-\underline{u})_{\nu}\leq \frac{\varepsilon_0}{4},
\end{aligned}
\end{equation}
where $w$ is construction by \eqref{supersolution}, and $\varepsilon_0$ obeys \eqref{opppp}.
Since $\lambda_{\omega'}(-{L}_{\partial M})\in \overline{\Gamma}_{\infty}$, for the given $\epsilon'$, there exists a uniformly positive constant $C(\epsilon',{L}_{\partial M})$ (depending only on $\epsilon'$ and ${L}_{\partial M}$) so that
\begin{equation}
\label{bdr-assump-2}
\begin{aligned}
\lambda\left(\left(\begin{matrix}
 -{L}_{\partial M}(\partial_\alpha,\overline{\partial}_{\beta})+\epsilon' I_{n-1}&   0_{(n-1)\times 1} \\
0_{1\times (n-1)}&  C(\epsilon',{L}_{\partial M})
\end{matrix}\right)\right)\in \Gamma.
\end{aligned}
\end{equation}
By  \eqref{bdr-ind} one has
\begin{equation}
\label{cha1}
\begin{aligned}
A(R)=\,& 
\underline{A}(\epsilon',R)+(u-\underline{u})_{\nu}
\left(\begin{matrix}
 -{L}_{\partial M}(\partial_\alpha,\overline{\partial}_{\beta})+\epsilon' I_{n-1}&   0_{(n-1)\times 1} \\
0_{1\times (n-1)}&  C(\epsilon',{L}_{\partial M})
\end{matrix}\right),  
\end{aligned}
\end{equation}
where
\begin{equation}
\begin{aligned}
\underline{A}(\epsilon',R) 
=\,&\left(\begin{matrix}
 \mathfrak{\underline{g}}_{1\bar1}-\epsilon' (u-\underline{u})_{\nu}&   &&\mathfrak{g}_{1\bar n}\\ 
 & \ddots&&\vdots \\ 
 &&    \mathfrak{\underline{g}}_{(n-1)\overline{(n-1)}}-\epsilon' (u-\underline{u})_{\nu}& \mathfrak{g}_{(n-1)\bar n}\\
 \mathfrak{g}_{n\bar 1} &\cdots& \mathfrak{g}_{n\overline{(n-1)}} & R-(u-\underline{u})_{\nu}C(\epsilon',{L}_{\partial M}) 
\end{matrix}\right). \nonumber
\end{aligned}
\end{equation}

\vspace{1mm} 
Next, Lemma  \ref{yuan's-quantitative-lemma}
 together with \eqref{opppp} enable us to
 establish the quantitative boundary estimates for double normal derivative.

\vspace{1mm} 
Let's pick  $\epsilon=\frac{\varepsilon_0}{4}$ in 
 Lemma  \ref{yuan's-quantitative-lemma}, and set
\begin{equation}
\begin{aligned}
 R_c=\,& \frac{4(2n-3)}{\varepsilon_0}
\sum_{\alpha=1}^{n-1} | \mathfrak{g}_{\alpha\bar n}|^2
 + (n-1)\sum_{\alpha=1}^{n-1}  | \mathfrak{\underline{g}}_{\alpha\bar \alpha}|
    +\frac{(n-2)\varepsilon_0}{4(2n-3)} +R_0+ 
    (u-\underline{u})_{\nu} C(\epsilon',{L}_{\partial M}), \nonumber
\end{aligned}
\end{equation}
where $\varepsilon_0$ and $R_0$ are fixed constants so that \eqref{opppp} holds.

\vspace{1mm}
It follows from Lemma   \ref{yuan's-quantitative-lemma} 
   that the eigenvalues of $\underline{A}(\epsilon', R_c)$, 
   $$\lambda(\underline{A}(\epsilon', R_c))=(\lambda_1(\underline{A}(\epsilon', R_c)),\cdots \lambda_{n-1}(\underline{A}(\epsilon', R_c),\lambda_n(\underline{A}(\epsilon', R_c))$$
   (possibly with an order) shall behavior like
\begin{equation}
\label{lemma12-yuan}
\begin{aligned}
\,& \lambda_{\alpha}(\underline{A}(\epsilon', R_c)\geq \mathfrak{\underline{g}}_{\alpha\bar \alpha}-\epsilon' (u-\underline{u})_{\nu}-\frac{\varepsilon_0}{4}, \mbox{  } 1\leq \alpha\leq n-1, \\\,&
\lambda_{n}(\underline{A}(\epsilon', R_c)\geq R_c-(u-\underline{u})_{\nu} 
 C(\epsilon',{L}_{\partial M}),
\end{aligned}
\end{equation}
in particular, $\lambda(\underline{A}(\epsilon', R_c))\in \Gamma$. So $\lambda(A(R_c))\in \Gamma$.
We then obtain
\begin{equation}
\label{bdry-assump-3}
\begin{aligned}
F(A(R_c))\geq F(\underline{A}(\epsilon',R_c))+F^{i\bar j}\left(A(R_c))(A(R_c)-\underline{A}(\epsilon',R_c)\right)_{i\bar j}\geq F(\underline{A}(\epsilon',R_c)),
\end{aligned}
\end{equation}
here we use \eqref{bdr-ind-2},  \eqref{bdr-assump-2},
 \eqref{cha1},  Lemma 6.2 in \cite{CNS3} and Lemma \ref{asymptoticcone1} 
 to derive the last inequality; and  \eqref{concave} is used to get the first one.
 Moreover, here is the only place where we use  \eqref{addistruc} (Lemma \ref{asymptoticcone1});
 and it can be removed if $\partial M$ is pseudoconcave, since $A(R)\geq \underline{A}(R)$ ($F(A(R))\geq F(\underline{A}(R))$) in this case.

\vspace{1mm}
 Applying \eqref{opppp}, \eqref{bdry-assump-4},  \eqref{lemma12-yuan} and \eqref{bdry-assump-3},  one hence has
\begin{equation}
\begin{aligned}
F(A(R_c))\geq \,& F(\underline{A}(\epsilon',R_c))  \\
\geq \,&
f(\mathfrak{\underline{g}}_{1\bar 1}-\frac{\varepsilon_0}{2},\cdots,
 \mathfrak{\underline{g}}_{(n-1) \overline{(n-1)}}-\frac{\varepsilon_0}{2},R_c-\sup_{\partial M}(u-\underline{u})_{\nu} C(\epsilon',{L}_{\partial M}))
\\
> \,&
 f(\mathfrak{\underline{g}}_{1\bar 1}-\varepsilon_{0}, \cdots, \underline{\mathfrak{g}}_{(n-1)\overline{(n-1)}}-\varepsilon_0, R_c-\sup_{\partial M}(u-\underline{u})_{\nu} C(\epsilon',{L}_{\partial M}))\geq \psi. \nonumber
\end{aligned}
\end{equation}
Therefore,
 \begin{equation}
 \label{nornor}
\begin{aligned}
\mathfrak{g}_{n\bar n}(x_0) \leq R_c.  \nonumber
\end{aligned}
\end{equation}
We thus complete the proof of  Proposition  \ref{proposition-quar-yuan1}.

\subsection{Global second order estimate}
\label{eaution-closed}

   The primary obstruction  to establishing second order estimate for solutions of  fully nonlinear elliptic equations involving gradient terms in complex setting is two different types of complex derivatives due to the gradient terms in   equations.
 The following theorem gives global second order estimate for Dirichlet problem \eqref{mainequ} and \eqref{mainequ1}.

\begin{theorem}
\label{globalsecond-Diri} 
Suppose, in addition to \eqref{elliptic}-\eqref{nondegenerate},
that there is an \textit{admissible} subsolution $\underline{u}\in C^{2}(M)$ satisfying \eqref{existenceofsubsolution}.
Then for any admissible function
 $u\in C^{4}(M)\cap C^{2}(\bar M)$ solving Dirichlet problem \eqref{mainequ} and \eqref{mainequ1}
  with $ \psi\in C^2(\bar M)$, we have
\begin{equation}
\label{2-seglobal}
\begin{aligned}
\sup_{ M}\Delta u
\leq
C (1+ \sup_{M}|\nabla u|^{2} +\sup_{\partial   M}|\Delta u|),
\end{aligned}
\end{equation}
where $C$ is a uniform positive constant depending only on $|u|_{C^{0}(\bar M)}$, 
$|\underline{u}|_{C^{2}(\bar M)}$, $|\chi|_{C^{2}(\bar M)}$, 
$\sup_M|\nabla\psi|$, $\inf_{z\in M}\inf_{\xi\in T^{1,0}_z M, |\xi|=1} \partial\overline{\partial}\psi(\xi,\bar\xi)$ 
 and other known data (but not on $(\delta_{\psi,f})^{-1}$).
\end{theorem}


We use the method  used by Hou-Ma-Wu  \cite{HouMaWu2010} and Sz\'{e}kelyhidi \cite{Gabor}.
We denote the eigenvalues of the matrix
$A=\{A^i_{j}\}=\{g^{i\bar q }\mathfrak{g}_{j\bar q}\}$
by $(\lambda_{1},\cdots,\lambda_{n})$.
We also suppose $\lambda_1\geq \lambda_2\cdots\geq\lambda_n$ at each point.
We want to apply the maximum principle to $H$,
\begin{equation}
\label{gqy-81}
\begin{aligned}
H:=\lambda_{1}e^{\phi},
\end{aligned}
\end{equation}
where the test function $\phi$ is to be chosen later. Suppose $H$ achieves its maximum at an
 interior point $p_{0}\in M$.
 Since the eigenvalues of $A$  need not be distinct at the point $p_{0}$, 
 and so $\lambda_1(p_0)=\lambda_2(p_0)$ may occur,
  $H$ may only
 be continuous. To circumvent this
 difficulty we use the perturbation argument used in    \cite{Gabor}.
 To do this, we choose a local coordinates $z=(z_{1},\cdots,z_{n})$ around $p_0$, such that at $p_0$
 $$g_{i\bar j}=\delta_{ij}, \mathfrak{g}_{i\bar j}=\delta_{ij}\lambda_{i} \mbox{ and } F^{i\bar j}=\delta_{ij}f_i.$$
Let $B$ be a
diagonal matrix $B^{p}_{q}$ with real entries satisfying $B^{1}_{1}=0$, $B^{n}_{ n}>2 B^{2}_{ 2}$ and
$B^n_n<B^{n-1}_{ n-1}<\cdots<B^{2}_{ 2}<0$ are small.
Then we define the matrix $ \widetilde{A}=A+B$ with the eigenvalues $ \widetilde{\lambda}=( \widetilde{\lambda_{1}},\cdots, \widetilde{\lambda}_{n})$.
 At the origin, $ \widetilde{\lambda_{1}}=\lambda_{1}=\mathfrak{g}_{1\bar 1},  \widetilde{\lambda_{i}}=\lambda_{i}+B^{i}_{ i}\mbox{ if $i\geq2$ }$
 and the eigenvalues of $\tilde{A}$ define $C^{2}$-functions near the origin.

Notice $\widetilde{H}=\widetilde{\lambda_{1}}e^{\phi}$ also achieves its maximum at the same
 point $p_{0}$ (we may assume $\lambda_{1}(p_{0})=\widetilde{\lambda_{1}}(p_{0})>1$).
  In what follows,
 the computations will be given at the origin $p_{0}$.
By maximum principle one has
\begin{equation}
\label{mp1}
\begin{aligned}
 \widetilde{\lambda_{1}}_{,i} +\lambda_{1}\phi_{i}=0, \mbox{  }  \widetilde{\lambda_{1}}_{,\bar i} +\lambda_{1}\phi_{\bar i}=0, \mbox{  } 
\frac{\widetilde{\lambda_{1}}_{,i\bar i}}{\lambda_{1}}
 -\frac{|\widetilde{\lambda_{1}}_{,i}|^{2}}{\lambda_{1}^{2}}
 +\phi_{i\bar i}\leq 0.
\end{aligned}
\end{equation}
By straightforward calculations, one obtains
\begin{equation}
\label{mp2}
\begin{aligned}
 \widetilde{\lambda_{1}},_i\,& =\mathfrak{g}_{1\bar 1 k}+(B^{1}_{1})_{i}.    
\end{aligned}
\end{equation}

As in  \cite{Gabor}, one obtains
\begin{equation}
\label{xp1}
\begin{aligned}
\widetilde{\lambda_{1}}_{,k\bar k} \geq \,&
 \mathfrak{g}_{1\bar 1 k \bar k}
 +\frac{1}{2(n+1)\lambda_{1}}\sum_{p>1} (|\mathfrak{g}_{p\bar 1 k}|^{2}+|\mathfrak{g}_{1\bar p k}|^{2})-C.  \\
\end{aligned}
\end{equation}
 It follows from   \eqref{jinguang1} that
\begin{equation}
\begin{aligned}
\chi_{i\bar ik\bar k}=\,&
\chi_{i\bar i,k\bar k}
+\chi_{i\bar i,\zeta_{i}}R_{k\bar k i\bar l}u_{l}
+2\mathfrak{g}_{k\bar k}\mathfrak{Re}\left(\chi_{i\bar i, k\zeta_{k}}\delta_{ik}-\chi_{i\bar i,\zeta_{i}}T^{k}_{ki}\right)
\\
\,&
+2\mathfrak{Re}
\left(\chi_{i\bar i,\zeta_{i}\bar k}u_{i k}
\right) -2\mathfrak{Re}
\left(\chi_{i\bar i,\zeta_{i}}(\chi_{i\bar i,i}
-T^{l}_{k\alpha}\chi_{l\bar k}
)
+\chi_{i\bar i,k\bar\zeta_{i}}\chi_{i\bar k}
\right)
\\
\,&
+2\mathfrak{Re}\left(\chi_{i\bar i,\zeta_{i}}\mathfrak{g}_{k\bar k i}\right)-2\mathfrak{Re}
\left(
\chi_{i\bar i,\zeta_{i}}\chi_{k\bar k,\zeta_{k}}u_{k i}
+\chi_{i\bar i,\zeta_{i}}\chi_{k\bar k,\bar\zeta_{k}}u_{i\bar k}\right).
\end{aligned}
\end{equation}

Differentiating the equation \eqref{mainequ}  twice (using covariant derivative), we obtain
\begin{equation}
\label{mp3}
\begin{aligned}
F^{i\bar i}\mathfrak{g}_{1\bar 1 i\bar i}
\geq \,&
-F^{i\bar j,l\bar m }\mathfrak{g}_{i\bar j1}\mathfrak{g}_{l\bar m\bar 1}
-2\mathfrak{Re}(F^{i\bar i}\chi_{i\bar i,\zeta_{i}}\mathfrak{g}_{1\bar 1 i})
+ 2\mathfrak{Re}(F^{i\bar i}\bar T^{j}_{1i}u_{i\bar j 1})
\\
\,&
+2\mathfrak{Re}(F^{i\bar i}\chi_{1\bar 1,\zeta_{1}\bar i}u_{1 i})
+\psi_{1\bar 1}
-C\mathfrak{g}_{1\bar 1}\sum F^{i\bar i},              \nonumber
\end{aligned}
\end{equation}
here we use the structure of $\chi$ in \eqref{def-chi}  and also the following standard formula
\begin{equation}
\label{deco1}
\begin{aligned}
u_{1\bar j k}-u_{k\bar j 1}  = \,& T^l_{1k}u_{l\bar j},  \\  \nonumber
u_{1\bar 1 i\bar i}-u_{i\bar i 1\bar 1} =R_{i\bar i 1\bar p}u_{p\bar 1}-
R_{1\bar 1 i\bar p}u_{p\bar i} \,& +2\mathfrak{Re}\{\bar T^{j}_{1i}u_{i\bar j 1}\}+T^{p}_{i1}\bar T^{q}_{i1}u_{p\bar q}.    \nonumber
\end{aligned}
\end{equation}

Let's set $\phi=\Phi(|\nabla u|^{2})+\Psi(u-\underline{u}), $ and $K=1+\sup_{\bar M}(|\nabla u|^{2}+|\nabla (u-\underline{u})|^2),$
\begin{equation}
\begin{aligned}
\Phi(t)=-\frac{1}{2}\log  (1-\frac{t}{2K}), \mbox{  }  \nonumber
\end{aligned}
\label{test func}
\begin{aligned}
\Psi(x)=\frac{C_{*}}{(1+x -\inf_{\bar M} (u-\underline{u}))^{N}},   \nonumber
\end{aligned}
\end{equation}
 where  $\Phi(t)$ and  $\Psi(x)$  are  respectively used in   \cite{HouMaWu2010} and     \cite{yuan2018CJM}.
 Here 
   $C_{*}\geq 1$ and $N\in \mathbb{N}$ are constants to be chosen later.
We get
\begin{equation}
\begin{aligned}
\,& \Phi''=2\Phi'^{2}, \quad (4K)^{-1}<\Phi' < (2K)^{-1} \mbox{ for } t\in [0,\sup_{M}|\nabla u|^2] \\ \nonumber
\mathcal{L}\phi=\,&
\Phi'\mathcal{L}(|\nabla u|^{2})+\Psi'\mathcal{L}(u-\underline{u})
+\Phi'' F^{i\bar i}|(|\nabla u|^{2})_{i}|^{2}
+\Psi'' F^{i\bar i}|(u-\underline{u})_{i}|^{2}.  \nonumber
\end{aligned}
\end{equation}
Together with Lemma \ref{Lnablau} one has a differential inequality
\begin{equation}
\begin{aligned}
\label{key11}
0\geq \,&
\frac{1}{2}\Phi'F^{i\bar i} \left(|u_{ki}|^{2}+|u_{k\bar i}|^{2}\right)
+2\Phi'\mathfrak{Re}\{ \psi_{ {k} } u_{\bar k} \}
+\Psi'\mathcal{L}(u-\underline{u})
 \\ \,&
+\Phi'' F^{i\bar i}|(|\nabla u|^{2})_{i}|^{2}
 +\Psi'' F^{i\bar i}|(u-\underline{u})_{i}|^{2}
+2\mathfrak{Re} ( F^{i\bar i}\bar T^{1}_{1i}\frac{\widetilde{\lambda_{1}}_{,i}}{ {\lambda_{1}}_{ 1}} )
 \\ \,&
 -\frac{F^{i\bar j,l\bar m }\mathfrak{g}_{i\bar j1}\mathfrak{g}_{l\bar m\bar 1}}{\mathfrak{g}_{1\bar 1}}
 -F^{i\bar i}\frac{|\widetilde{\lambda_{1}}_{,i}|^2}{\lambda_{ 1}^2}
+\frac{\psi_{1\bar 1}}{\mathfrak{g}_{1\bar 1}}  -C \sum F^{i\bar i}.
\end{aligned}
\end{equation}

Note that in this computation $\widetilde{\lambda_{1}}$ denotes the largest eigenvalue of
the perturbed endomorphism $\widetilde{A}=A+B$. At the origin $p_{0}$ where
 we carry out the computation,  $\widetilde{\lambda_{1}}$ coincides the largest eigenvalue of $A$.
 However, at nearby points,  it is a small perturbation.
  We would take $B\rightarrow 0$, and obtain the above differential inequality
   \eqref{key11} as well.
It only holds in a viscosity
sense because the largest eigenvalue of $A$ may not be $C^{2}$ at the origin $p_{0}$, if some eigenvalues coincide.

\vspace{1mm}
\textbf{Case I}: Assume that $\delta\lambda_{1}\geq -\lambda_{n}$($0<\delta\ll \frac{1}{2}$).
Set $$I= \{i:F^{i\bar i}>\delta^{-1}F^{1\bar 1}\},\mbox{ and } J= \{i:F^{i\bar i} \leq \delta^{-1}F^{1\bar 1}\}.$$
Clearly $1\in J$ and identity \eqref{mp1} implies that
for any fixed index $i$
\begin{equation}
\label{zhuo22}
\begin{aligned}
-F^{i\bar i}\frac{|\widetilde{\lambda_{1}}_{,i}|^2}{\lambda_{1 }^2}
\geq \,&
 -2\Psi'^{2} F^{i\bar i}|(u-\underline{u})_{i}|^{2}-2\Phi'^{2} F^{i\bar i}|(|\nabla u|^{2})_{i}|^{2}\\
 =\,&
 -2\Psi'^{2} F^{i\bar i}|(u-\underline{u})_{i}|^{2}-\Phi'' F^{i\bar i}|(|\nabla u|^{2})_{i}|^{2}. \nonumber
\end{aligned}
\end{equation}
The assumption $\delta\lambda_{1}\geq -\lambda_{n}$ implies that
$\frac{1-\delta}{\lambda_{1}-\lambda_{i}}\geq \frac{1-2\delta}{\lambda_{1}}.$
Then
\begin{equation}
\begin{aligned}
\label{key33}
 -\frac{1}{\mathfrak{g}_{1\bar 1}}F^{i\bar j,l\bar m }\mathfrak{g}_{i\bar j1}\mathfrak{g}_{l\bar m\bar 1}
  -F^{i\bar i}\frac{|\widetilde{\lambda_{1}}_{, i}|^2}{\lambda_{ 1}^2}
  \geq \,&
   \sum_{i\in I}\frac{F^{i\bar i}
-F^{1\bar 1}}{\lambda_{1}-\lambda_{i}} \frac{|\mathfrak{g}_{i\bar 1 1}|^{2}}{\lambda_1}
  -F^{i\bar i}\frac{|\widetilde{\lambda_{1}}_{, i}|^2}{\lambda_{ 1}^2} \\
 \geq \,&
 (1-2\delta)\sum_{i\in I}F^{i\bar i}\frac{|\mathfrak{g}_{i\bar 1 1}|^2-|\widetilde{\lambda_{1}}_{, i}|^2}{\mathfrak{g}_{1\bar 1}^2}
 \\
 \,&
 -4\delta \Psi'^2 \sum_{i\in I} F^{i\bar i}|(u-\underline{u})_{i}|^2
 \\
 -2\Psi'^2\delta^{-1} \,& F^{1\bar 1}K
 -\Phi'' F^{i\bar i} |(|\nabla u|^2)_{i}|^2.
\end{aligned}
\end{equation}

The estimates for the lower bound of the differences $|\mathfrak{g}_{i\bar 1 1}|^{2}-|\mathfrak{g}_{1\bar 1 i}|^{2}$
for $i\in I$ are crucial. A straightforward computation gives
\begin{equation}
\begin{aligned}
\mathfrak{g}_{i\bar 1 1}
=\,&
\widetilde{\lambda_{1}}_{, i}+
\tau_{i}-\chi_{1\bar 1,\zeta_{\alpha}}u_{\alpha i}
+\chi_{i\bar 1,\zeta_{\beta}}u_{\beta 1} \\           \nonumber
=\,&
\widetilde{\lambda_{1}}_{, i}+
\tau_{i}-\chi_{1\bar 1,\zeta_{1}}u_{1 i}
+\chi_{i\bar 1,\zeta_{i}}u_{i 1},
\end{aligned}
\end{equation}
where
$\tau_i=\chi_{i\bar 1,1}-\chi_{1\bar 1, i}
+\chi_{i\bar 1,\bar\zeta_{ \alpha}} u_{\bar\alpha 1}
-\chi_{1\bar 1,\bar\zeta_{\alpha}}u_{\bar\alpha i}
+T_{i1}^lu_{l\bar 1}+(B^1_1)_i.$  Here, \eqref{key-chi} is crucial.
Combining it with \eqref{mp1} and \eqref{mp2} one has
\begin{equation}
\label{key119111}
\begin{aligned}
 F^{i\bar i} \frac{(|\mathfrak{g}_{i\bar 1 1}|^2-|\widetilde{\lambda_{1}}_{,i}|^2)}{\mathfrak{g}_{1\bar 1}^2} \geq \,&
 -CF^{i\bar i} \frac{ |\mathfrak{g}_{1\bar 1 i} |}{\mathfrak{g}_{1\bar 1 }}(1 +\frac{|u_{1i}|}{\mathfrak{g}_{1\bar 1}})   \\
 \geq \,&
 -\frac{C }{\sqrt{K}}F^{i\bar i}(|u_{k\bar i}|+|u_{ki}|)
  \\\,&
 -\frac{C}{\mathfrak{g}_{1\bar 1} \sqrt{2K}} F^{i\bar i} |u_{ki}|^2
 -C |\Psi'|F^{i\bar i}|(u-\underline{u})_i|.
\end{aligned}
\end{equation}

We now choose small $\delta$ such that $4\delta \Psi'^{2} < \frac{1}{2}\Psi''$.
It follows from \eqref{mp1}, \eqref{key11}, \eqref{key33},  \eqref{key119111} and an elementary inequality
$\frac{1}{2}\Psi''F^{k\bar k}|(u-\underline{u})_{k}|^{2}-C|\Psi'|F^{k\bar k}|(u-\underline{u})_{k}|\geq -\frac{C^2\Psi'^2}{2\Psi''}\sum F^{i\bar i}$ that
\begin{equation}
\begin{aligned}
\label{key001}
0\geq \,&
 \frac{1}{16K}F^{k\bar k} \left(|u_{ik}|^{2}+|u_{i\bar k}|^{2}\right)
+\Psi'\mathcal{L}(u-\underline{u})
 -2\Psi'^2\delta^{-1}F^{1\bar 1}K
\\\,&
+2\Phi'\mathfrak{Re}\{ \psi_{ {k} } u_{\bar k} \}
 -(\frac{C_0 \Psi'^2}{ \Psi''}+C) \sum F^{i\bar i}
   +\frac{\psi_{1\bar 1}}{\mathfrak{g}_{1\bar 1}}.
\end{aligned}
\end{equation}

\textbf{Case II}:
 We assume that $\delta\lambda_{1}<-\lambda_{n}$ with the constants $C_{*}, N$ and $\delta$ fixed as in the previous case.
Then $|\lambda_{i}|\leq \frac{1}{\delta}|\lambda_{n}| \mbox{ for all } i$. Moreover
\begin{equation}
\label{zhuo222}
\begin{aligned}
-\sum F^{i\bar i}\frac{|\widetilde{\lambda_{1}}_{,i}|^{2}}{\lambda_{1}^{2}}
\geq \,&
 -2\Psi'^{2}\sum F^{i\bar i}|(u-\underline{u})_{i}|^{2}-2\Phi'^{2} \sum F^{i\bar i}|(|\nabla u|^{2})_{i}|^{2}.  \nonumber
\end{aligned}
\end{equation}
Combining it with \eqref{mp1}  and \eqref{key11}, one derives that
\begin{equation}
\begin{aligned}
0\geq \,&
\frac{1}{16K}F^{i\bar i}(|u_{ki}|^{2}+|u_{k\bar i}|^{2})
+(\Psi''-2\Psi'^{2})F^{i\bar i}|(u-\underline{u})_{i}|^{2}
\\
\,&
+2\Phi' \mathfrak{Re}\{ \psi_{{k} } u_{\bar k} \}
- C |\Psi'| F^{i\bar i}|(u-\underline{u})_{i}|
+\frac{\psi_{1\bar 1}}{\mathfrak{g}_{1\bar 1}}-C \sum F^{i\bar i},    \nonumber
 \nonumber
\end{aligned}
\end{equation}
here we use  the elementary inequality $ax^2-bx\geq -\frac{b^2}{4a}$ for $a>0$.

\vspace{1mm} 
Note that $F^{n\bar n} \geq \frac{1}{n} \sum F^{i\bar i}$ and $|u_{n\bar n}|^{2}=
\frac{1}{2}|\mathfrak{g}_{n\bar n}|^2- |\chi_{n\bar n}|^{2}\geq  \frac{\delta^2}{2}|\mathfrak{g}_{1\bar 1}|^{2}-C_2(1+|\nabla u|^2),$
one has
\begin{equation}
\label{hananguo1}
\begin{aligned}
0\geq \,&
\frac{\delta^2}{32nK} |\mathfrak{g}_{1\bar 1}|^2\sum F^{i\bar i}
 +(\Psi''-2\Psi'^{2})F^{i\bar i}|(u-\underline{u})_{i}|^{2}
\\
\,&
+2\Phi' \mathfrak{Re}\{ \psi_{ {k} } u_{\bar k} \}
-C (1+|\Psi'| \sqrt{K}) \sum F^{i\bar i}
 +\frac{\psi_{1\bar 1} }{\mathfrak{g}_{1\bar 1}}.
\end{aligned}
\end{equation}

\begin{proof}
 [Proof of Theorem \ref{globalsecond-Diri}]

  Suppose    $\mathfrak{g}_{1\bar 1}\gg 1+\sup_M |\nabla u|^2$ (otherwise we are done).

\vspace{1mm} 
In case I, we assume $\delta\lambda_{1}\geq -\lambda_{n}$ ($0<\delta\ll \frac{1}{2}$), and then have  \eqref{key001}. In the proof we use  Lemma \ref{guan2014}.
In the first subcase one assumes
\begin{equation}
\label{bbbbb3}
\begin{aligned}
\mathcal{L}(\underline{u}-u)\geq \varepsilon (1+ \sum F^{i\bar i}).
\end{aligned}
\end{equation}
Moreover, based on \eqref{key001}, we choose $C_{*}\gg 1$ and $N\gg 1$   such that $\frac{C_0|\Psi'|}{ \Psi''}\leq\frac{\varepsilon}{2 }$ and
 \begin{equation}
\begin{aligned}
-\frac{\varepsilon}{2}\Psi'\geq
 \frac{\varepsilon C_{*}N}{2(1+\sup_{\bar M}(u-\underline{u})-\inf_{\bar M}(u-\underline{u}))^{N+1}} \gg 1. \nonumber
  \end{aligned}
\end{equation}
 Thus one derives
\begin{equation}
\label{wukong1}
\begin{aligned}
\mathfrak{g}_{1\bar 1} \leq CK.   
\end{aligned}
\end{equation}
%
%
%
%
%

 In the other subcase, we can assume $f_{i} \geq  \frac{\beta }{\sqrt{n}} \sum_{j=1}^n f_{j} \mbox{ for each } i.$  
Note that
\begin{equation}
\label{tangsen}
\begin{aligned}
|u_{i\bar i}|^{2}= |\mathfrak{g}_{i\bar i}-\chi_{i\bar i}|^{2}
\geq \frac{1}{2}|\mathfrak{g}_{i\bar i}|^{2}-C(1+|\nabla u|^2).       \nonumber
\end{aligned}
\end{equation}
Combining it with   \eqref{key001} and \eqref{bound-lambda}, we derive \eqref{wukong1}. 

\vspace{1mm} 
In  case II
 we assume that $\delta\lambda_{1}<-\lambda_{n}$ with the constants $C_{*}, N$ and $\delta$ fixed as in the previous case. By \eqref{hananguo1} we have \eqref{wukong1}.
Here we use  condition  \eqref{bound-lambda}  again.


\end{proof}

With \textit{admissible} subsolutions   replaced by 
  $\mathcal{C}$-subsolutions, 
 we can prove \eqref{2-seglobal}. Here we use Lemma \ref{gabor'lemma}.
Based on it,
as in \cite{Gabor}, we  derive $C^0$-estimate and   gradient estimate for equation \eqref{mainequ} on closed Hermitian manifolds.
  Together with  Evans-Krylov theorem, 
one then slightly extends Sz\'{e}kelyhidi's $C^{2,\alpha}$-estimates  
for  fully nonlinear elliptic equations with $\eta^{1,0}=0$  in \cite{Gabor} and also partially extends a result in  \cite{yuan2018CJM}.
See also \cite{CollinsJacobYau,Dinew2017Kolo,Sun2017CPAM,Tosatti2010Weinkove,Tosatti2017Weinkove,Tosatti2013Weinkove,ZhangDk} 
  for $C^{2,\alpha}$-estimates for some special equations.  

\begin{proposition}
\label{C2alpha-Thm1}
Let $(M,\omega)$ be a closed Hermitian manifold, and $\psi$
be a smooth function.
Suppose  \eqref{elliptic}, \eqref{concave}, \eqref{nondegenerate},  \eqref{addistruc} and \eqref{existenceofsubsolution2} hold. Then
for any \textit{admissible}  solution  $u\in C^4(M)$  of equation \eqref{mainequ} with $\sup_M u=0$,  we have
 $C^{2,\alpha}$-estimates 
$|u|_{C^{2,\alpha}(M)}\leq C$  for $0<\alpha<1$ and a uniformly positive constant $C$. 
\end{proposition}


Applying Proposition  \ref{C2alpha-Thm1} 
one has
 \begin{proposition}
\label{sol1}
Let $(M,\omega)$ be a closed Hermitian manifold of complex dimension, $\phi\in C^\infty(M)$, 
 $2\leq k\leq n$. 
Then there is a unique real valued smooth function 
$u\in C^\infty(M)$ with $\lambda(\mathfrak{g}[u]) \in \Gamma_k$ and $\sup_Mu=0$,
 and a unique constant $b$ such that 
\begin{equation}
\begin{aligned}
(\mathfrak{g}[u])^k\wedge \omega^{n-k}=e^{\phi+b} \omega^n, \mbox{ in } M. \nonumber
\end{aligned}
\end{equation}
\end{proposition}

\begin{remark} \label{application-equation1}  
 Proposition \ref{sol1} for complex Monge-Amp\`ere equation   
 was   by
Tosatti-Weinkove  \cite{TW123}  for higher dimensions independently, while for the case of $n=2$ it was proved  by  the author in earlier work \cite{yuan2018CJM}.
   
\end{remark}

  \section{Construction of subsolutions on products}
  \label{constructionofsubsolutions}
 
  The construction of subsolutions is a key ingredient to solve  the Dirichlet problem 
according to the above theorems. 

\vspace{1mm}  If one could construct a function $h$ with satisfying $h|_{\partial M}=0$,
 $\mathfrak{g}[h]-\tilde{\chi}\geq 0$ and $\mathrm{rank}(\mathfrak{g}[h]-\tilde{\chi})\geq 1$,
  which is extremely hard to construct in general, then one may use it to construct subsolutions. 
 
\vspace{1mm} 
 In this  section we construct the subsolutions by considering the special case that $\mathrm{rank}(\mathfrak{g}[h]-\tilde{\chi})=1$.
 The key ingredient  is that 
 if $\mathrm{rank}(\mathfrak{g}[h]-\tilde{\chi})=1$ then it corresponds to second order linear elliptic operator.
 This leads naturally to an interesting case that the background space is a product $(M,J,\omega)=(X\times S,J,\omega)$.  
 It is noteworthy that $J$ is the standard induced complex structure, 
 while $\omega$ is not needed to be $\omega=\pi_1^*\omega_X+\pi_2^*\omega_S$.
 On such products,\renewcommand{\thefootnote}{\fnsymbol{footnote}}\footnote{The construction of subsolutions applies to certain fibered manifolds.}
 we can construct  strictly \textit{admissible} subsolutions  with 
  $\frac{\partial}{\partial \nu} \underline{u}|_{\partial M}< 0$ for Dirichlet problem \eqref{mainequ} and \eqref{mainequ1},  
if    $\eta^{1,0}=\pi_2^*\eta^{1,0}_S$ and   \eqref{cone-condition1} holds.

 \vspace{1mm} 
  The subsolution is precisely given by 
    \begin{equation}   \begin{aligned}
   \underline{u}=w+N\pi_2^* h
  \end{aligned}\end{equation}
for  $N\gg1$
   ($\pi_2^* h=h\circ\pi_2$, still denoted by $h$ 
    for simplicity), where $h$ is the solution to 
 \begin{equation}
   \label{possion-def}
   \sqrt{-1}\partial_S\overline{\partial}_S h+\sqrt{-1}\partial_S h\wedge \overline{\eta_{S}^{1,0}} 
   +\sqrt{-1} \eta_{S}^{1,0}\wedge \overline{\partial}_S h
   =\omega_S \mbox{ in } S, \mbox{  } h=0 \mbox{ on } \partial S.
   \end{equation}
   In the local coordinate $z_n=x_n+\sqrt{-1}y_n$ of $S$, Dirichlet problem \eqref{possion-def} reduces to 
  $$\Delta_S h+ \frac{\overline{\eta}_{S; z_n,\bar z_n}}{g_{S; z_n,\bar z_n}} \frac{\partial h}{\partial z_n}+
 \frac{\eta_{S; z_n,\bar z_n} }{g_{S; z_n,\bar z_n}}\frac{\partial h}{\partial \bar z_n}=1\mbox{ in } S, \mbox{  } h=0 \mbox{ on } \partial S,$$
 where  $\omega_S=\sqrt{-1}g_{S; z_n,\bar z_n}dz_n\wedge d\bar{z}_n$, $\eta^{1,0}_S=\eta_{S; z_n,\bar z_n}dz_n$, and 
 $\Delta_S=\frac{1}{g_{S; z_n,\bar z_n}}\frac{\partial^2 }{\partial z_n\partial \bar z_n}$ is the complex Laplacian operator with respect to 
 $\omega_S$.
   In particular, if $\eta_S^{1,0}=0$ then 
    \begin{equation}  
    \label{possion-def-local}
    \begin{aligned}
   \Delta_S h =1 \mbox{ in } S, \mbox{  } h=0 \mbox{ on } \partial S.  
   \end{aligned}
  \end{equation}
  
     According to the theory of elliptic equations,  
       if $\partial S\in C^{2,\beta}$,
      Dirichlet problem \eqref{possion-def}   is uniquely solvable in 
   the class of $C^{2,\beta}$ functions, and the maximum principle further implies $h|_S<0$, $\frac{\partial}{\partial \nu}h|_{\partial S}<0$. 
   Thus  $\mathfrak{g}[\underline{u}]=\mathfrak{g}[w]+N\pi_2^*\omega_S, \mbox{ and }
    \frac{\partial}{\partial \nu}\underline{u}|_{\partial M}<0 \mbox{ if } N\gg1.$
    Such a condition  $\frac{\partial}{\partial \nu} \underline{u}|_{\partial M}< 0$ is important for 
Theorems  \ref{dege-thm-c2alpha}, \ref{thm4-diri-de-c2alpha} and \ref{mainthm-09}.


\vspace{1mm}
 In the special case if  $\varphi\in C^2(\partial S)$, $\eta^{1,0}=\pi_2^*\eta^{1,0}_S$ and $\lambda (\tilde{\chi}+ c\pi_2^*\omega_S)\in \Gamma$ for some $c>0$, then for each function
$w\in C^2(\bar S)$ with $w|_{\partial S}=\varphi$, $w$ always lies in $\mathfrak{F}(\varphi)$, and moreover    \eqref{cone-condition1}
can be derived from
  \begin{equation}
  \label{cone-condition3}
  \begin{aligned}
 \lim_{t\rightarrow +\infty} f(\lambda(\tilde{\chi}+ t\pi_2^*\omega_S))>\psi  \mbox{ in } \bar M.
 \end{aligned}
\end{equation}
From Theorem \ref{solvability1-equiv} below, such a condition guarantees the solvability in this special case
 when $\varphi\in C^2(\partial S)$.
Significantly, if $\omega=\pi_1^*\omega_X+\pi_2^*\omega_S$,
     $\tilde{\chi}$ splits into 
     $\tilde{\chi}=\pi_1^*\tilde{\chi}_1+ \pi_2^*\tilde{\chi}_2,$
 where $\omega_X$ is the K\"ahler form on $X$, $\tilde{\chi}_1$ is a real $(1,1)$-form on $X$, $\tilde{\chi}_2$  is a real $(1,1)$-form on $S$,
     by 
     Lemma \ref{yuan's-quantitative-lemma}, condition 
     \eqref{cone-condition3} then reduces to 
      \begin{equation}
  \label{cone-condition1-1}
  \begin{aligned}
 \lim_{t\rightarrow +\infty} f(\lambda_{\omega_{X}}(\tilde{\chi}_1), t)>\psi \mbox{ in } \bar M,  
  \mbox{ and  } \lambda_{\omega_X}(\tilde{\chi}_1)\in \Gamma_\infty,  
  \end{aligned}
\end{equation}
  where $\lambda_{\omega_{X}}(\tilde{\chi}_1)$ are the eigenvalues of $\tilde{\chi}_1$ 
  with respect to $\omega_X$, as in Trudinger \cite{Trudinger95}
   $\Gamma_\infty=\{(\lambda_1,\cdots,\lambda_{n-1}) \in \mathbb{R}^{n-1}: (\lambda_1,\cdots,\lambda_{n-1},c)\in \Gamma \mbox{ for some } c>0\}$ denotes the  projection of $\Gamma$ onto $\mathbb{R}^{n-1}$. 
   This shows that,  if $\tilde{\chi}=\pi_1^*\tilde{\chi}_1+ \pi_2^*\tilde{\chi}_2,$ the solvability of Dirichlet problem 
 is then heavily determined by $\tilde{\chi}_1$ rather than by $\tilde{\chi}_2$.  




\section{The Dirichlet problem  with less regular boundary and boundary data} 
\label{Dirichlet-problem-less-rugularity}

For purpose of investigating the equations on complex manifolds with  less regular boundary,\renewcommand{\thefootnote}{\fnsymbol{footnote}}\footnote{We emphasize that   the geometric quantities of $(M,\omega)$ (curvature $R_{i\bar j k\bar l}$ and the torsion $T^k_{ij}$) keep bounded as approximating to $\partial M$, and all derivatives of $\tilde{\chi}_{i\bar j}$ and $\eta^{1,0}$ have continues extensions to $\bar M$,
 whenever $M$ has less regularity boundary. Typical examples are as follows: $M\subset M'$, $\mathrm{dim}_{\mathbb{C}}M'=n$, 
$\omega=\omega_{M'}|_{M}$ and the given data $\tilde{\chi}, \eta^{1,0}$ can be smoothly defined on $M'$.} we need to seek 
those manifolds which allow us to apply Theorem \ref{thm1-diri-estimate} 
to do the submanifold/domain approximation. 
Clearly, there are two types of complex manifolds that fulfill the request. That is 
$M=X\times S$ 
 (or more generally $M$ is as mentioned in Theorem \ref{dege-thm-c2alpha} below),
 or the manifold whose boundary 
 satisfies 
 \begin{equation}
 \label{bdry-assumption1-strictly}
 \begin{aligned}
 \lambda_{\omega'}(-{L}_{\partial M})\in \Gamma_{\infty}
 \end{aligned}
 \end{equation}

   
\begin{theorem}
\label{opti-regularity-thm1}
 
Suppose  $M=X\times S$ 
 (or more generally $M$ is as mentioned in Theorem \ref{dege-thm-c2alpha} below but with  $\partial M_i$ converge to $\partial M$ in
  $C^{2,1}$-norm),
   or $\partial M$ satisfies \eqref{bdry-assumption1-strictly}. 
Suppose in addition that $\partial M\in C^3$, $\varphi\in C^3(\partial M)$, $\psi\in   C^{2}(\bar M)$ 
and the other assumptions of Theorem \ref{thm1-diri-estimate} hold. Then
  Dirichlet problem  \eqref{mainequ} and \eqref{mainequ1}  
   admits a unique  admissible solution $u\in C^{2,\alpha}(\bar M)$ for some $0<\alpha<1$. Moreover, 
    $u$ satisfies \eqref{c2alpha-estimate1}.
   \end{theorem}
   
   \begin{theorem}
   \label{opti-regularity-thm1-degenerate}
   Let $M$ be as in Theorem \ref{opti-regularity-thm1} but we assume $\partial M \in C^{2,1}$.
 Suppose in addition that the other assumptions of Theorem  \ref{thm3-diri-estimate-de}  hold. Then
     the Dirichlet problem has a  (weak)
   solution $u\in C^{1,\alpha}$  $(\forall 0<\alpha<1)$ with $\lambda(\mathfrak{g}[u])\in \bar\Gamma$ 
   and $\Delta u\in L^\infty(\bar M)$.
\end{theorem}

   
A somewhat remarkable fact to us is that, 
  the regularity assumptions on boundary and boundary data
   can be further weakened under certain assumptions, which extends extensively a result of \cite{yuan2017} to general settings. 
  The motivation 
  is mainly based on the estimates which state that 
   \begin{itemize}
  \item If $\partial M$ is \textit{holomorphically flat}  and  boundary value is furthermore a constant, 
  then the constant in quantitative boundary estimate \eqref{bdy-sec-estimate-quar1} 
  depends only on $\partial M$ up to second derivatives and other known data (see  Theorem \ref{mix-Leviflat-thm1}). 
   \item If we furthermore assume $M=X\times S$ and the boundary data $\varphi\in C^2(\partial S)$,
     then 
     the constant in  \eqref{bdy-sec-estimate-quar1} is bounded from above 
     by a positive constant depending on
    $|\varphi|_{C^{2}(\bar S)}$, $|\psi|_{C^{1}(\bar M)}$, 
$|\underline{u}|_{C^{2}(\bar M)}$, $\partial S$
up to second derivatives and other known data (see Theorem \ref{mix-Leviflat-thm1-product}).
   \end{itemize}
Besides, we can use a result due to Silvestre-Sirakov \cite{Silvestre2014Sirakov} to derive
the $C^{2,\alpha}$ boundary regularity with only assuming $C^{2,\beta}$ boundary.

\begin{theorem}
\label{dege-thm-c2alpha}
  
   Let $(M,J,\omega)$ be a compact Hermitian manifold  with $C^{2,\beta}$  
 boundary which additionally satisfies that,
  for any sequence of smooth complex submanifolds $\{M_i\}$ (with the same complex dimension)
 whose boundaries $\partial M_i$ converge to $\partial M$ in the $C^{2,\beta}$ norm, there is $N\gg1$ 
 such that $\partial M_i$ is \textit{holomorphically flat} for any $i\geq N$. 
 Suppose in addition that \eqref{elliptic}, \eqref{concave}, \eqref{nondegenerate}, \eqref{addistruc} and
  $\psi\in C^{2}(\bar M)$.
Then  Dirichlet problem  \eqref{mainequ} and \eqref{mainequ1}
  with homogeneous boundary data $u|_{\partial M}= 0$
  supposes a unique $C^{2,\alpha}$ \textit{admissible} solution  for some $0<\alpha\leq\beta$, provided that 
      the Dirichlet problem has a $C^{2,\beta}$ admissible subsolution with either
       $\frac{\partial}{\partial \nu} \underline{u}|_{\partial M}\leq 0$ or $\frac{\partial}{\partial \nu} \underline{u}|_{\partial M}\geq 0$.
   
\end{theorem}

 \begin{theorem}
 \label{thm4-diri-de-c2alpha}
 Let $(M,J,\omega)$ be 
 as mentioned in Theorem \ref{dege-thm-c2alpha}.
 Suppose, in addition to  \eqref{elliptic}, \eqref{concave}, \eqref{addistruc}, that
  $f\in C^\infty(\Gamma)\cap C(\bar \Gamma)$ and
   $\psi\in C^{1,1}(\bar M)$ is a function with $\delta_{\psi,f}=0$.
Suppose that 
 the Dirichlet problem has a $C^{2,\beta}$ strictly admissible subsolution
 with $\frac{\partial}{\partial \nu} \underline{u}|_{\partial M}\leq 0$ or $\frac{\partial}{\partial \nu} \underline{u}|_{\partial M}\geq 0$.
Then the Dirichlet problem with homogeneous boundary data  
 admits  a  (weak)   solution $u\in C^{1,\alpha}(\bar M)$, $\forall 0<\alpha<1$,
 with $\lambda(\mathfrak{g}[u])\in \bar \Gamma$ and $\Delta u \in L^{\infty}(\bar M)$.
 \end{theorem}


   \begin{corollary}
\label{dege-thm-c2alpha-special}
Let $M$ be the complex manifold  as mentioned in Theorem \ref{dege-thm-c2alpha}.
Suppose, in addition to \eqref{elliptic}, \eqref{concave} and $\lambda(\tilde{\chi})\in \Gamma$, 
that $f$ is a homogeneous function of degree one with $f|_{\partial \Gamma}=0$. 
Then $f(\lambda(\mathfrak{g}[u]))=0$ with $u|_{\partial M}=0$ supposes a  weak 
solution $u\in C^{1,\alpha}(\bar M)$ $(\forall 0<\alpha<1)$,
 with  $\lambda(\mathfrak{g}[u])\in \bar \Gamma$ and $\Delta u \in L^{\infty}(\bar M)$.
  \end{corollary}

   Together with Theorem \ref{solvability1-equiv} below, 
 Theorems \ref{dege-thm-c2alpha}, \ref{thm4-diri-de-c2alpha} and  the construction of subsolutions in Section \ref{constructionofsubsolutions} immediately lead to 
  
   \begin{theorem}
  \label{mainthm-09}
   Let  $(M, J,\omega)=(X\times S,J, \omega)$
   be  as mentioned above. Let  $\partial S\in C^{2,\beta}$,   $\varphi\in C^{2,\beta}(\partial S)$, 
   $\eta^{1,0}=\pi_2^*\eta^{1,0}_S$, and  $f$ satisfy \eqref{elliptic}, \eqref{concave} and \eqref{addistruc}.
   Suppose in addition that \eqref{cone-condition1} holds for some $w\in C^{2,\beta}(\bar M)\cap\mathfrak{F}(\varphi)$. 
   Then we have the  two conclusions:
   \begin{itemize}
   \item   Equation \eqref{mainequ} has a 
   unique $C^{2,\alpha}$ admissible solution with $u|_{\partial M}=\pi_2^*\varphi=\varphi\circ\pi_2$
   (still denoted by $u|_{\partial M}=\varphi$ for convenience) for some 
   $0<\alpha\leq\beta$, provided 
 $\psi\in C^2(\bar M)$ and $\inf_{M} \psi> \sup_{\partial \Gamma}f.$
   \item Suppose furthermore
     $f\in C^\infty(\Gamma)\cap C(\bar\Gamma)$, $\psi\in C^{1,1}(\bar M)$ and  
    $\inf_{M} \psi=sup_{\partial \Gamma}f$. Then the Dirichlet problem 
    has a weak solution 
   $u\in C^{1,\alpha}(\bar M)$, $\forall 0<\alpha<1$,
 with $u|_{\partial M}=\varphi$, $\lambda(\mathfrak{g}[u])\in \bar \Gamma$ and $\Delta u \in L^{\infty}(\bar M)$.
   \end{itemize}
   \end{theorem}






\subsubsection{Proof of Theorem  \ref{dege-thm-c2alpha}}
In the proof we only need to consider the level sets of $\underline{u}$ or of approximating functions
  near the boundary.
Without loss of generality we assume $\underline{u}$ is not a constant (otherwise,  it is trivial if $\underline{u}=\mathrm{constant}$).  In what follows, the level sets what we use in the proof are all near the boundary $M_\delta$ for some
$0<\delta\ll1$, here $M_\delta$ is defined in \eqref{distances-domain}.

\vspace{1mm} 
Suppose $\frac{\partial \underline{u}}{\partial \nu}|_{\partial M}\neq 0$. 
Without loss of generality,  we assume $\frac{\partial \underline{u}}{\partial \nu}|_{\partial M}< 0$. Then
$\underline{u}<0$,
$\nabla\underline{u}\neq 0$ in 
 $M_{\delta_{0}}$ for some $\delta_0>0$. 
We first choose a sequence of smooth approximating functions  $\{\underline{u}^{(k)}\}$ 
such that $\underline{u}^{(k)}\rightarrow \underline{u}$ in $C^{2,\beta}(\bar M)$ as $k$ tends to infinity.
 Then
$\frac{\partial \underline{u}^{(k)}}{\partial \nu}|_{\partial M}\neq 0$ for $k\gg1$.
Next, for any $k\gg1$, 
 by the diagonal method and Sard's theorem if necessary,
we  carefully choose $\{\alpha_k\}$ satisfying $\alpha_{k}\rightarrow 0$ as $k\rightarrow +\infty$ and a sequence of level sets
 of $\underline{u}^{(k)}$, say $\{\underline{u}^{(k)}=-\alpha_k\}$,
 and use them to enclose  smooth
complex submanifolds with the same complex dimension, say $M^{(k)}$, such that 
 $\cup M^{(k)}=M$ and $\partial M^{(k)}$ converge to $\partial M$ in the norm of $C^{2,\beta}$. 
Moreover, we can choose
 $\{\beta_k\}$  $(\beta_k>0)$ with $ \beta_k\rightarrow 0$ as $k\rightarrow+\infty$, such that
 \begin{equation}
 \label{subsolution-newregularity}
\begin{aligned}
F(\mathfrak{g}[\underline{u}^{(k)}])\geq \psi-\beta_k \mbox{ in } M.
\end{aligned}
\end{equation}
 Furthermore,  $\underline{u}^{(k)}$ are \textit{admissible} functions for sufficiently large $k$, as $\Gamma$ is open.

\vspace{1mm} 
According to Theorem \ref{thm1-diri-estimate} we have a unique smooth \textit{admissible} function $u^{(k)}\in C^{\infty}(\overline{M^{(k)}})$ to solve
\begin{equation}
\label{solution-newregularity}
\begin{aligned}
F(\mathfrak{g}[u^{(k)}])= \psi-\beta_k \mbox{ in } M^{(k)}, \mbox{  } u^{(k)}=-\alpha_k \mbox{ on } \partial M^{(k)}.
\end{aligned}
\end{equation}
 Moreover, Theorems \ref{mix-Leviflat-thm1}
 and  \ref{globalsecond-Diri} imply
  \begin{equation}
 \label{uniform-00}
\begin{aligned}
\sup_{\overline{M^{(k)}}}\Delta u^{(k)}\leq C_k (1+\sup_{\partial M^{(k)}}|\nabla u^{(k)}|^2)(1+\sup_{M^{(k)}}|\nabla u^{(k)}|^2)
\end{aligned}
\end{equation}
holds for $C_k$ depending on $|\psi|_{C^{2}(M^{(k)})}$, $|\underline{u}^{(k)}|_{C^{2}(M^{(k)})}$, $|u^{(k)}|_{C^0(M^{(k)})}$, 
$\partial M^{(k)}$ up to second order derivatives  and other known data.

\vspace{1mm} 
 If we could prove there is a uniform constant $C$ depending not on $k$, such that
  \begin{equation}
 \label{uniform-c0-c1}
\begin{aligned}
|u^{(k)}|_{C^0(M^{(k)})}+\sup_{\partial M^{(k)}}|\nabla u^{(k)}|\leq C,
\end{aligned}
\end{equation}
together with the construction of approximating Dirichlet problems, 
 \eqref{uniform-00} then holds for a uniformly constant $C'$ which is independent of $k$. 
Thus we have $|u|_{C^2(M^{(k)})}\leq C$ depending not on $k$ (here we  use blow up argument to derive gradient estimate).  
Thus the equations are all uniformly elliptic, and as in  \cite{Guan2010Li} we can directly prove the bound of real Hessians of solutions.

\vspace{1mm} 
Finally, we are able to apply Silvestre-Sirakov's \cite{Silvestre2014Sirakov} result to 
derive $C^{2,\alpha'}$ estimates on the boundary, while
the convergence of  $\partial M^{(k)}$ in the norm $C^{2,\beta}$ 
 allows  us to take a limit ($\alpha'$ can be uniformly chosen).
 
\vspace{1mm} Next, we  prove \eqref{uniform-c0-c1}.  Let $w^{(k)}$ be the solution of
 \begin{equation}
 \label{supersolution-k}
\begin{aligned}
\Delta w^{(k)} + g^{i\bar j}w_i^{(k)} \eta_{\bar j} +g^{i\bar j}\eta_i w^{(k)}_{\bar j} +g^{i\bar j}\tilde{\chi}_{i\bar j} =0 \mbox{ in } M^{(k)}, \mbox{  } w^{(k)}=-\alpha_k \mbox{ on } \partial M^{(k)}.
\end{aligned}
\end{equation}
By  maximum principle and the boundary value condition $\underline{u}^{(k)}= u^{(k)}=w^{(k)} =-\alpha_k \mbox{ on } \partial M^{(k)}$ we have 
\begin{equation}
\label{approxi-boundary1}
\begin{aligned}
\underline{u}^{(k)}\leq u^{(k)}\leq w^{(k)} \mbox{ in } M^{(k)}, \mbox{ and } 
 \frac{\partial \underline{u}^{(k)}}{\partial \nu} \leq  \frac{\partial u^{(k)}}{\partial \nu} \leq \frac{\partial w^{(k)}}{\partial \nu}   \mbox{ on } \partial M^{(k)}.
\end{aligned}
\end{equation}
By the regularity theory of elliptic equations we have
\begin{equation}
\label{key-proof-newregularity1}
\begin{aligned}
\sup_{M^{(k)}} w^{(k)} + \sup_{\partial M^{(k)}} \frac{\partial w^{(k)}}{\partial \nu}  \leq C.
\end{aligned}
\end{equation}
We thus complete the proof of \eqref{uniform-c0-c1}, and then
obtain $C^{2,\alpha}$-\textit{admissible} solution of 
 \begin{equation}
 \label{approx-equ-homogeneous1}
\begin{aligned}
F(\mathfrak{g}[u])= \psi \mbox{ in } M, 
\mbox{  } u=0 \mbox{ on }  \partial M.
\end{aligned}
\end{equation}

If $\frac{\partial}{\partial \nu}\underline{u}|_{\partial M}\geq 0$ or 
$\frac{\partial}{\partial \nu}\underline{u}|_{\partial M}\leq 0$,  
we can apply  
 a $C^{2,\beta}$ function extended by the distance function to boundary to 
 make perturbation for $\underline{u}$ to construct 
 a sequence of approximating problems on $M$ with homogeneous boundary data.
We then  continue the perturbation process as previous case.

 \begin{remark}
The proof of Theorem   \ref{dege-thm-c2alpha}  is based on approximation method. The approximating Dirichlet problems are constructed as follows:  one first approximates $\underline{u}$ by smooth functions in the norm of $C^{2,\beta}$, and then uses the level sets of smooth approximating functions to enclose submanifolds/domains. 
\end{remark}
 
\subsubsection{The significant phenomena on \text{$M=X\times S$}}
Firstly, we have the observation:
\begin{theorem}
\label{solvability1-equiv}
Let $(M,J,\omega)=(X\times S,J,\omega)$, $\eta^{1,0}=\pi_2^*\eta^{1,0}_S$ be as in Theorem  \ref{mainthm-09},
 we assume $\tilde{u}$ 
 is the solution to $F(\mathfrak{g}[\tilde{u}])=\psi$ with boundary data $\tilde{u}|_{\partial M}=\tilde{\varphi}.$
Suppose $\hat{\varphi}=\tilde{\varphi}+\varphi$ for some $\varphi\in C^{2,\beta}(\partial S)$. Then  
 $\hat{u}=\tilde{u}+v$ coincides with the  solution of 
$$F(\mathfrak{g}[\hat{u}])=\psi \mbox{ in } M, \mbox{  } \hat{u}=\hat{\varphi} \mbox{ on } \partial M,$$
where $v$ solves
 \begin{equation}
   \label{harmonic-varphi}
    \begin{aligned}
 \sqrt{-1}\partial_S\overline{\partial}_S v+\sqrt{-1}\partial_S v\wedge \overline{\eta_{S}^{1,0}}  +\sqrt{-1} \eta_{S}^{1,0}\wedge \overline{\partial}_S v=0 
 \mbox{ in } S, \mbox{  } v=\varphi \mbox{ on } \partial S.
   \end{aligned} 
   \end{equation}
\end{theorem}

\begin{proof}
$\mathfrak{g}[u]=\mathfrak{g}[\tilde{u}]+\pi_2^*(\sqrt{-1}\partial_S\overline{\partial}_S v+\sqrt{-1}\partial_S v\wedge \overline{\eta_{S}^{1,0}}  +\sqrt{-1} \eta_{S}^{1,0}\wedge \overline{\partial}_S v)=\mathfrak{g}[\tilde{u}]. $
\end{proof}
\begin{corollary}
Let $\tilde{u}$, $\hat{u}$ are  respectively admissible solutions of 
\begin{equation}
 \begin{aligned}
  F(\mathfrak{g}[\tilde{u}])=\psi \mbox{ in } M, \mbox{  } \tilde{u}=\tilde{\varphi} \mbox{ on } \partial M, \mbox{  } \nonumber
   F(\mathfrak{g}[\hat{u}])=\psi \mbox{ in } M, \mbox{  } \hat{u}=\hat{\varphi} \mbox{ on } \partial M, 
 \end{aligned} 
 \end{equation}
 and $\hat{\varphi}-\tilde{\varphi}=\varphi$ for some $\varphi\in C^{2,\beta}(\partial S)$. Then $\hat{u}-\tilde{u}=v$, where $v$ solves \eqref{harmonic-varphi}.
\end{corollary}

\vspace{1mm}
\begin{proof}
[Proof of Theorem \ref{mainthm-09}]
Here we give two proofs. 

\vspace{1mm}
\noindent{\bf First proof}: 
By Theorem \ref{solvability1-equiv} we only need to consider the special case when the boundry data $\varphi=0$.
  As above Dirichlet problem \eqref{approx-equ-homogeneous1}
  admits \textit{admissible} subsolutions $\underline{u}=Nh$ of  for large $N$, where $h$ is the solution to \eqref{possion-def}.
  Theorems \ref{dege-thm-c2alpha} and \ref{solvability1-equiv}, together with  the construction of subsolution, immediately yield
Theorem \ref{mainthm-09}. 

\vspace{1mm} 
Besides,  
we construct the approximating Dirichlet problems precisely as in the following, and 
 also give another constructive proof 
 by only using $h$ and maximum principle but without using the regularity theory of elliptic equations.

   \vspace{1mm} 
Note that $h$ is a function on $S$ with $\frac{\partial h}{\partial\nu}|_{\partial S}<0$. 
Similarly, there is a sequence of smooth approximating functions $\{h^{(k)}\}$ on $S$ such that 
 $h^{(k)}\rightarrow h$ in $C^{2,\beta}(\bar S)$,  
 and 
 \begin{equation}
 \begin{aligned}
  \frac{1}{2}\omega_S\leq \sqrt{-1}\partial_S\overline{\partial}_S h^{(k)}+\sqrt{-1}\partial_S h^{(k)}\wedge \overline{\eta_{S}^{1,0}} 
   +\sqrt{-1} \eta_{S}^{1,0}\wedge \overline{\partial}_S h^{(k)}
   \mbox{ in } S,  \nonumber
   \end{aligned}
   \end{equation}
 (here we use \eqref{possion-def}),
 moreover, we also get a sequence of level sets of $h^{(k)}$, say $\{h^{(k)}=-\frac{\alpha_k}{N}\}$
 and use it to enclose a smooth
complex submanifold of complex dimension one, say $S^{(k)}$, such that 
 $\cup S^{(k)}=S$ and $\partial S^{(k)}$ converge to $\partial S$ in the norm of $C^{2,\beta}$. 
  
 \vspace{1mm} 
  Let us denote $M^{(k)}=X\times S^{(k)}$, and
 $\underline{u}^{(k)}= Nh^{(k)}$. 
  Then $\underline{u}^{(k)}$ satisfies \eqref{subsolution-newregularity}, $\underline{u}^{(k)}|_{\partial M^{(k)}}=  -\alpha_k$ and
 \begin{equation}
 \label{supersolution-k-product}
\begin{aligned}
 \frac{N}{2}\mathrm{tr}_\omega (\pi_2^*\omega_S)\leq \Delta \underline{u}^{(k)} + g^{i\bar j}(\underline{u}^{(k)})_i \eta_{\bar j} +g^{i\bar j}\eta_i (\underline{u}^{(k)})_{\bar j}  
 \mbox{ in } M^{(k)}.  
\end{aligned}
\end{equation}
In other words, Dirichlet problem \eqref{solution-newregularity} has an \textit{admissible} subsolution $\underline{u}^{(k)}$.
To derive \eqref{supersolution-k-product} we use the fact that $J$ is the induced complex structure.

\vspace{1mm} 
Let $u^{(k)}$ be the solution to  \eqref{solution-newregularity}.
We only need to prove \eqref{key-proof-newregularity1}.  
Let $N\geq  \kappa_1^{-1}\sup_M \mathrm{tr}_\omega \tilde{\chi}$, 
where $\kappa_1=\inf_M  \frac{1}{2}\mathrm{tr}_\omega (\pi_2^*\omega_S)$.  
Applying comparison principle to \eqref{supersolution-k} and \eqref{supersolution-k-product}, 
$$w^{(k)}\leq -\underline{u}^{(k)} -2\alpha_k  \mbox{ in } M^{(k)}, \mbox{ and }  \frac{\partial w^{(k)}}{\partial \nu} \leq -\frac{\partial \underline{u}^{(k)}}{\partial \nu} \mbox{ on } \partial M^{(k)},$$
which is exactly \eqref{key-proof-newregularity1}. Here we don't use the regularity theory of elliptic equations.
Therefore, we get the $C^{2,\alpha}$ \textit{admissible} solution $u$ to \eqref{approx-equ-homogeneous1}.




\vspace{1mm}
\noindent{\bf Second proof}: The second proof is based on Theorem \ref{mix-Leviflat-thm1-product}, and so the boundary data of approximating Dirichlet problems need not be constant. It is straightforward by using standard approximation.

\vspace{1mm} 
Let $v$ be the solution to \eqref{harmonic-varphi}.
The subsolution is given by  $\underline{u}=v+Nh$ for $N\gg1$.
The conditions of $\eta^{1,0}=\pi_2^*\eta^{1,0}_S$  and $\lambda(\tilde{\chi}+c\pi_2^*\omega_S)\in \Gamma$ for some $c>0$ yield $\underline{u}$ is \textit{admissible} for large $N$.

\vspace{1mm} 
Let $S^{(k)}\subset S$ be smooth domains which approximate $S$ in the $C^{2,\beta}$ norm 
($\partial S^{(k)}\rightarrow \partial S$ in $C^{2,\beta}$ norm). We also approximate $\underline{u}$ by smooth functions $\underline{u}^{(k)}$ and then obtain a sequence of approximating Dirichlet problems
\begin{equation}
\begin{aligned}
F(\mathfrak{g}[u^{(k)}])= \psi-\beta_k \mbox{ in } X\times S^{(k)}, 
\mbox{  } u^{(k)}=\underline{u}^{(k)} \mbox{ on } X\times \partial S^{(k)},
\end{aligned}
\end{equation}
with \textit{admissible} subsolution $\underline{u}^{(k)}$ $(\beta_k\rightarrow 0, \beta_k>0)$. Thus there is a unique  smoothly \textit{admissible} solution $u^{(k)}\in C^{\infty}(X\times\bar S^{(k)})$ for each $k\gg1$. 

\vspace{1mm} 
It suffices to prove \eqref{uniform-c0-c1}.  Let $w^{(k)}$ be the solution to
 \begin{equation}
\begin{aligned}
\Delta w^{(k)} + g^{i\bar j}w_i^{(k)} \eta_{\bar j} +g^{i\bar j}\eta_i w^{(k)}_{\bar j} +g^{i\bar j}\tilde{\chi}_{i\bar j} =0 \mbox{ in } X\times S^{(k)}, \mbox{  } w^{(k)}=\underline{u}^{(k)}  \mbox{ on } X\times\partial S^{(k)}. \nonumber
\end{aligned}
\end{equation}
Thus one has \eqref{approxi-boundary1} and \eqref{key-proof-newregularity1} by using standard regularity theory of elliptic equations.

\end{proof}

\begin{proof}
[Proof of Theorem \ref{mainthm-10-degenerate}]
We prove it by using standard method of approximation. 
Let's choose a sequence of smooth Riemann surfaces/domains $S_k\subset S$ such that $\cup_k S_k=S$, and $\partial S_k\rightarrow \partial S$ in $C^{2,1}$-norm as $k\rightarrow +\infty$. 
Also, we denote $M_k=X\times S_k$. 

\vspace{1mm} 
Let's consider a family of approximating problems
 \begin{equation}
 \label{appro-problems3}
\begin{aligned}
F(\mathfrak{g}[u^{(k)}])=\psi \mbox{ in } M_k, \mbox{   } u^{(k)}|_{\partial M_k}=w|_{\partial M_k}.
\end{aligned}
\end{equation}
The condition \eqref{cone-condition1} allows us applying the solution of 
$\Delta_S h^{(k)}=1\mbox{ in } S_k, \mbox{   } h^{(k)}|_{\partial S_k}=0$
 to construct  admissible subsolutions 
$\underline{u}^{(k)}$ to \eqref{appro-problems3}. Then  \eqref{appro-problems3} has a unique admissible solutions  $u^{(k)}$
according to Theorem 
   \ref{opti-regularity-thm1}. 
Similarly,  the standard regularity theory of elliptic equations gives  \eqref{approxi-boundary1} and \eqref{key-proof-newregularity1}.
This completes the proof.
\end{proof}

\subsubsection{Discussion on significant phenomena  on weakening  regularity assumptions} 

When $M=X\times S$  
or $\partial M$ obeys \eqref{bdry-assumption1-strictly}, 
according to Theorem \ref{opti-regularity-thm1-degenerate},
it is only required to assume  $\varphi\in C^{2,1}$ and $\partial M\in C^{2,1}$ to solve Dirichlet problem of degenerate equations. 
Such regularity assumptions on the boundary and boundary data
are impossible for homogeneous Monge-Amp\`ere equation on certain
bounded domains in $\mathbb{R}^n$ as shown by some counterexamples in Caffarelli-Nirenberg-Spruck \cite{CNS-deg} 
who show that the $C^{3,1}$-regularity assumptions on the boundary and boundary data
are optimal for the $C^{1,1}$ global regularity of the weak 
solution to homogeneous real Monge-Amp\`ere equation on $\Omega$.
(For a strictly convex domain and the homogeneous boundary data,
 Guan \cite{GuanP1997Duke} proved that $\partial \Omega$ can be weakened to be $C^{2,1}$).   
 Theorem \ref{mainthm-10-degenerate} gives specific examples to support this new feature.
 We find another significant  phenomenon on weakening regularity assumptions on boundary
  and boundary data. More precisely,
  when 
 $M=X\times S$ and the boundary data only varies along $\partial S$ (which is independent of $X$),
  the regularity of $\partial M$ and $\varphi$ can be weakened to be $C^{2,\beta}$; while
  such $C^{2,\beta}$ regularity assumptions on the boundary and boundary data are impossible even for Dirichlet problem of
  certain nondegenerate real Monge-Amp\`ere equation on certain bounded domains $\Omega\subset \mathbb{R}^2$,
as shown by Wang \cite{WangXujia-1996}   
 the optimal regularity assumptions on
the boundary and boundary data  are both $C^3$-smooth for such real Monge-Amp\`ere equations.
Theorem \ref{mainthm-09} and Corollary \ref{dege-thm-c2alpha-special} show specific 
examples to support the surprising phenomenon.  
We also have new phenomenon on weakening regularity assumption on boundary in Theorems \ref{dege-thm-c2alpha} and \ref{thm4-diri-de-c2alpha}  
where the boundary data is constant. 

 \vspace{1mm} 
 As a contrast, for (real) Monge-Amp\`ere equation with less regular right-hand side, 
we are referred to the book \cite{Figalli2017MA}   and references therein  including  \cite{Caffarelli1990Ann,Caffarelli1990Ann-W,Caffarelli1991CPAM,De-2013-Figalli,Savin2013,Wang1995counterexample}.

\section{The Dirichlet problem of  Monge-Amp\`ere equation for $(n-1)$-PSH functions and extensions}
\label{appendix-gauduchon}

 Gauduchon \cite{Gauduchon77} proved  every closed Hermitian manifold $(M,\omega)$ of complex dimension $n\geq 2$
  admits a unique (up to rescaling)  Gauduchon metric
 ($\omega$ is called Gauduchon if $\partial\overline{\partial}(\omega^{n-1})=0$)
 which is conformal to the original Hermitian metric, and furthermore
 conjectured in \cite{Gauduchon84} that the Calabi-Yau theorem 
 for Gauduchon metric is also true.
 On  complex surfaces and astheno-K\"ahler manifolds (i.e. $\partial\overline{\partial}(\omega^{n-2})=0$)  introduced in  \cite{Jost1993Yau},
 the Gauduchon conjecture was proved by  Cherrier \cite{Cherrier} and Tosatti-Weinkove \cite{Tosatti2013Weinkove},  respectively.
   In order to look for Gauduchon metric, say $\Omega_u$, on  general Hermitian manifolds, with prescribed volume form
 $$\Omega_u^{n}= e^\phi\omega^n $$
with  $\Omega_u^{n-1}=\omega_0^{n-1}+\sqrt{-1}\partial \overline{\partial}u \wedge \omega^{n-2} + \mathfrak{Re}(\sqrt{-1}\partial u \wedge\overline{\partial}\omega^{n-2})$ (i.e. \eqref{Omega_u},
and if $\omega_0$ is  Gauduchon  then so is $\Omega_u$), by using
\begin{equation}
\label{trans-lambda}
\begin{aligned}
 \lambda_i \left(*\frac{\Theta^{n-1}}{(n-1)!}\right)= \lambda_1(\Theta)\cdots \lambda_{i-1}(\Theta) \widehat{\lambda_i(\Theta)}\lambda_{i+1}(\Theta)\cdots \lambda_n(\Theta)
\end{aligned}
\end{equation}
 for each $i$ and  real $(1,1)$-form $\Theta$, where $\widehat{*}$ indicates deletion,
it can be reduced to solve Monge-Amp\`ere equation   for $(n-1)$-PSH  functions  
  in the sense of Harvey-Lawson \cite{Harvey2012Lawson}:
  \begin{equation}
 \label{gamma-cone}
\begin{aligned}
*\left(\frac{1}{(n-1)!}\Omega_u^{n-1}\right)=\tilde{\chi}+\frac{1}{n-1}((\Delta u) \omega-\sqrt{-1}\partial \overline{\partial}u)+Z >0 \mbox{ in } \bar M,
\end{aligned}
\end{equation}
where $\tilde{\chi}_{i\bar j}=\left(\frac{1}{(n-1)!}*( \omega_0^{n-1}) )\right)_{i\bar j}$, $Z_{i\bar j} 
=\left(\frac{1}{(n-1)!}*\mathfrak{Re}(\sqrt{-1}\partial u\wedge \bar\partial (\omega^{n-2}))\right)_{i\bar j}$.
  That is 
\begin{equation}
\label{mainequ-gauduchon}
\begin{aligned}
\log P_{n-1}(\lambda(\mathfrak{g}[u]))=\psi \mbox{ in }M,
\end{aligned}
\end{equation}
for $\psi=(n-1)\phi+n\log (n-1)$, where  
$\log P_{n-1}(\lambda) = \sum_{i=1}^n \log(\lambda_{{1}}+\cdots+\lambda_{i-1}+\widehat{\lambda_i}+ \lambda_{{i+1}}+\cdots+\lambda_n)$
 corresponding to 
 $\mathcal{P}_{n-1}=\{\lambda\in \mathbb{R}^n: \lambda_{{1}}+\cdots+\lambda_{i-1}+\widehat{\lambda_i}+ \lambda_{{i+1}}+\cdots+\lambda_n>0, \mbox{  } \forall 1\leq i\leq n\},$
 $\mathfrak{g}_{i\bar j}=u_{i\bar j}+\chi_{i\bar j}+W_{i\bar j}$,
here $\chi_{i\bar j}=(\mathrm{tr}_\omega \tilde{\chi})g_{i\bar j}-(n-1)\tilde{\chi}_{i\bar j}$,
$W_{i\bar j}=(\mathrm{tr}_\omega Z)g_{i\bar j}-(n-1)Z_{i\bar j}$ (see also  \cite{Popovici2015,Tosatti2013Weinkove}).
Locally,
\begin{equation}
\label{Z-tensor1}
\begin{aligned}
Z_{i\bar j}=\,& \frac{1}{2(n-1)} \left( g^{p\bar q}   \bar T^l_{ql} g_{i\bar j}u_p
+g^{p\bar q} T^{k}_{pk} g_{i\bar j}u_{\bar q}
-g^{k\bar l}g_{i\bar q} \bar T^q_{lj}u_k
 -g^{k\bar l}g_{q\bar j}T^{q}_{ki}u_{\bar l}
-\bar T^{l}_{jl}u_i -T^{k}_{ik}u_{\bar j}  \right), 
\end{aligned}
\end{equation}
see also \cite{GTW15}.
One can verify that condition \eqref{gamma-cone} is equivalent to 
$\lambda(\mathfrak{g}[u])\in \mathcal{P}_{n-1} \mbox{ in } \bar M,$
which allows one to seek the solutions of  \eqref{mainequ-gauduchon} or equivalently  \eqref{MA-n-1}
 within the framework of elliptic equations, since
 $f=\log P_{n-1}$ satisfies \eqref{elliptic}, \eqref{concave} and  \eqref{addistruc} in  $\mathcal{P}_{n-1}$. 

\vspace{1mm} 
 In \cite{Tosatti2013Weinkove} Tosatti-Weinkove 
 further showed that the Gauduchon conjecture for general case can be reduced to show the second order estimate of form \eqref{sec-estimate-quar1} 
 for the solution $u$ to \eqref{mainequ-gauduchon} with $\sup_M u=0$ and
$\lambda(\mathfrak{g}[u])\in \mathcal{P}_{n-1}$.

\vspace{1mm} 
However, it is much more complicated for general cases, 
 since the equation \eqref{mainequ-gauduchon}  (or equivalently \eqref{MA-n-1}) involves gradient terms.
 Recently, by carefully  dealing with the special structures of $W(\partial u,\overline{\partial} u)$ and $\log P_{n-1}$,
Sz\'ekelyhidi-Tosatti-Weinkove  \cite{GTW15}
derived \eqref{sec-estimate-quar1} 
(see Theorem \ref{GTW-second} below) 
and then
completely proved Gauduchon's conjecture.
See also \cite{guan-nie} for related work.
Also, we are referred to
\cite{ChuHuangZhu2018,ChuHuangZhu2019,Phong-Picard-Zhang2017,Phong-Picard-Zhang2018arxiv,GTW15} for the progress on Fu-Yau equation \cite{Fu2007Yau,Fu2008Yau}
and Form-type Calabi-Yau equation \cite{FuWangWuFormtype2010,FuWangWuFormtype2015},
and to \cite{GQY2018} for 
 $\mathfrak{g}[u]=(\Delta u) \omega-\beta_1\sqrt{-1}\partial \overline{\partial}u+\chi(z,u,\partial u,\overline{\partial}u)$ with
  $\beta_1<1$.

\begin{theorem}
 \label{GTW-second}
 Let $\psi\in C^2(M)$, and we assume $u\in C^4(M)$ is the solution to \eqref{mainequ-gauduchon} with $\sup_M u=0$ and
   $\lambda(\mathfrak{g}[u])\in \mathcal{P}_{n-1}$,
then one has the second order estimate \eqref{sec-estimate-quar1}.

\end{theorem}

 This section is mainly devoted to investigating  Dirichlet problem of  equation \eqref{MA-n-1} 
 \begin{equation}
\label{mainequ-gauduchon1}
\begin{aligned}
\log P_{n-1}(\lambda(\mathfrak{g}[u]))=\psi \mbox{ in } M,  \mbox{   }
u=\varphi \mbox{ on }   \partial M
\end{aligned}
\end{equation}
and more general equations \eqref{mainequ-gauduchon-general*}
\begin{equation}
\label{mainequ-gauduchon-general**}
\begin{aligned}
 f(\lambda(*\Phi[u]))=\psi \mbox{ in } M,     \mbox{   }
u=\varphi \mbox{ on }   \partial M 
\end{aligned}
\end{equation}
on compact Hermitian manifolds with boundary satisfying \eqref{bdry-assumption1-Gauduchon}. 
 As stated in Introduction,  $\Phi[u] =*\chi+\frac{1}{(n-2)!}\sqrt{-1}\partial\overline{\partial}u\wedge\omega^{n-2}+\frac{\varrho}{(n-1)!}\mathfrak{Re}(\sqrt{-1}\partial u\wedge \overline{ \partial}\omega^{n-2})$.
    
 \vspace{1mm}
As a result, we prove Theorems \ref{Gauduchon-yuan} and \ref{thm1-general-equation}.  We 
now state the existence of weak solutions to degenerate equations as a complement.
\begin{theorem}
\label{thm1-general-equation-degenerate}
Let $(M,\omega)$ be a compact Hermitian manifold with smooth boundary. Suppose the boundary further obeys
 \eqref{bdry-assumption1-Gauduchon}.
In addition to \eqref{elliptic}, \eqref{concave} and \eqref{addistruc}, we assume
$f\in C^{\infty}(\Gamma)\cap C(\overline{\Gamma})$, $\varphi\in C^{2,1}(\bar M)$, $\varrho\in C^{\infty}(\bar M)$, and
the right-hand side satisfies  $\psi\in C^{1,1}(\bar M)$ and $\inf_M\psi=\sup_{\partial \Gamma}f$.
Suppose   there is a strictly $C^{2,1}$-admissible subsolution to Dirichlet problem 
  \eqref{mainequ-gauduchon-general**}. 
  Then it admits  a weak  $C^{1,\alpha}$  $(\forall 0<\alpha<1)$ solution $u$ with
 $\lambda(*\Phi[u])\in\overline{\Gamma}$ and $\Delta u\in L^\infty(\bar M)$.

Furthermore, if $(M,\omega)=(X\times S, \pi_1^*\omega_X+\pi_2^*\omega_S)$ and $\omega_X$ is balanced, then we can construction such strictly subsolutions, provided \eqref{cone-condition1-general} holds 
for an admissible function $\underline{v}\in C^{2,1}(\bar M)$ with $\underline{v}|_{\partial M}=\varphi$.

\end{theorem}
 

\subsection{Construction of subsolutions on   products with  balanced factors}
\label{construction2-subsection}
On the standard product $(M,\omega)=(X\times S, \pi_1^*\omega_X+\pi_2^*\omega_S)$ with  $(X,\omega_X)$ being balanced ($d\omega_{X}^{n-2}=0$,  note $\mathrm{dim}_{\mathbb{C}}=n-1$),
 we can construct (strictly) subsolution with 
$\frac{\partial \underline{u}}{\partial \nu}|_{\partial M}<0$, provided that
for given boundary value $\varphi$, there exists an extension, say $\underline{v}$, being an admissible function 
with $\lambda(*\Phi[\underline{v}])\in \Gamma$ such that \eqref{cone-condition1-general} holds. 


\vspace{1mm}
Let $\underline{v} \in C^{2}(\bar M)$ be an admissible function  with $\underline{v}|_{\partial M}=\varphi$ and \eqref{cone-condition1-general}
 as mentioned above.
Let    $h$  be the solution to  \eqref{possion-def-local}, then
 $$*(\sqrt{-1}\partial\overline{\partial}h\wedge\omega^{n-2})= *(\pi_2^*\omega_S\wedge \pi_1^*\omega_X)=(n-2)!\pi_1^*\omega_X.$$
Since 
$\partial h\wedge \overline{\partial}\omega^{n-2}
=\partial h\wedge  \overline{\partial}(\pi_1^*\omega_X^{n-2}),$
the obstruction to construct subsolutions 
automatically vanishes if  $\overline{\partial}\omega_X^{n-2}=0$, i.e. $\omega_X$ is balanced. 
More precisely, 
 when $\omega_X$ is balanced, the subsolution is given by
$$\underline{u}=\underline{v}+Ah, \mbox{  }A\gg1.$$ 
Straightforward computation shows 
$\Phi[{\underline{u}}]=\Phi[\underline{v}] +\frac{A}{(n-2)!}\pi_2^*\omega_S\wedge\omega^{n-2}=\Phi[\underline{v}] +\frac{A}{(n-2)!}\pi_1^*\omega_X\wedge\pi_2^*\omega_S$ and $*\Phi[{\underline{u}}]=*\Phi[{\underline{v}}]+ A\pi_1^*\omega_X.$

\subsection{The Dirichlet problem with less regularity assumptions on  products}
 Similarly, 
 when the boundary has less regularity $\partial S\in C^{2,\beta}$ ($0<\beta<1$),
we have

\begin{theorem}
\label{main-thm4}
Let $(M,\omega)=(X\times S, \pi_1^*\omega_X+\pi_2^*\omega_S)$ with balanced factor $(X,\omega_X)$.
Let $\partial S\in C^{2,\beta}$
and  $\psi\in C^{2}(\bar M)$,
then equation \eqref{mainequ-gauduchon1} with homogeneous boundary data admits  a unique $C^{2,\alpha}$-smooth $(n-1)$-PSH 
solution for some 
$0<\alpha\leq\beta$.

Moreover, the balanced assumption on $\omega_X$ can be removed if $\psi\leq \log P_{n-1}(\lambda(\tilde{\chi}))$.
\end{theorem}


\begin{theorem}
\label{main-thm3}
Suppose the given data $\partial S\in C^{2,\beta}$ $(0<\beta<1)$.
Then $$\left(\tilde{\chi}+\frac{1}{n-1}((\Delta u) \omega-\sqrt{-1}\partial \overline{\partial}u)+Z\right)^{n}=0, \mbox{  } u|_{\partial M}=0,$$
admits  a weak  $C^{1,\alpha}$  $(\forall 0<\alpha<1)$ solution $u$ with
 $\tilde{\chi}+\frac{1}{n-1}((\Delta u) \omega-\sqrt{-1}\partial \overline{\partial}u)+Z\geq 0$, 
and $\Delta u\in L^\infty(\bar M)$.
\end{theorem}


The results can be further extended to general equations as follows:

\begin{theorem}
\label{thm3-general-equation}
Let $(M,\omega)=(X\times S, \pi_1^*\omega_X+\pi_2^*\omega_S)$ be a product with 
   the factor $(X,\omega_X)$ being a closed balanced manifold. Let $\partial S\in C^3$, 
   $\psi\in C^2(\bar M)$,
$\varrho\in C^{3}(\bar M)$,  $\varphi\in C^{3}(\partial M)$. Suppose in addition that \eqref{cone-condition1-general},  \eqref{elliptic}, \eqref{concave}, \eqref{nondegenerate} and \eqref{addistruc} hold.
Then Dirichlet problem \eqref{mainequ-gauduchon-general**} 
has a unique $C^{2,\alpha}$ admissible solution for some $0<\alpha<1$. 
If the data $\partial S, \psi, \varphi \in C^{\infty}$ then the admissible solution is smooth.
\end{theorem}


\begin{theorem}
\label{thm2-general-equation}
Suppose $(M,\omega)=(X\times S, \pi_1^*\omega_X+\pi_2^*\omega_S)$ is a product whose factor $(X,\omega_X)$ is balanced. Let $\psi\in C^2(\bar M)$,
$ \partial S \in C^{2,\beta}$, $\varrho \in C^{2,\beta}(\bar M) $ 
 for some $0<\beta<1$.
In additon to  \eqref{elliptic}, \eqref{concave}, \eqref{nondegenerate} and \eqref{addistruc} hold, there is an $C^{2,\beta}$ admissible function with $\underline{v}|_{\partial M}=0$ to satisfy
  \eqref{cone-condition1-general}.
Then  Dirichlet problem \eqref{mainequ-gauduchon-general**} with homogeneous boundary data has a unique $C^{2,\alpha}$-admissible solution for some $0<\alpha\leq\beta$.
\end{theorem}


 


One can check if $\omega_X$ is both astheno-K\"ahler and Gauduchon,  
then so is  $\omega=\pi_1^*\omega_X+\pi_2^*\omega_S$. 
(Such closed non-K\"ahler manifolds, which further endow with  balanced metrics $\omega_{X,0}$,
were constructed in
\cite{Latorre2017Ugarte}).
 It is worthy to note that our results in this section also hold for the Form-type Calabi-Yau equation on such manifolds, 
 since the Form-type Calabi-Yau equation on such manifolds can be reduced to $\Psi_{u}^{n}=e^\psi \omega^n$ with
 \begin{equation}
 \label{FWW-equ}
 \begin{aligned}
 \Psi_u^{n-1}=\omega_0^{n-1}+\sqrt{-1}\partial \overline{\partial}u \wedge \omega^{n-2} + 2\mathfrak{Re}(\sqrt{-1}\partial u \wedge\overline{\partial}(\omega^{n-2})),
 \end{aligned}
 \end{equation}
 where $\omega_0$ is balanced. 
As above we can construct subsolutions if $\omega_X$ is further balanced. However,
a balanced metric  on a closed complex manifold cannot be astheno-K\"ahler unless it is K\"ahler (see \cite{Matsuo2001Takahashi}).



  \subsection{Proof of main estimates}
\label{proof-of-mainresults-mongeampere}

  Equation \eqref{mainequ-gauduchon-general*} is   as in the following
 \begin{equation}
\label{mainequ-gauduchon-general}
\begin{aligned}
F(U[u])=f(\lambda(U[u]))=\psi \mbox{ in } M,
\end{aligned}
\end{equation}
in which $U[u]:=*\Phi[u]=\chi+ (\Delta u)\omega-\sqrt{-1}\partial\overline{\partial}u +\varrho Z(\partial u,\overline{\partial}u)$.
   In what follows, 
   \begin{equation}
   \label{g-gau-form1}
   \begin{aligned}
   \mathfrak{g}[u]=\sqrt{-1}\partial\overline{\partial}u+  \check{\chi}
   +\frac{\varrho}{n-1}W(\partial u,\overline{\partial}u), 
   \mbox{   } \check{\chi}= \frac{1}{n-1}(\mathrm{tr}_\omega\chi)\omega-\chi,
   \end{aligned}
   \end{equation} 
 \begin{equation}
 \label{eigenvalues-denote1}
\begin{aligned}
 \lambda(\mathfrak{g}[u])=(\lambda_1,\cdots,\lambda_n) \mbox{ and } \lambda(U[u])=(\mu_1,\cdots,\mu_n),
 \end{aligned}
\end{equation}
then $\mu_i=\lambda_1+ \cdots+ \widehat{\lambda}_i+\cdots +\lambda_n$.
Let $P': \Gamma \longrightarrow P'(\Gamma)=:\widetilde{\Gamma}$ be a map given by
\begin{equation}
\begin{aligned}
(\mu_1,\cdots,\mu_n)  \longrightarrow (\lambda_1,\cdots,\lambda_n)=(\mu_1,\cdots,\mu_n) Q^{-1},
\end{aligned}
\end{equation}
where $Q=(q_{ij})$ and $q_{ij}=1-\delta_{ij}$ ($Q$ is symmetric). Here $Q^{-1}$ is well defined, since
 $\mathrm{det}Q=(-1)^{n-1} (n-1)\neq 0$.
 Thus, $\widetilde{\Gamma}$ is also an open symmetric convex cone of $\mathbb{R}^n$.

\vspace{1mm} 
Let's define $\tilde{f}: \widetilde{\Gamma}\rightarrow \mathbb{R}$ by $f(\mu)=\tilde{f}(\lambda)$.
Thus equation \eqref{mainequ-gauduchon-general} is rewritten as
\begin{equation}
\begin{aligned}
\tilde{f}(\lambda(\mathfrak{g}[u]))=f(\lambda(U[u])).
\end{aligned}
\end{equation}
In particular, for equation \eqref{mainequ-gauduchon},
 $f(\mu)=\sum_{i=1}^n \log \mu_i, \mbox{  } \tilde{f}(\lambda)= \log P_{n-1}(\lambda).   $
  
 \vspace{1mm}  One can verify that $\widetilde{f}$ also satisfies \eqref{elliptic}, \eqref{concave}  and \eqref{addistruc} in $\widetilde{\Gamma}$. 
 The linearized operator $\tilde{\mathcal{L}}$ of equation \eqref{mainequ-gauduchon-general} is given by
$$\tilde{\mathcal{L}}v=G^{i\bar j}v_{i\bar j}+\varrho F^{i\bar j}Z_{i\bar j,\zeta_{k}}v_k +\varrho F^{i\bar j}Z_{i\bar j,\zeta_{\bar k}}v_{\bar k},$$
where
$F^{i\bar j}=\frac{\partial F}{\partial U_{i\bar j}}$,
 $G^{i\bar j}=\sum_{k,l=1}^n(F^{k\bar l}g_{k\bar l})g^{i\bar j}-F^{i\bar j}$, and
  $U_{i\bar j}=\chi_{i\bar j}+\Delta u g_{i\bar j}-u_{i\bar j}+\varrho Z_{i\bar j}$.

\vspace{1mm}
 By constructing supersolution we can use maximum principle to derive \eqref{herer}, i.e.
 $$\sup_{ M}|u|+\sup_{\partial M}|\nabla u|\leq C.$$
  Theorem \ref{GTW-second} essentially shows
\textit{admissible} solutions to 
 \eqref{mainequ-gauduchon1} obey \eqref{quantitative-2nd-boundary-estimate}, i.e.
  $$\sup_M \Delta u\leq C(1+\sup_M |\nabla u|^2+\sup_{\partial M}|\Delta u|).$$
    Indeed, 
 \eqref{quantitative-2nd-boundary-estimate}  for   \eqref{mainequ-gauduchon-general} can be derived by following the outline 
 in \cite{GTW15}. 
 We sketch briefly the proof but omit more details to cut the paper's length.
Let's first recall two crucial ingredients in their proof of Theorem \ref{GTW-second} for equation \eqref{mainequ-gauduchon}: One is about the special structure of $Z$, and the other one is about the coefficient matrix
$(G^{i\bar j})$  of $\tilde{\mathcal{L}}$, whose eigenvalues are
$(\frac{\partial\tilde{f}}{\partial\lambda_1},\cdots,\frac{\partial\tilde{f}}{\partial\lambda_n})$.
More precisely,
 \begin{itemize}
 \item $Z$ satisfies the assumption in Page 187 of \cite{GTW15}.
 \item If $\lambda_1\geq \cdots \geq \lambda_n$, then 
$\frac{\partial \tilde{f}}{\partial \lambda_i}=\sum_{j\neq i} \frac{1}{\mu_j}\geq  \frac{1}{\mu_1}\geq  \frac{1}{n(n-1)} \sum_{k=1}^n \frac{\partial \tilde{f}}{\partial \lambda_k}$ for each $i\geq 2$.
 \end{itemize}

  For equation \eqref{mainequ-gauduchon-general}, as in \cite{GTW15}, $\varrho Z$ clearly 
  satisfies the assumption in Page 187 of \cite{GTW15}.
On the other hand, if $\lambda_1\geq \cdots \geq \lambda_n$, then
$\frac{\partial \tilde{f}}{\partial \lambda_i}= \sum_{j\neq i}\frac{\partial f}{\partial \mu_j}\geq
\frac{\partial f}{\partial \mu_1}\geq \frac{1}{n(n-1)}\sum_{k=1}^n \frac{\partial \tilde{f}}{\partial \lambda_k}$ for each $i\geq 2$.
The same argument as  
in \cite{GTW15} 
works and then derives 
\eqref{quantitative-2nd-boundary-estimate}.

\vspace{1mm}
The main goal is to derive gradient estimate. 
To derive the gradient estimate, as above,
it is also required to set up the  quantitative boundary estimate \eqref{bdy-sec-estimate-quar1}, i.e.
$$\sup_{\partial M} \Delta u \leq C (1+\sup_M |\nabla u|^2).$$
Such an 
estimate 
 immediately follows from \eqref{herer},
Propositions  \ref{boundary-mixed1} and \ref{proposition1-normal}.

 \begin{proposition}
\label{boundary-mixed1}
 Let $(M,\omega)$ be a compact Hermitian manifold with $C^3$ boundary.
Suppose \eqref{elliptic}, \eqref{concave} and \eqref{nondegenerate} hold.
Let $u\in C^{3}(M)\cap C^2(\bar M)$ be an admissible solution to Dirichlet problem \eqref{mainequ-gauduchon-general**}. 
Assume there is a $C^2$ admissible subsolution $\underline{u}$.
Then there is a uniformly positive constant  $C$ depending on $|\varphi|_{C^{3}(\bar M)}$,
$\sup_{M}|\nabla\psi|$,
$|\underline{u}|_{C^{2}(\bar M)}$, $\partial M$
up to second derivatives
and other known data 
\begin{equation}
\label{mixed-2}
\begin{aligned}
\sup_{\partial M}|\mathfrak{g}_{\alpha \bar n}| \leq C(1+\sup_{\partial M}|\nabla u|)(1+\sup_{M}|\nabla u|), 1\leq\alpha\leq n-1.
\end{aligned}
\end{equation}
Moreover, if $M=X\times S$ and 
 $\varphi\in C^2(\partial S)$ 
then the $C$ depends only on $\partial M$
up to second order derivatives, $|\varphi|_{C^{2}(\bar S)}$,
$\sup_{M}|\nabla\psi|$ and $|\underline{u}|_{C^{2}(\bar M)}$
and other known data.
\end{proposition}
The proof of Proposition \ref{boundary-mixed1} is  almost the same as that of Proposition  \ref{mix-general}, which goes through word by word. So we omit the proof.



  \begin{proof}
 [Proof of Proposition \ref{proposition1-normal}]
 Fix $x_0\in \partial M$. 
Around $x_0$ we set local holomorphic coordinates $z=(z_1,\cdots,z_n)$ defined as in \eqref{goodcoordinate1},
and we  further assume that $ ({\mathfrak{g}}_{\alpha\bar\beta})$ is diagonal at $x_0$.
 In the proof the discussion is done at $x_0$, and 
 the Greek letters, such as $\alpha, \beta$, range from $1$ to $n-1$.
Let's denote
\begin{equation}
{\tilde{A}}(R)=\left(
\begin{matrix}
R-\mathfrak{{g}}_{1\bar 1}&&  &-\mathfrak{g}_{1 \bar n}\\
&\ddots&&\vdots \\
& &  R-\mathfrak{{g}}_{{(n-1)} \overline{(n-1)}}&- \mathfrak{g}_{(n-1) \bar n}\\
-\mathfrak{g}_{n \bar 1}&\cdots& -\mathfrak{g}_{n \overline{(n-1)}}& \sum_{\alpha=1}^{n-1}\mathfrak{g}_{\alpha \bar\alpha}  \nonumber
\end{matrix}
\right),
\end{equation}
\begin{equation}
\tilde{\underline{A}}(a,R)=\left(
\begin{matrix}
R-\mathfrak{{g}}_{1\bar 1}&&  &-\mathfrak{g}_{1\bar n}\\
&\ddots&&\vdots \\
& & R-\mathfrak{{g}}_{{(n-1)}  \overline{(n-1)}}&- \mathfrak{g}_{(n-1) \bar n}\\
-\mathfrak{g}_{n \bar1}&\cdots& -\mathfrak{g}_{n \overline{(n-1)}}& \sum_{\alpha=1}^{n-1}\underline{\mathfrak{g}}_{\alpha\bar\alpha}-a  \nonumber
\end{matrix}
\right),
\end{equation}
\begin{equation}
B(R)=\left(
\begin{matrix}
R-\mathfrak{\underline{g}}_{1\bar1}&-\mathfrak{\underline{g}}_{1 \bar 2}&\cdots &-\mathfrak{\underline{g}}_{1\overline{(n-1)}} &-\mathfrak{\underline{g}}_{1\bar n}\\
-\mathfrak{\underline{g}}_{2\bar1} &R-\mathfrak{\underline{g}}_{2\bar 2}&\cdots& -\mathfrak{\underline{g}}_{2\overline{(n-1)}}&-\mathfrak{\underline{g}}_{2\bar n}\\
\vdots&\vdots&\ddots&\vdots&\vdots \\
-\mathfrak{\underline{g}}_{(n-1)\bar1}&-\mathfrak{\underline{g}}_{(n-1)\bar 2}& \cdots&  R-\mathfrak{\underline{g}}_{{(n-1)} \overline{(n-1)}}& -\mathfrak{\underline{g}}_{(n-1)\bar n}\\
-\mathfrak{\underline{g}}_{n\bar1}&-\mathfrak{\underline{g}}_{n\bar 2}&\cdots&-\mathfrak{\underline{g}}_{n \overline{(n-1)}}& \sum_{\alpha=1}^{n-1}\underline{\mathfrak{g}}_{\alpha\bar \alpha}  \nonumber
\end{matrix}\right).\end{equation}
In particular, 
$\tilde{A}(\mathrm{tr}_\omega(\mathfrak{g}))=\left(U_{i\bar j}([u])\right)$, 
$B(\mathrm{tr}_\omega(\mathfrak{\underline{g}}))=\left(U_{i\bar j}([\underline{u}])\right)$.
We also denote eigenvalues of $(n-1)\times (n-1)$ matrix $\left(\mathfrak{\underline{g}}_{\alpha \bar\beta}\right)$ by
 $\underline{\lambda}'=(\underline{\lambda}'_1,\cdots, \underline{\lambda}'_{n-1})$.
Similar as in \eqref{opppp}, 
$$ f(R_1-\underline{\lambda}'_1,\cdots, R_1-\underline{\lambda}'_{n-1},\sum_{\alpha=1}^{n-1}\underline{\mathfrak{g}}_{\alpha \bar\alpha})\geq F(\lambda(B(\mathrm{tr}_\omega(\mathfrak{\underline{g}}))))\geq \psi, $$
and $(R_1-\underline{\lambda}'_1,\cdots, R_1-\underline{\lambda}'_{n-1},\sum_{\alpha=1}^{n-1}\underline{\mathfrak{g}}_{\alpha\bar\alpha})\in\Gamma$  for 
$R_1>0$ 
depending only on $\mathfrak{\underline{g}}$.
Therefore, there are positive constants $\varepsilon_{0}$, $R_{0}$
 depending  on  $\mathfrak{\underline{g}}$  and $f$, such that $(R_0-\underline{\lambda}'_1-\varepsilon_{0},\cdots, R_0-\underline{\lambda}'_{n-1}-\varepsilon_{0},\sum_{\alpha=1}^{n-1}\underline{\mathfrak{g}}_{\alpha\bar\alpha}-\varepsilon_{0})\in \Gamma$, and
\begin{equation}
\label{opppp-Gauduchon}
\begin{aligned}
  f(R_0-\underline{\lambda}'_1-\varepsilon_{0},\cdots, R_0-\underline{\lambda}'_{n-1}-\varepsilon_{0},\sum_{\alpha=1}^{n-1}\underline{\mathfrak{g}}_{\alpha\bar\alpha}-\varepsilon_{0})\geq   \psi=F(\tilde{A}(\mathrm{tr}_\omega(\mathfrak{g}))).
\end{aligned}
\end{equation}

\vspace{1mm}
Let $\epsilon'$ be a uniformly positive constant with
 $\epsilon'\sup_{\partial M}(w-\underline{u})_{\nu}\leq\frac{\varepsilon_0}{4}$ (so $\epsilon'\sup_{\partial M}(u-\underline{u})_{\nu}\leq\frac{\varepsilon_0}{4}$), where $w$ satisfies \eqref{supersolution}. 
 Since  
 $\mathrm{tr}_{\omega'}(-{L}_{\partial M})\in \overline{\Gamma}^\infty_{\mathbb{R}^1}$ there is a positive constant 
 $c'=C(\epsilon',\mathrm{tr}_{\omega'}(-{L}_{\partial M}))$ such that 
 $(c',\cdots,c',\epsilon'-\mathrm{tr}_{\omega'}({L}_{\partial M}))\in \Gamma.$
 Next, we check
 \begin{equation}
\begin{aligned}
 \tilde{{A}}(R)=\tilde{\underline{A}}(\epsilon'(u-\underline{u})_{\nu},R-c'(u-\underline{u})_{\nu})
 +(u-\underline{u})_{\nu}\mathrm{diag}(c',\cdots, c', \epsilon'-\mathrm{tr}_{\omega'}({L}_{\partial M})). \nonumber
\end{aligned}
\end{equation}
By \eqref{bdr-ind}, \eqref{Z-tensor1},  $T_{nn}^n=0,$ and $g_{i\bar j}=\delta_{ij}$, one has
\begin{equation}
\begin{aligned}
2(n-1)Z_{n\bar n}=\,&  \sum_{p,k=1}^n (\bar T^k_{pk}  u_p
+ T^{k}_{pk}  u_{\bar p})
-  \sum_{k=1}^n(\bar T^n_{kn}u_k
 + T^{n}_{kn}u_{\bar k}
+\bar T^{k}_{nk}u_n + T^{k}_{nk}u_{\bar n} )  \\
=\,&  \sum_{\alpha, \beta=1}^{n-1}  (\bar T^\beta_{\alpha \beta}  u_\alpha
+ T^{\beta}_{\alpha \beta}  u_{\bar \alpha})
 =  \sum_{\alpha,\beta=1}^{n-1}   (\bar T^\beta_{\alpha \beta}  \underline{u}_\alpha
+ T^{\beta}_{\alpha \beta}  \underline{u}_{\bar \alpha}) \nonumber
  = 2(n-1)\underline{Z}_{n\bar n}.
\end{aligned}
\end{equation}
Thus we achieve the goal, 
since $W_{i\bar j}=(\mathrm{tr}_\omega Z)g_{i\bar j}-(n-1)Z_{i\bar j}$, \eqref{bdr-ind-2} and
\eqref{bdr-ind}
\begin{equation}
\begin{aligned} \sum_{\alpha=1}^{n-1}\mathfrak{g}_{\alpha\bar\alpha} = \,& \sum_{\alpha=1}^{n-1}(u_{\alpha\bar\alpha}+ \check{\chi}_{\alpha\bar\alpha})+\frac{\varrho}{n-1}\sum_{\alpha=1}^{n-1}W_{\alpha\bar\alpha}
=  \sum_{\alpha=1}^{n-1}(u_{\alpha\bar\alpha}+  \check{\chi}_{\alpha\bar\alpha})+ \varrho Z_{n\bar n}  \\
= \,& \sum_{\alpha=1}^{n-1}(\underline{u}_{\alpha\bar\alpha}+ \check{\chi}_{\alpha\bar\alpha})
-(u-\underline{u})_{\nu}\mathrm{tr}_{\omega'}({L}_{\partial M}) +\varrho\underline{Z}_{n\bar n} \\
 =\,&
  \sum_{\alpha=1}^{n-1}\underline{\mathfrak{g}}_{\alpha\bar\alpha}-(u-\underline{u})_{\nu}\mathrm{tr}_{\omega'}({L}_{\partial M}). \nonumber
\end{aligned}
\end{equation}

\vspace{1mm} Let's pick  $\epsilon=\frac{\varepsilon_0}{2(n-1)}$ in  Lemma  \ref{yuan's-quantitative-lemma} and set
\begin{equation}
\begin{aligned}
 R_s=\,& \frac{2(n-1)(2n-3)}{\varepsilon_0}
\sum_{\alpha=1}^{n-1} | \mathfrak{g}_{\alpha \bar n}|^2
+ (n-1)\sum_{\alpha=1}^{n-1} | \mathfrak{{g}}_{\alpha \bar\alpha}| 
    + \sum_{\alpha=1}^{n-1}( \mathfrak{\underline{g}}_{\alpha \bar\alpha}
 + |\underline{\lambda}'_\alpha| )
 +R_0 +\varepsilon_0+
    c'(u-\underline{u})_{\nu}, \nonumber
\end{aligned}
\end{equation}
where $\varepsilon_0$ and $R_0$ are fixed constants so that \eqref{opppp-Gauduchon} holds, and  
$c'=C(\epsilon',\mathrm{tr}_{\omega'}(-{L}_{\partial M}))$ is chosen above.

\vspace{1mm}
 Let $\lambda(\tilde{\underline{A}}(\epsilon'(u-\underline{u})_{\nu},R_s-c'(u-\underline{u})_{\nu}))=(\lambda_1(\epsilon',R_s),\cdots,\lambda_n(\epsilon',R_s))$ be the eigenvalues of $\tilde{\underline{A}}(\epsilon'(u-\underline{u})_{\nu},R_s-c'(u-\underline{u})_{\nu})$.
 It follows from  Lemma  \ref{yuan's-quantitative-lemma}   
 that 
\begin{equation}
\label{lemma12-yuan-Gauduchon}
\begin{aligned}
\lambda_\alpha(\epsilon',R_s) \geq \,& R_s-c'(u-\underline{u})_{\nu}- \mathfrak{g}_{1\bar 1}-\frac{\varepsilon_0}{2(n-1)},
\mbox{  } \forall 1\leq \alpha<n, \\
\lambda_n(\epsilon',R_s) \geq \,& \sum_{\alpha=1}^{n-1}\mathfrak{\underline{g}}_{\alpha\bar \alpha}-\epsilon'(u-\underline{u})_{\nu}
-\frac{\varepsilon_0}{2},
\end{aligned}
\end{equation}
in particular, $\lambda(\tilde{\underline{A}}(\epsilon'(u-\underline{u})_{\nu},R_s-c'(u-\underline{u})_{\nu}))\in\Gamma$. 
Therefore,  $\lambda(\tilde{{A}}(R_s))\in \Gamma$ and 
 \begin{equation}
\label{puretangential2-gauduchon}
\begin{aligned}
F( \tilde{{A}}(R_s))\geq F(\tilde{\underline{A}}(\epsilon'(u-\underline{u})_{\nu},R_s-c'(u-\underline{u})_{\nu})).
\end{aligned}
\end{equation}
Here,  we use \eqref{addistruc} as in proof of Proposition \ref{proposition-quar-yuan1}. 
 In addition, if $\partial M$ is mean pseudoconcave, then 
 $\sum_{\alpha=1}^{n-1} {\mathfrak{g}}_{\alpha\bar\alpha}
\geq  \sum_{\alpha=1}^{n-1}\underline{\mathfrak{g}}_{\alpha\bar\alpha}$, and so
 $F( \tilde{{A}}(R))\geq F(\tilde{\underline{A}}(0,R).$
   So \eqref{addistruc} can  be removed in the case when $\partial M$ is mean pseudoconcave.
 
 \vspace{1mm}
 Putting \eqref{elliptic},  \eqref{opppp-Gauduchon}, \eqref{puretangential2-gauduchon} and \eqref{lemma12-yuan-Gauduchon} together,  
  we get
   $\mathrm{tr}_\omega(\mathfrak{g}) \leq R_s.$



 \end{proof}

\section{Further discussions}
\label{further-discussion}

\subsection{Quantitative boundary estimate revisited}
An extension of Proposition  \ref{proposition-quar-yuan1} is as follows.
\begin{proposition} 
\label{extension2-supersolution}
Let $\psi\in C^0(\bar M)$, $\varphi\in C^2(\partial M)$.
In addition to  \eqref{elliptic}, \eqref{concave}, \eqref{nondegenerate}, \eqref{addistruc} and \eqref{unbound}, we assume that
$\lambda_{\omega'}({L}_{\partial M})\in \overline{\Gamma}_{\infty}$ and there is an admissible function $\breve{w}\in C^2(\bar M)$ with satisfying
\begin{equation}
\label{supersolution*}
\begin{aligned}
f(\lambda(\mathfrak{g}[\breve{w}]))\leq \psi \mbox{ in } M, \mbox{  } \breve{w}=\varphi \mbox{ on } \partial M.
\end{aligned}
\end{equation}
Let $u$ be a $C^2$-smooth admissible solution to Dirichlet problem \eqref{mainequ} and \eqref{mainequ1},
then \eqref{yuan-prop1} holds.
\end{proposition}
A somewhat interesting fact is that it does not require the subsolution obeying \eqref{existenceofsubsolution}  in 
Proposition  \ref{extension2-supersolution}, and the subsolution 
is only used to derive 
$\sup_{M}|u|+\sup_{\partial M}|\nabla u|\leq C.$
Comparison principle yields  $(\breve{w}-u)_\nu|_{\partial M}\geq 0$; 
moreover, we have an inequality analogous to \eqref{opppp} since $\breve{w}$ is a $\mathcal{C}$-subsolution of  equation \eqref{mainequ}.

\vspace{1mm}
It is noteworthy that, comparing with that of  Proposition  \ref{proposition-quar-yuan1}, the upper bound asserted in Proposition 
\ref{extension2-supersolution} may not depend on $(\delta_{\psi,f})^{-1}$. So it can not be applied to degenerate equations.

\vspace{1mm}
The discussion above also works for Dirichlet problem \eqref{mainequ-gauduchon-general*} and \eqref{mainequ1}.  We obtain

\begin{theorem}
 Let $(M,\omega)$ be a compact Hermitian manifold with general smooth boundary.
 Let $u\in C^3(M)\cap C^2(\bar M)$ be an admissible solution to Dirichlet problem 
 \eqref{mainequ-gauduchon-general*} and \eqref{mainequ1}.
Suppose, in addition to \eqref{elliptic}, \eqref{concave}, \eqref{nondegenerate},  $\psi\in C^\infty(\bar M)$, $\varphi\in C^\infty(\partial M)$, that \eqref{existenceofsubsolution}, \eqref{addistruc}, \eqref{unbound-strong} and \eqref{supersolution*} hold.
Then the quantitative boundary estimate \eqref{bdy-sec-estimate-quar1} holds, and the Dirichlet problem has a unique smooth admissible solution.
\end{theorem}

As a result, with assumptions that both \eqref{existenceofsubsolution}  and  \eqref{supersolution*} hold, we can solve Dirichlet problem of Monge-Amp\`ere equation for $(n-1)$-PSH functions on general compact Hermitian manifolds without mean pseudoconcave boundary restriction.

\subsection{Uniqueness of weak solution} \label{uniqueness-weak-solution}
We prove the weak solutions what we obtain in this paper
by using continuity method are weak $C^0$-solutions to Dirichlet problem correspondingly
 in the sense of Definition \ref{def-c0-weak}. 
Following Chen \cite{Chen}, we define 
\begin{definition}
\label{def-c0-weak}
A continuous function $u\in C(\bar M)$ is   a weak $C^0$-solution to degenerate equation \eqref{mainequ-ios}  with prescribed boundary data $\varphi$ if,
for any $\epsilon>0$ there is a $C^2$-\textit{admissible} function $\widetilde{u}$ such that $|u-\widetilde{u}|<\epsilon$, where $\widetilde{u}$ solves
\begin{equation}
\begin{aligned}
F(\mathfrak{g}[\widetilde{u}])=\psi+\rho_{\epsilon} \mbox{ in } M, \mbox{  } \widetilde{u}=\varphi \mbox{ on } \partial M. \nonumber
\end{aligned}
\end{equation}
Here $\rho_{\epsilon}$ is a function satisfying $0<\rho_{\epsilon}<C(\epsilon)$, and $C(\epsilon)\rightarrow 0$ as $\epsilon\rightarrow 0$.
\end{definition}

\begin{theorem}
\label{weakc0comparison}
Suppose $u^1$, $u^2$ are two $C^0$-weak solutions to  degenerate equation  \eqref{mainequ} with 
boundary data $\varphi^1$, $\varphi^2$. 
Then $\sup_{M}|u^1-u^2|\leq \sup_{\partial M}|\varphi^1-\varphi^2|.$
\end{theorem}

The proof is almost parallel to that of Theorem 4 in \cite{Chen}. 
We thus omit it here. 
\begin{corollary}
\label{unique-weak-solution}
The weak $C^0$-solution to Dirichlet problem \eqref{mainequ} and \eqref{mainequ1} for degenerate equation 
 is unique provided the boundary data is fixed.
\end{corollary}
In addition, Theorem \ref{weakc0comparison} and Corollary \ref{unique-weak-solution} are both valid for degenerate 
equation \eqref{mainequ-gauduchon-general*}, correspondingly.

\begin{appendix}

\section{Proof of Lemma \ref{yuan's-quantitative-lemma}}
  \label{appendix}

In this appendix we present the proof of Lemma \ref{yuan's-quantitative-lemma} 
 for  convenience and completeness. 

\vspace{1mm} 
We start with   $n=2$. In this case,
we can verify that  if $\mathrm{{\bf a}} \geq \frac{|a_1|^2}{ \epsilon}+ d_1$ then
$$0\leq d_1- \lambda_1=\lambda_2-\mathrm{{\bf a}} <\epsilon.$$
It is much more complicated for $n\geq 3$. To achieve our goal the author  proposed  in \cite{yuan2017}
the following lemma which states that, for the Hermitian matrix $\mathrm{A}$ (stated as in Lemma  \ref{yuan's-quantitative-lemma}), if $\mathrm{{\bf a}}$ satisfies a quadratic
 growth condition \eqref{guanjian2}, 
 then the eigenvalues  concentrate near certain  diagonal elements and the number of eigenvalues near the corresponding diagonal elements is stable, which
 enables us  to count  the eigenvalues near the diagonal elements
via a deformation argument.
\begin{lemma}
[\cite{yuan2017}]
\label{refinement}
Let $\mathrm{A}$ be a Hermitian $n$ by $n$ matrix defined as in Lemma  \ref{yuan's-quantitative-lemma}
with $d_1,\cdots, d_{n-1}, a_1,\cdots, a_{n-1}$ fixed, and with $\mathrm{{\bf a}}$ variable.
Denote
$\lambda=(\lambda_1,\cdots, \lambda_n)$ by the the eigenvalues of $\mathrm{A}$ with the order
$\lambda_1\leq \lambda_2 \leq\cdots \leq \lambda_n$.
Fix a positive constant $\epsilon$.
Suppose that the parameter $\mathrm{{\bf a}}$ in the matrix $\mathrm{A}$ satisfies  the 
 quadratic growth condition
\begin{equation}
\label{guanjian2}
\begin{aligned}
\mathrm{{\bf a}} \geq \frac{1}{\epsilon}\sum_{i=1}^{n-1} |a_i|^2+\sum_{i=1}^{n-1}  [d_i+ (n-2) |d_i|]+ (n-2)\epsilon.
\end{aligned}
\end{equation}
Then for   any $\lambda_{\alpha}$ $(1\leq \alpha\leq n-1)$ there exists  an  $d_{i_{\alpha}}$
with 
$1\leq i_{\alpha}\leq n-1$ such that
\begin{equation}
\label{meishi}
\begin{aligned}
 |\lambda_{\alpha}-d_{i_{\alpha}}|<\epsilon,
\end{aligned}
\end{equation}
\begin{equation}
\label{mei-23-shi}
\begin{aligned}
0\leq \lambda_{n}-\mathrm{{\bf a}} <(n-1)\epsilon + |\sum_{\alpha=1}^{n-1}(d_{\alpha}-d_{i_{\alpha}})|.
\end{aligned}
\end{equation}
\end{lemma}

\begin{proof}
Without loss of generality, we assume $\sum_{i=1}^{n-1} |a_i|^2>0$ and  $n\geq 3$
(otherwise we are done).
It is  well known that, for a Hermitian matrix,
 every diagonal element is less than or equals to the  largest eigenvalue.
 In particular,
 \begin{equation}
 \label{largest-eigen1}
 \lambda_n \geq \mathrm{{\bf a}}.
 \end{equation}

We only need to prove   \eqref {meishi}, since  \eqref{mei-23-shi} is a consequence of  \eqref{meishi}, \eqref{largest-eigen1}  and
\begin{equation}
\label{trace}
 \sum_{i=1}^{n}\lambda_i=\mbox{tr}(\mathrm{A})=\sum_{\alpha=1}^{n-1} d_{\alpha}+\mathrm{{\bf a}}.
 \end{equation}

 Let's denote   $I=\{1,2,\cdots, n-1\}$. We divide the index set   $I$ into two subsets  by
$${\bf B}=\{\alpha\in I: |\lambda_{\alpha}-d_{i}|\geq \epsilon, \mbox{   }\forall i\in I\} $$
and $ {\bf G}=I\setminus {\bf B}=\{\alpha\in I: \mbox{There exists an $i\in I$ such that } |\lambda_{\alpha}-d_{i}| <\epsilon\}.$

\vspace{1mm} 
To complete the proof we only need to prove ${\bf G}=I$ or equivalently ${\bf B}=\emptyset$.
    It is easy to see that  for any $\alpha\in {\bf G}$, one has
   \begin{equation}
   \label{yuan-lemma-proof1}
   \begin{aligned}
   |\lambda_\alpha|< \sum_{i=1}^{n-1}|d_i| + \epsilon.
   \end{aligned}
   \end{equation}

   Fix $ \alpha\in {\bf B}$,  we are going to give the estimate for $\lambda_\alpha$.
The eigenvalue $\lambda_\alpha$ satisfies
\begin{equation}
\label{characteristicpolynomial}
\begin{aligned}
(\lambda_{\alpha} -\mathrm{{\bf a}})\prod_{i=1}^{n-1} (\lambda_{\alpha}-d_i)
= \sum_{i=1}^{n-1} (|a_{i}|^2 \prod_{j\neq i} (\lambda_{\alpha}-d_{j})).
\end{aligned}
\end{equation}
By the definition of ${\bf B}$, for $\alpha\in {\bf B}$, one then has $|\lambda_{\alpha}-d_i|\geq \epsilon$ for any $i\in I$.
Therefore
\begin{equation}
\begin{aligned}
|\lambda_{\alpha}-\mathrm{{\bf a}} |
\leq \sum_{i=1}^{n-1} \frac{|a_i|^2}{|\lambda_{\alpha}-d_{i}|}\leq
\frac{1}{\epsilon}\sum_{i=1}^{n-1} |a_i|^2, \mbox{ if } \alpha\in {\bf B}.
\end{aligned}
\end{equation}
Hence,  for $\alpha\in {\bf B}$, we obtain
\begin{equation}
\label{yuan-lemma-proof2}
\begin{aligned}
 \lambda_\alpha \geq \mathrm{{\bf a}}-\frac{1}{\epsilon}\sum_{i=1}^{n-1} |a_i|^2.
\end{aligned}
\end{equation}

For a set ${\bf S}$, we denote $|{\bf S}|$ the  cardinality of ${\bf S}$.
We shall use proof by contradiction to prove  ${\bf B}=\emptyset$.
Assume ${\bf B}\neq \emptyset$.
Then $|{\bf B}|\geq 1$, and so $|{\bf G}|=n-1-|{\bf B}|\leq n-2$. 

\vspace{1mm} 
We compute   the trace of the matrix $\mathrm{A}$  as follows:
\begin{equation}
\begin{aligned}
\mbox{tr}(\mathrm{A})=\,&
\lambda_n+
\sum_{\alpha\in {\bf B}}\lambda_{\alpha} + \sum_{\alpha\in  {\bf G}}\lambda_{\alpha}\\
> \,&
\lambda_n+
|{\bf B}| (\mathrm{{\bf a}}-\frac{1}{\epsilon}\sum_{i=1}^{n-1} |a_i|^2 )-|{\bf G}| (\sum_{i=1}^{n-1}|d_i|+\epsilon ) \\
\geq \,&
 2\mathrm{{\bf a}}-\frac{1}{\epsilon}\sum_{i=1}^{n-1} |a_i|^2 -(n-2) (\sum_{i=1}^{n-1}|d_i|+\epsilon )
\\
\geq \,& \sum_{i=1}^{n-1}d_i +\mathrm{{\bf a}}= \mbox{tr}(\mathrm{A}),
\end{aligned}
\end{equation}
where we use  \eqref{guanjian2},   \eqref{largest-eigen1}, \eqref{yuan-lemma-proof1} and \eqref{yuan-lemma-proof2}.
This is a contradiction. So  ${\bf B}=\emptyset$.
\end{proof}


\begin{proof}
[Proof of Lemma  \ref{yuan's-quantitative-lemma}]
The proof is based on Lemma  \ref{refinement} and a deformation argument.
 Fix $d_1,\cdots, d_{n-1}, a_1,\cdots, a_{n-1}$, and we let $\mathrm{{\bf a}}$ be variable.
 Denote $\lambda_1(\mathrm{{\bf a}}), \cdots, \lambda_n(\mathrm{{\bf a}})$ by
 the eigenvalues of $\mathrm{A}$. 
  Clearly,  the eigenvalues $\lambda_i(\mathrm{{\bf a}})$ can be viewed as  continuous functions 
  of $\mathrm{{\bf a}}$.
  For simplicity, we write $\lambda_i=\lambda_i(\mathrm{{\bf a}})$.  
\vspace{1mm}
Without loss of generality, we may assume $n\geq 3$, $\sum_{i=1}^{n-1} |a_i|^2>0$,
 $d_1\leq d_2\leq \cdots \leq d_{n-1}$ and 
 $\lambda_1\leq \lambda_2 \leq \cdots \lambda_{n-1}\leq \lambda_n.$

\vspace{1mm} 
Fix $\epsilon>0$.  Let $I'_\alpha=(d_\alpha-\frac{\epsilon}{2n-3}, d_\alpha+\frac{\epsilon}{2n-3})$ and
$P_0'=\frac{2n-3}{\epsilon}\sum_{i=1}^{n-1} |a_i|^2+ (n-1)\sum_{i=1}^{n-1} |d_i|+ \frac{(n-2)\epsilon}{2n-3}.$
 In what follows we assume  $\mathrm{{\bf a}}\geq P_0'$ (i.e. \eqref{guanjian1-yuan} holds).

\vspace{1mm}  
The connected components of $\bigcup_{\alpha=1}^{n-1} I_{\alpha}'$ are as in the following:
$$J_{1}=\bigcup_{\alpha=1}^{j_1} I_\alpha',
J_2=\bigcup_{\alpha=j_1+1}^{j_2} I_\alpha'  \cdots, J_i =\bigcup_{\alpha=j_{i-1}+1}^{j_i} I_\alpha' \cdots,
 J_{m} =\bigcup_{\alpha=j_{m-1}+1}^{n-1} I_\alpha',$$
 (here we denote $j_0=0$ and $j_m=n-1$).
 Moreover, $J_i\bigcap J_k=\emptyset, \mbox{ for }   1\leq i<k\leq m$.

\vspace{1mm} 
Let $\mathrm{{\bf \widetilde{Card}}}_k:[P_0',+\infty)\rightarrow \mathbb{N}$
be the function that counts the eigenvalues which lie in $J_k$.
   (When the eigenvalues are not distinct,  the function $\mathrm{{\bf \widetilde{Card}}}_k$ denotes  the summation of all the multiplicities of  distinct eigenvalues which
 lie in $J_k$).
  This function measures the number of the  eigenvalues which lie in $J_k$.
  
 \vspace{1mm} 
  The crucial ingredient is that  Lemma \ref{refinement}  yields the continuity of   $\mathrm{{\bf \widetilde{Card}}}_i(\mathrm{{\bf a}})$ if
   $\mathrm{{\bf a}}\geq P_0'$. More explicitly,
by  Lemma \ref{refinement}  
we conclude that
 if   $\mathrm{{\bf a}}$ satisfies 
  \eqref{guanjian1-yuan} then
   \begin{equation}
  \label{yuan-lemma-proof5}
  \begin{aligned}
    \lambda_\alpha \in \bigcup_{i=1}^{n-1} I_{i}'=\bigcup_{i=1}^m J_{i} \mbox{ for } 1\leq\alpha<n, \mbox{ and }
   \lambda_n \in \mathbb{R}\setminus (\bigcup_{k=1}^{n-1} \overline{I_k'})
   =\mathbb{R}\setminus (\bigcup_{i=1}^m \overline{J_i}).
  \end{aligned}
  \end{equation}
Hence,  $\mathrm{{\bf \widetilde{Card}}}_i(\mathrm{{\bf a}})$ is a continuous function
of the variable $\mathrm{{\bf a}}$. So it is a constant.
 Together with  the line of the proof Lemma 1.2  of Caffarelli-Nirenberg-Spruck \cite{CNS3}
in the setting of Hermitian matrices  we see
 that $ \mathrm{{\bf \widetilde{Card}}}_i(\mathrm{{\bf a}}) =j_i-j_{i-1}$ for sufficiently large $\mathrm{{\bf a}}$.
The constant of $ \mathrm{{\bf \widetilde{Card}}}_i$  therefore follows that
$$ \mathrm{{\bf \widetilde{Card}}}_i(\mathrm{{\bf a}})=j_i-j_{i-1}.$$
We thus know that the   $(j_i-j_{i-1})$ eigenvalues
$\lambda_{j_{i-1}+1}, \lambda_{j_{i-1}+2}, \cdots, \lambda_{j_i}$
lie in the connected component $J_{i}$.
Thus, for any $j_{i-1}+1\leq \gamma \leq j_i$,  we have $I_\gamma'\subset J_i$ and  $\lambda_\gamma$
   lies in the connected component $J_{i}$.
Therefore,
$$|\lambda_\gamma-d_\gamma| < \frac{(2(j_i-j_{i-1})-1) \epsilon}{2n-3}\leq \epsilon.$$
Here we also use the fact that $d_\gamma$ is midpoint of  $I_\gamma'$ and every $J_i\subset \mathbb{R}$ is an open subset.

\vspace{1mm} 
Roughly speaking, for each fixed index $1\leq i\leq n-1$, if the eigenvalue $\lambda_i(P_0')$ lies in $J_{\alpha}$ for some $\alpha$, 
then  Lemma \ref{refinement} implies that, for any ${\bf a}>P_0'$, the corresponding eigenvalue  $\lambda_i({\bf a})$ lies in the same  interval $J_{\alpha}$.
Adapting the outline of proof the Lemma 1.2  of 
\cite{CNS3} to our context,
 we get the asymptotic behavior as $\mathrm{\bf a}$ goes to infinity.
\end{proof}

\end{appendix}


\bigskip

\textit{Added in proof}:  
 As the present work neared completion,
 I learned of a preprint  \cite{TW123} by V. Tosatti and B. Weinkove  which
considers  complex 
Monge-Amp\`ere equation on closed Hermitian manifolds as a special case of \eqref{mainequ}.  
  I append some further discussion in this version.
Comparing with previous version, 
this new version 
 discusses the Monge-Amp\`ere equation for $(n-1)$-PSH functions associated with Gauduchon's conjecture 
 (the second part of this paper), 
 and as a result the present paper was posted later.



\bigskip

\noindent


\small
\bibliographystyle{plain}

\end{document}